\documentclass[11pt,reqno]{amsart}
\usepackage[utf8]{inputenc}
\usepackage[a4paper,  left=25mm, right=25mm, top=32mm, bottom=31mm, marginparwidth= 20mm]{geometry}

\usepackage{amsmath}
\usepackage{amssymb}
\usepackage{bbm}
\usepackage{blkarray}
\usepackage{changepage}
\usepackage{enumerate}
\usepackage{graphicx}
\usepackage{mathrsfs}
\usepackage{mathtools}
\usepackage{pifont}
\usepackage[dvipsnames]{xcolor}
\usepackage{imakeidx}
\usepackage{tikz-cd}
\usepackage{todonotes}
\usepackage[colorlinks=true,linkcolor=NavyBlue,urlcolor=RoyalBlue,citecolor=PineGreen,bookmarks=true,bookmarksopen=true,bookmarksopenlevel=2,unicode=true,linktocpage]{hyperref}

\newtheorem{theorem}{Theorem}[section]
\newtheorem{definition}[theorem]{Definition}
\newtheorem{lemma}[theorem]{Lemma}
\newtheorem{proposition}[theorem]{Proposition}

\newcommand\build[3]{\mathrel{\mathop{#1}\limits_{#2}^{#3}}}
\renewcommand{\d}{{\; {\rm d}}}
\newcommand{\nsd}{{\rm d}}
\renewcommand{\epsilon}{\varepsilon}
\renewcommand{\phi}{\varphi}

\newcommand{\ul}{\underline}
\def\leq{\leqslant}
\def\geq{\geqslant}

\newcommand{\sbullet}{\raisebox{0.5pt}{$\scriptstyle\bullet$}}
\newcommand{\vsbullet}{\raisebox{1pt}{\scalebox{0.7}{$\scriptstyle\bullet$}}\hspace{1pt}}

\def\b{{\sf b}}
\def\e{{\sf e}}
\def\g{{\sf g}}
\def\k{{\sf k}}
\def\x{{\sf b}^{\sf x}}

\newcommand{\lf}{\text{\sc{l}}}
\newcommand{\rf}{\text{\sc{r}}}
\newcommand{\cne}{\text{\sc{ne}}}
\newcommand{\cnw}{\text{\sc{nw}}}
\newcommand{\cse}{\text{\sc{se}}}
\newcommand{\csw}{\text{\sc{sw}}}
\newcommand{\cn}{\text{\sc{n}}}
\newcommand{\cs}{\text{\sc{s}}}
\newcommand{\ce}{\text{\sc{e}}}
\newcommand{\cw}{\text{\sc{w}}}

\def\B{{\mathcal B}}
\newcommand{\HW}{{\mathbb Z}^{N}_{\downarrow}}

\def\cont{{\sf c}}
\newcommand\leads{\!\uparrow\!}
\newcommand{\Part}{\mathscr P}
\newcommand{\Y}{\mathcal Y}

\newcommand{\C}{\mathbb{C}}
\newcommand{\R}{\mathbb{R}}
\newcommand{\Z}{\mathbb Z}

\def\End{{\rm End}}
\newcommand{\GL}{\mathrm{GL}}
\renewcommand{\S}{\mathsf S}
\newcommand{\T}{{\sf T}}
\newcommand{\Tr}{\mathrm{Tr}}
\newcommand{\U}{\mathrm U}
\renewcommand{\u}{\mathfrak u}

\newcommand{\sG}{{\sf G}}
\newcommand{\sV}{{\sf V}}
\newcommand{\sE}{{\sf E}}
\newcommand{\sCE}{\sf CE}
\newcommand{\sLE}{\sf LE}
\newcommand{\sF}{{\sf F}}

\newcommand{\bG}{{\mathbb G}}
\newcommand{\bV}{{\mathbb V}}
\newcommand{\bE}{{\mathbb E}}
\newcommand{\bF}{{\mathbb F}}

\title[Wilson loop expectations on compact surfaces]{A combinatorial formula for\\Wilson loop expectations on compact surfaces}

\author{Thierry L\'evy}
\address{Sorbonne Université, Université Paris Cité, CNRS, Laboratoire de Probabilités, Statistique et Modélisation, LPSM, F-75005 Paris, France}
\email{thierry.levy@sorbonne-universite.fr}

\date{\today}

\keywords{
Yang--Mills holonomy process, Driver--Sengupta formula, Wilson loop expectation, Schur--Weyl duality, Collins--\'Sniady integration formula, Okounkov--Vershik theory, Gelfand--Zetlin lines, spin networks, classical $6j$-symbols.
}

\subjclass[2020]{81T13, 05E10}

\makeindex

\begin{document}

\begin{abstract} 
We give an almost purely combinatorial expression for Wilson loop expectations of the Yang–Mills holonomy process with values in the unitary group on a compact oriented surface, possibly with boundary and arbitrary boundary conditions.

Our main result computes the non-normalized expectation of products of traces of holonomies along an arbitrary family of immersed curves with transverse self-intersections and no triple points. It is expressed as a sum over assignments of highest weights of the unitary group to the connected components of the complement of the curves. Each term is a product of a Gaussian exponential factor, dimensions of unitary representations, and local contributions at the intersection points given by the sine or cosine of an angle determined by the surrounding highest weights.

As an application, we obtain a new and short proof of the Makeenko–Migdal equations on arbitrary compact surfaces.
\end{abstract}

\maketitle

\centerline{\small\it This paper is dedicated to James R. Norris on the occasion of his 60th birthday.}

\bigskip

\setcounter{tocdepth}{2}
\makeatletter
\def\l@subsection{\@tocline{2}{0pt}{2.5pc}{5pc}{}}
\makeatother
\noindent\begin{adjustwidth}{26pt}{36pt}
{\small \tableofcontents}
\end{adjustwidth}


\section*{Introduction}

The present paper is devoted to the exposition and proof of a new expression of the Wilson loop expectations for the $2$-dimensional Yang--Mills holonomy process with values in the unitary group $\U(N)$ on an  oriented compact surface with or without boundary, and in the case with boundary, with arbitrary boundary conditions, constrained or free. 

The Yang--Mills holonomy process is a stochastic process, that is, a probability measure on a space of functions, that arose from the broad line of research seeking to give a rigorous mathematical basis to various pieces of quantum field theory in general, and in this particular case to the pure Euclidean $2$-dimensional Yang--Mills theory.

\subsection*{Pure Euclidean $2$-dimensional Yang--Mills theory}

Yang--Mills theory is a field theory in which one of the dynamical variables is a connection on a principal bundle over a manifold that plays the role of space-time. This connection is a mathematical model for entities called gauge bosons that mediate a certain kind of interaction between elementary particles, for instance photons for the electromagnetic interaction or gluons for the strong interaction. The nature of the interaction is reflected by the choice of the structure group of the principal bundle, which is a compact Lie group. 

The adjective {\em pure} indicates that the model contains no fermions, that is, no matter, for example, no electrons nor quarks. The reason why pure Yang--Mills theory is not trivial, at least when the structure group is not abelian, is that the gauge bosons interact with each other, in a way that is directly related to the lack of commutativity of the structure group. 

In general, Yang--Mills theory depends on a metric structure on the base manifold of the principal bundle. According to the prescriptions of special relativity, this metric structure should be Lorentzian, but the word {\em Euclidean} indicates that we prefer to work with a Riemannian analogue of the physical theory. Along the same lines, one also prefers to turn oscillatory path integrals into Gaussian-like integrals, which are much more amenable to rigorous mathematical treatment. 

Finally, {\em $2$-dimensional} means that the base manifold is taken to be a real $2$-dimensional manifold, most often the plane or a compact surface, closed or with boundary. It turns out that in the $2$-dimensional case, the only part of a Riemannian metric on the base manifold that is relevant to Yang--Mills theory is the Riemannian area. 

\subsection*{A random connection and its random holonomy} Given a compact surface $\Sigma$ with a measure of area, a compact Lie group $G$ with an invariant inner product on its Lie algebra and a principal $G$-bundle over $\Sigma$, the {\em Yang--Mills action} is a non-negative functional on the space of affine connections on the principal bundle, and the {\em Yang--Mills measure} is heuristically defined as the Boltzmann measure associated with this action, that is, the probability measure that picks a connection at random with a weight proportional to the exponential of minus its action (see \cite{LS4} and \cite{LevyMM} for more detailed introductions).

For several reasons, including the fact that a typical connection sampled under this measure will be very irregular, to the point of being a distributional differential form, and the fact that this probability measure should be invariant under the action of an infinite-dimensional symmetry group, the construction of the Yang--Mills measure on the space of connections on a principal bundle is a difficult task, and a different road was first taken to turn its heuristic description into a mathematical definition.
\medskip

This first mathematical approach to this problem can be traced to works of Migdal \cite{Migdal} and Witten \cite{Witten} and was developed on the mathematical side by Gross \cite{GrossME}, Driver \cite{DriverYM2}, Sengupta \cite{GKS, SenguptaAMS}, the author \cite{LevyAMS, LevyMHF, LevyMF,LevyMM}, and pursued, among others, by Norris and the author \cite{LevyNorris} and Sauzedde \cite{Sauzedde}.

The key to this approach is that a connection determines a holonomy, which is an integrated version of the connection that assigns, up to some choices, an element of the structure group to each loop on the base manifold, that is, in the $2$-dimensional case, to each loop on the surface. Rather than constructing the Yang--Mills measure as a probability measure on the space of connections, one constructs its push-forward by the holonomy map, which is a probability measure on the space of functions from the set of loops on the surface to the structure group. 

This approach results in the construction of a sound mathematical object called the Yang--Mills holonomy process, that can be described as  a family of group-valued random variables indexed by the set of sufficiently regular loops on a surface. This process, in the case where the structure group is a unitary group, is the main object of study of the present paper, and we will describe it in Sections \ref{sec:ymhp} and \ref{sec:DSformula}. 
\medskip

Over the last decade, a fresh impulse was given to the study of the Yang--Mills measure as a random distributional differential form, by works of Chevyrev \cite{ChevyrevT,Chevyrev2}, then of Chandra, Chevyrev, Hairer, Shen, Zhu, Zhu \cite{CCHS, ChevyrevShen, ShenZhuZhu}, and with a different approach of Dang and Nohra \cite{DangNohra}. Some of these authors also deal with the $3$-dimensional case \cite{ChevyrevShen3d}, for which the holonomy process was not yet constructed. On the general problem of constructing distributional differential forms, see \cite{Jaffard} and the references therein.

Another active direction of investigation of the $2$-dimensional Yang--Mills measure is its so-called {\em large $N$} behaviour, as the rank of the unitary group $\U(N)$ tends to infinity, and although we do not pursue this direction in the present work, we expect the results of this paper to be useful for the study of this limit, see in particular \cite{Antoine}. On this problem, let us mention contributions of Xu \cite{Xu}, Sengupta \cite{SenguptaLargeN}, the author \cite{LevyMF}, Cébron, Dahlqvist and Gabriel \cite{CDG}, Hall \cite{Hall}, Dahlqvist and Norris \cite{DahlqvistNorris}, and on closely related problems, of Maïda and the author \cite{LevyMaida}, Gilliers \cite{Gilliers}, Dahlqvist and Lemoine \cite{DahlqvistLemoine1,DahlqvistLemoine2}, Lemoine and Maïda \cite{LemoineMaida}.

Let us also mention that, also in the last decade, the mathematical study of Yang--Mills theory developed in the direction of statistical mechanics, on lattices, with different questions and methods, see \cite{Chatterjee, FLV, CPS, CRS, FV, ChevyrevGarban, Forsstrom}.

\subsection*{Wilson loop expectations} 
The Yang--Mills holonomy process on a surface $\Sigma$ associates to each loop $\ell$ on $\Sigma$ a random variable $H_{\ell}$ that takes its values in a group, which in this paper will be a unitary group $\U(N)$. The measure that underlies this process is a probability measure, which turns out to have an almost canonical scalar multiple, which is therefore a finite measure, with total mass a number called the {\em Yang--Mills partition function}, that we denote by $Z(\Sigma)$. In the case where $\Sigma$ has a boundary, this number can depend on the boundary conditions, but we leave this aside for the moment. 

The purpose of this paper is to give a new expression of the numbers
\begin{equation*}
Z(\Sigma)\, {\bf E}\big[\Tr(H_{\ell_{1}})\ldots \Tr(H_{\ell_{m}})\big],
\end{equation*}
where $\ell_{1},\ldots,\ell_{m}$ is an arbitrary nice family of loops on $\Sigma$. These numbers are called {\em non-normalized Wilson loop expectations}, and although they may look like rather special observables of the Yang--Mills holonomy process, this process enjoys properties of multiplicativity, gauge invariance and regularity that ensures that these numbers suffice to characterize its distribution completely. 

From the construction of the holonomy process mentioned in the previous paragraph, and its description by the Driver--Sengupta formula (see Section \ref{sec:DS}), one gets an expression of any Wilson loop expectation as an integral over a finite product of copies of the unitary group, of a product of traces and heat kernels on the unitary group evaluated at various words in the integration variables. 
\medskip

To put it in one sentence, we systematically apply the tools harmonic analysis on the unitary groups to this integral and turn it into an infinite sum, over what we call {\em highest weight configurations} (see Section \ref{sec:CHW}), which are assignments of a highest weight of $\U(N)$ to each connected component of the complement of the range of the family of loops under consideration. In the expression that we obtain (see Theorem \ref{thm:main}), the contribution of each highest weight configuration is the product of 
\begin{enumerate}[\indent \sbullet]
\item an exponential prefactor that reflects the presence of the heat kernels in the Driver--Sengupta formula,
\item a product of dimensions of the representations of the unitary group parametrized by the highest weights of the configuration, raised to powers which depend on the topological configuration of the family of loops on the surface,
\item a product of sines or cosines of angles that depend on the structure of the configuration of highest weights around each point of self-intersection of the family of loops.
\end{enumerate}
This last element, the product of sines and cosines of angles, is perhaps the most characteristic aspect of our formula. These angles arise in the course of the proof from the geometry of irreducible representations of symmetric groups, in Section \ref{sec:detscal}. Some of them are irrational numbers, but we prove that their product is rational (see Proposition \ref{prop:rational}).
\medskip

Although I was not aware of this as I undertook the present work, the fact that an expression of Wilson loop expectations such as that given by Theorem \ref{thm:main} should exist was, in its principle, explained by Witten in one of his celebrated papers on $2$-dimensional Yang--Mills theory \cite{Witten}. In this seemingly inexhaustible paper, this explanation is the subject of Section 2.5, titled `Comparison to IRF models', where IRF stands for {\em interaction round faces}. Neither the details of the computation nor the final formula appear in Witten's paper, and the present work can be seen, in retrospect, as filling in these gaps. 

In this section of his paper, Witten also indicates that some parts of the computation should be closely related to the determination of classical $6j$-symbols. This is a direction that we will not pursue in this paper, but that is certainly worth some further investigation. 

\subsection*{Structure of the paper}

After this introduction, this paper consists of eight sections. The first section contains the statement of the main result, Theorem \ref{thm:main}, after the necessary preliminaries. This section is numbered~S in order to spare the numbers $1$ to $7$ for the next seven sections, which present, in seven steps, the computation that constitutes the proof of the main theorem. 

We tried to start each of these sections with a presentation of the tools that we use in it, to continue with a preparation of the application of these tools in the context of the computation, and to conclude with an intermediary result, that is stated, and proved. 

A summary of each step of the proof that is more detailed than the one that we could give at this stage can be found in Section \ref{sec:structure}, after the statement of the main result.

The last section contains several complements to the main theorem, notably a proof of the rationality of the product of sines and cosines mentioned above, and a proof of the fact that our main result implies the Makeenko--Migdal equations with little effort. 
\medskip

For the convenience of the reader, we provide an index after the bibliography summarizing the notations used throughout the paper. 

\subsection*{Acknowledgements} The main result of this paper was first presented at a conference held at the Center for Mathematical Sciences in Cambridge, in honour of James R. Norris's 60th birthday. The conference, originally planned for 2021, was postponed by the Covid-19 pandemic and finally took place in September 2022. This paper was meant as a small birthday dedication to James, but the writing took longer than expected—so we hope it can now be enjoyed as a gift for his 65th birthday instead.
\medskip

Another effect of the delay with which this paper was written after the result was presented is that I could benefit from several stimulating and helpful conversations, notably with Antoine Dahlqvist, whom I warmly thank for this, as well as for his help in the proof of Proposition \ref{prop:rational}. Antoine also recently completed a related work, that will be published at the same time as this one \cite{Antoine}.  


\setcounter{section}{18}
\renewcommand*{\thesection}{\Alph{section}}

\section{Statement of the main result}

\subsection{The Yang--Mills holonomy process} \label{sec:ymhp} 

The goal of this paper is to establish a formula for {\em Wilson loop expectations}, which are the $n$-point functions of the {\em Yang--Mills holonomy process}, a group-valued stochastic process indexed by loops on a manifold. The only situation so far where this stochastic process can be mathematically defined is when the manifold is an oriented compact surface. In this case, two objects are needed to construct the Yang--Mills holonomy process. 

\begin{enumerate}[\sbullet]
\item A compact connected Lie group, the Lie algebra of which is endowed with an inner product invariant under the adjoint action. In this paper, we fix an integer $N\geq 1$ and take the Lie group to be the unitary group $\U(N)=\{U\in M_{N}(\C) : UU^{*}=I_{N}\}$. We endow its Lie algebra $\u(N)=\{X\in M_{N}(\C) : X+X^{*}=0\}$ with the invariant inner product $\langle X,Y\rangle=N\Tr(X^{*}Y)$, where $\Tr$ denotes the usual trace of a matrix, in the sense that $\Tr(I_{N})=N$. 
\item A smooth oriented connected compact surface, that we fix and denote by $\Sigma$, endowed with a measure of area, that is, a finite measure on the Borel $\sigma$-field of $\Sigma$ that is equivalent to the Lebesgue measure in every coordinate chart. We denote by~$|F|$ the area of a measurable subset $F$ of $\Sigma$. 

\index{Sigma@$\Sigma$, compact surface}
\index{ar@ $\lvert \ \cdot\  \rvert$, measure of area}

The surface~$\Sigma$ can be closed, or have a boundary, and we will deal with both cases at the same time. If $\Sigma$ has a boundary, then it is possible to impose, at some of its boundary components, boundary conditions for the holonomy process, each boundary condition being specified by a conjugacy class of the group $\U(N)$.

If $\Sigma$ has a boundary, we split the set of its boundary components into the subset of those that we called {\em constrained} and to which we attach a boundary condition, and the complementary subset, of {\em free} boundary components. 

We will assume that $\Sigma$ has $\k$ constrained boundary components $C_{1},\ldots,C_{\k}$, on which are imposed the boundary conditions specified by the respective conjugacy classes of $\k$ elements $w_{1},\ldots,w_{\k}$ of $\U(N)$.
We will denote collectively by $\ul{w}=(w_{1},\ldots,w_{\k})$ the set of boundary conditions. 

\index{wC @ $\ul{w}$, boundary conditions}
\index{k @ $\k$, number of constrained b.c.\ of $\Sigma$}
\end{enumerate}
\smallskip

Let $S^{1}$ denote a smooth oriented $1$-manifold diffeomorphic to a circle and let $(S^{1},\vsbullet)$ denote the same manifold with a distinguished point. For us, and although a weaker regularity could be allowed, a {\em based loop} on $\Sigma$ will be a piecewise smooth immersion of $(S^{1},\vsbullet)$ in $\Sigma$. 

As explained for instance in \cite{LevyMM}, to the Lie group $\U(N)$ and the invariant inner product on its Lie algebra $\u(N)$, the surface $\Sigma$ with its measure of area, and possibly boundary conditions along some boundary components of $\Sigma$, is associated the distribution (in the probabilistic sense of the word) of a $\U(N)$-valued stochastic process indexed by the set of based loops on $\Sigma$, called the Yang--Mills holonomy process. This means that there exists a probability space $(\Omega,\mathscr F,\mathbf P)$ and, for every based loop $l$, a $\U(N)$-valued random variable $H_{l}$ on this probability space, such that the joint distribution of any finite collection $(H_{l_{1}},\ldots,H_{l_{n}})$ of these random variables is a certain prescribed probability measure on $\U(N)^{n}$, which depends on the inner product on the Lie algebra $\u(N)$, the surface, the measure of area, the boundary conditions, and of course on the based loops $l_{1},\ldots,l_{n}$ (see Section \ref{sec:DSformula}).

\index{H@$H_{l}$, Yang--Mills process}

An important property of the holonomy process is that if two based loops $l$ and $l'$ differ by a positive reparametrisation, a change of basepoint, or both, that is, if they differ by pre-composition by an orientation-preserving diffeomorphism of $S^{1}$ that does not necessarily preserve the distinguished point, then the $\U(N)$-valued random variables $H_{l}$ and $H_{l'}$ belong almost surely to the same conjugacy class. In particular, the complex-valued random variables $\Tr(H_{l})$ and $\Tr(H_{l'})$ are almost surely equal. For this reason, our main result involves unbased loops, more precisely nice finite collections of unbased loops, that we will now describe.

\subsection{Loop configurations and Wilson loop expectations} \label{sec:wle}
By a  {\em loop configuration} on $\Sigma$, we mean a smooth immersion $\ul\ell=(\ell_{1},\ldots,\ell_{m}):(S^{1})^{\sqcup m} \to \Sigma$ of the disjoint union of a certain number $m\geq 1$ of copies of $S^{1}$ into $\Sigma$, that has no triple points nor multiple points of higher order, and that is transverse to itself at each double point, if any. If $\Sigma$ has a boundary, we assume that the range of $\ul{\ell}$ does not meet $\partial\Sigma$. 

\index{l@$\ul\ell$, loop configuration}

Around any point of its range, a loop configuration is, up to diffeomorphism, a line through this point, or two lines crossing at this point. Because $\Sigma$ is compact, the second possibility occurs at finitely many points, called the {\em double points} of the loop configuration. A connected component of the range of a loop configuration that contains no double point is diffeomorphic to a circle, and is equal to the range of one of the immersions $\ell_{1},\ldots,\ell_{m}$, that is in fact an embedding. We will call such a connected component an {\em isolated embedded loop}.

Two loop configurations differ by pre-composition by a smooth orientation-preserving diffeomorphism of $(S^{1})^{\sqcup m}$ if and only if they have the same range with the same orientation. This defines a natural equivalence relation on loop configurations, and all the quantities that we will consider are invariant under the replacement of a loop configuration by an equivalent one.
\smallskip

Our main result gives, for every loop configuration $\ul\ell=(\ell_{1},\ldots,\ell_{m})$ on $\Sigma$, a formula for the {\em non-normalized Wilson loop expectation}
\begin{equation}\label{eq:WLE}
Z(\Sigma)\,  {\bf E}\big[\Tr(H_{\ell_{1}})\ldots \Tr(H_{\ell_{m}})\big] \ \, \text{ or } \ \, Z(\Sigma,\ul{w})\,  {\bf E}\big[\Tr(H_{\ell_{1}})\ldots \Tr(H_{\ell_{m}})\big],
\end{equation}
depending on whether $\Sigma$ is closed or has a boundary, where ${\bf E}$ is the expectation with respect to the probability ${\bf P}$ underlying the Yang--Mills holonomy process, and $Z(\Sigma)$ or  $Z(\Sigma,\ul{w})$ is the partition function of the 2-dimensional Yang--Mills measure on $\Sigma$,  is a positive number that will be defined precisely later, see \eqref{eq:defZ}, \eqref{eq:Z=1}, and \eqref{eq:defZwb}. 

It may seem that expectations of the form \eqref{eq:WLE} are only a small part of those that one could form with the holonomy process. Indeed, we restrict ourselves to very regular finite collections of loops, and we do not consider the expectation of products of traces of {\em products} of random variables. However, the holonomy process enjoys properties of multiplicativity, gauge invariance, and continuity (see \cite{LevyAMS, LevyMHF}) which together imply that the Wilson loop expectations of the form \eqref{eq:WLE}, for all loop configurations as defined above, characterize completely its distribution. 

The expression that we give for \eqref{eq:WLE} takes the form of a sum, indexed by assignments of an irreducible representation of $\U(N)$ to each face of the graph drawn on $\Sigma$ by the loop configuration. In order to state the result, we need to discuss these objects and introduce some notation.

\subsection{The graph associated to a loop configuration}\label{sec:graphskein}
Let us consider a loop configuration $\ul\ell=(\ell_{1},\ldots,\ell_{m})$ on $\Sigma$. This loop configuration gives rise to a graph $\sG$ embedded in $\Sigma$, of which the sets $\sV$ of vertices, $\sE$ of edges and $\sF$ of faces are defined as follows. 

\begin{enumerate}[\noindent {\raisebox{0.5pt}{$\scriptstyle\bullet$}}]
\item $\sV$ is 
the finite subset of $\Sigma$ consisting of the double points of the loop configuration $\ul\ell$. 
\item $\sE$ is the set of restrictions of $\ul\ell$ to the closures of the connected components of $(S^{1})^{\sqcup m} \setminus \ul\ell^{-1}(\sV)$.
\item $\sF$ is the set of connected components of $\Sigma\setminus \ul\ell\big((S^{1})^{\sqcup m}\big)$. 
\end{enumerate}

\index{V@$\sV$, vertices of $\sG$}
\index{E@$\sE$, edges of $\sG$}
\index{F@$\sF$, faces of $\sG$}
\index{G@$\sG$, graph}

There is little to say about vertices. Let us discuss edges. Each circle of $(S^{1})^{\sqcup m}$ can contain no point of $\ul\ell^{-1}(\sV)$, or one, or several, and this gives rise to three types of edges. A circle that contains no point of $\ul\ell^{-1}(\sV)$ parametrises an isolated embedded loop, that is indeed an edge according to our definition, and that we call a {\em circular edge}. It is an edge without an endpoint. A circle that contains one preimage of a vertex yields an edge that is parametrised by a circle, but contains a vertex. We think of such an edge as being parametrised by a segment and having identical endpoints, and call it a {\em closed linear edge}. Finally, a circle that contains several preimages of vertices gives rise to edges that are embeddings of segments, and that we call {\em open linear edges}.

Let us introduce the notation
\[\sLE = \{e \in \sE : e \text{ linear edge}\} \ \text{ and } \  \sCE = \{e \in \sE : e \text{ circular edge}\},\]
the distinction between closed and open linear edges being less essential.

\index{LE@$\sLE$, linear edges of $\sG$}
\index{CE@$\sCE$, circular edges of $\sG$}

Circular edges are somewhat unusual in the context of a graph, and they will sometimes need a special treatment, but this will not create any difficulty. 

\index{Sigmabar@$\bar \Sigma$, surface $\Sigma$ blown up along $\sG$}
\index{Fbar@$\bar F$, compactification of a face}

Let us now discuss faces. Our definition of faces is simple, but as we define them, they are not compact topological spaces. For several reasons, it will be convenient to have an alternative description of the faces of $\sG$ as compact surfaces with boundary. Intuitively, these surfaces are the connected component of the result of cutting (in the most concrete sense) the surface of $\Sigma$ along the arcs of the loop configuration. They can be defined rigorously as arising from a real oriented blowup of $\Sigma$ along~$\sG$. This operation replaces each point $m$ of $\Sigma$ by the set of germs of pairs $(m,U)$, where $U$ is a connected component of the intersection of a neighbourhood of $m$ with the complement of the range of $\sG$, such that the closure of $U$ contains $m$. Concretely, one takes one copy of each point of $\Sigma$ that is not on the range of $\sG$, one replaces each point that lies on a strand of a loop by two points, one attached to each side of the loop, and each point that lies at the crossing of two loops by four points, one attached to each corner defined by the crossing. In this way, one produces a surface with boundary $\bar \Sigma$, together with a projection map $\bar \Sigma \to \Sigma$. For each face $F$ of $\sG$, there is a unique connected component $\bar F$ of $\bar \Sigma$ of which the image by the projection map contains $F$. This component $\bar F$ is the compact surface that we were looking for as being attached to the face $F$. There is no restriction on the topology of $\bar F$ which can be, depending on $\Sigma$ and $\ul\ell$, any orientable compact surface with boundary. 

The reason for this relatively complicated construction is that in general, $\bar F$ is not homeomorphic to the closure of $F$ in $\Sigma$, see Figure~\ref{fig:surface} for an example. Nevertheless, in a non-canonical way, $\bar F$ is always diffeomorphic to a subset of~$\Sigma$, namely to the connected component contained in $F$ of the complement in $\Sigma$ of a small open tubular neighbourhood of the range of the loop configuration $\ul\ell$ (this sentence needs to be read backwards). 

If $\Sigma$ is closed, then the boundary of $\bar F$ arises from the operation of blowing up. However, if $\Sigma$ has a boundary, $F$ may contain certain components of $\partial \Sigma$, which are also components of $\partial \bar F$ : we call such boundary components {\em internal}, and the others, arising from the blow up {\em external}.

To each face $F$ of $\sG$, we attach the following integers:
\begin{enumerate}[\indent \sbullet]
\item $\e_{F}$, the Euler characteristic of $\bar F$,
\item $\b_{F}$, the number of boundary components of $\bar F$,
\item $\x_{F}$, the number of external boundary components of $\bar F$, equal to $\b_{F}$ if $\Sigma$ is closed.
\end{enumerate}

\index{bF@$\b_{F}$, number of boundary components of $\bar F$}
\index{eF@$\e_{F}$, Euler characteristic of $\bar F$}
\index{bxF@$\x_{F}$, number of external b.c.\ of $\bar F$}

\begin{figure}[h!]
\begin{center}
\includegraphics{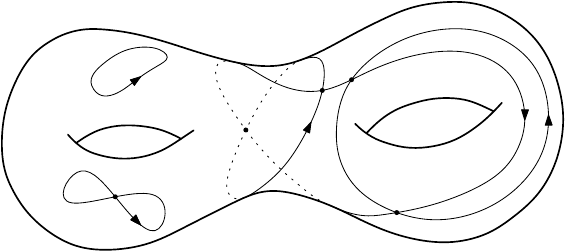}
\caption{\label{fig:surface}\small This loop configuration on a surface of genus $2$ contains five loops, one of which is an isolated embedded loop. It produces a graph with 5 vertices, 11 edges, one of which is circular, and 8 faces. The face $F$ containing the left handle of the surface has a boundary with $\b_{F}=3$ connected components and Euler characteristic $\e_{F}=-3$. However, the closure of $F$ in~$\Sigma$ is not homeomorphic to a torus with $3$ holes. Instead, it is homeomorphic to a torus with $3$ holes of which two distinct points of the same boundary component have been identified, to produce the vertex located on the $8$-shaped loop.}
\end{center}
\end{figure}

\subsection{Configurations of highest weights}\label{sec:CHW}

Our main result expresses the non-normalised Wilson loop expectation~\eqref{eq:WLE} as a sum over {\em configurations of highest weights} on the faces of the graph~$\sG$.

By a {\em highest weight}, we mean a non-increasing $N$-tuple of elements of~$\Z$. We denote the set of highest weights by $\HW$. 

To each highest weight $\lambda=(\lambda_{1}\geq\ldots \geq \lambda_{N})$ we associate two non-negative integers,
its {\em dimension}
\begin{equation}\label{eq:Weyldim}
d_{\lambda}=\prod_{1\leq i < j \leq N} \frac{\lambda_{i}-\lambda_{j}+j-i}{j-i}
\end{equation}
and its {\em quadratic Casimir number}
\begin{equation}\label{eq:casimir}
c_{\lambda}=\sum_{1\leq i \leq N} \lambda_{i}^{2}+\sum_{1\leq i < j \leq N}(\lambda_{i}-\lambda_{j})=\sum_{i=1}^{N} \lambda_{i}^{2}+(N-2i+1)\lambda_{i}.
\end{equation}
These numbers arise from the fact that the set of highest weights are in one-to-one correspondence with the set of isomorphism classes of irreducible representations of the group $\U(N)$. The number $d_{\lambda}$ is the dimension of the representation parametrised by the highest weight $\lambda$, and~\eqref{eq:Weyldim} is called the Weyl dimension formula. The character of the representation parametrised by $\lambda$ is a complex-valued function on $\U(N)$ that is an eigenfunction, with eigenvalue $-\frac{1}{N}c_{\lambda}$, of the Laplace--Beltrami operator on $\U(N)$ associated to the bi-invariant Riemannian metric determined by the invariant inner product on $\u(N)$. We will discuss characters of unitary representations in more detail later, see Section \ref{sec:SFFS}.
\medskip

\index{ZN@$\HW$, set of highest weights of $\U(N)$}
\index{dlambda@$d_{\lambda}$, dimension of the irrep $\lambda$}
\index{clambda@$c_{\lambda}$, quadratic Casimir number of the irrep $\lambda$}
\index{lzambda@$\lambda$, irrep of $\U(N)$}

A {\em configuration of highest weights} on the faces of the graph~$\sG$ is a map $\Lambda:\sF\to \HW$ from the set of faces of $\sG$ to the set of highest weights. We will always denote by $\lambda_{F}$ the highest weight associated to a face $F$, instead of $\Lambda(F)$. 
\index{Lzambda@$\Lambda$, configuration of highest weights}

Only certain special configurations of highest weights will appear in our main result. To explain which ones, and how they contribute, we need to introduce some more notation.

Consider $\lambda,\mu\in \HW$. We say that $\mu$ {\em extends} $\lambda$  if 
\begin{equation}\label{eq:extends}
\lambda=(\lambda_{1},\ldots,\lambda_{i},\ldots,\lambda_{N}) \ \text{ and } \ 
\mu=(\lambda_{1},\ldots,\lambda_{i}+1,\ldots,\lambda_{N})
\end{equation}
for some $i\in \{1,\ldots,N\}$. In this case, we write $\lambda\leads\mu$ and we define 
\begin{equation} \label{eq:defcont}
\cont(\mu/\lambda)=\mu_{i}-i.
\end{equation}
When the components of $\lambda$ and $\mu$ are non-negative, and in the context of the representations of symmetric groups, this integer is known as the {\em content} of the only box of the skew Young diagram $\mu/\lambda$, that is, the content of the unique box of the Young diagram of shape $\mu$ that is not in the Young diagram of shape $\lambda$. 

\index{lzambdamu@$\lambda\leads\mu$, $\mu$ extends $\lambda$}
\index{cmulambda@$\cont(\mu/\lambda)$, content of $\mu/\lambda$}

An elementary observation that will play an important role for us is the following.

\begin{lemma}\label{lem:contpasegal}
If three highest weights $\lambda, \mu,\xi$ are such that $\lambda\leads\mu\leads\xi$, then $\cont(\xi/ \mu)\neq \cont(\mu/ \lambda)$. 
\end{lemma}

\begin{proof}
If $\mu=(\lambda_{1},\ldots,\lambda_{i-1},\lambda_{i}+1,\lambda_{i+1},\ldots,\lambda_{N})$ and $\xi=(\mu_{1},\ldots,\mu_{j-1},\mu_{j}+1,\mu_{j+1},\ldots,\mu_{N})$, then either $j\leq i$ and $\cont(\xi/\mu)=\mu_{j}-j\geq \mu_{i}-i=\lambda_{i}-i+1>\lambda_{i}-i=\cont(\mu/\lambda)$, or $j\geq i+1$, in which case $\cont(\xi/\mu)=\mu_{j}-j\leq \mu_{i+1}-i-1=\lambda_{i+1}-i-1<\lambda_{i}-i=\cont(\mu/\lambda)$.
\end{proof}

For each edge $e$ of $\sG$, let us denote by $\lf_{e}$ (resp. $\rf_{e}$) the face of $\sG$ located immediately on the left (resp. on the right) of $e$. This notion makes sense because each edge is oriented, as well as the surface $\Sigma$. 

\index{Le@$\lf_{e}$, the face on the left of the edge $e$}
\index{Re@$\rf_{e}$, the face on the right of the edge $e$}

Recall from Section \ref{sec:ymhp}  that a free boundary component of $\Sigma$ is one at which no boundary condition is imposed on the holonomy process. 

\begin{definition} \label{def:wellbalanced}
We say that a configuration of highest weights $\Lambda:\sF\to \HW$ is {\em balanced} if for each 
edge $e\in \sE$, the highest weight $\lambda_{\rf_{e}}$ extends $\lambda_{\lf_{e}}$:
\[\forall e\in \sE, \ \ \lambda_{\lf_{e}}\leads \lambda_{\rf_{e}}.\]
We say that the configuration {\em vanishes at the free boundary} if for every face $F\in \sF$ that contains a free boundary component of $\Sigma$, 
\[\lambda_{F}=(0,\ldots,0).\]
\end{definition}

We denote by $\B$ the set of balanced configurations of highest weights that vanish at the free boundary, the dependence of this set on $\Lambda$ and $N$ being implicit.
Note that it can happen that for some edge $e$ one has $\lf_{e}=\rf_{e}$, for example if $\Sigma$ is a torus and $\ul\ell$ consists in a single meridian of $\Sigma$. In this case, no configuration of highest weights is balanced. More generally, we shall prove that if the homology class of $\ul\ell$ in $H_{1}(\Sigma)$ does not belong to the subgroup generated by the classes of the constrained boundary components, then there does not exist any balanced configuration of highest weights that vanishes at the free boundary (see Proposition \ref{prop:homcond} and its proof).

\index{B@$\B$, balanced configurations of highest weights}

The most characteristic aspect of our main result is the appearance of the cosine or sine of an angle attached to each vertex of $\sG$ and which depends on the configuration $\Lambda$. Let us explain how this angle is defined.

Let us consider $\Lambda \in \B$ and a vertex $v\in \sV$. Let us label the four faces of $\sG$ adjacent to~$v$ as $\cn_{v}, \cw_{v}, \cs_{v}, \ce_{v}$, in reference to the cardinal points,  in the unique way consistent with the orientation of $\Sigma$ and such that $\cn_{v}$ is the face located between the two outgoing edges at $v$ (see Figure \ref{fig:cardinal} below). Note that these faces need not be pairwise distinct.

\index{Nv@$\cn_{v}, \cw_{v}, \cs_{v}, \ce_{v}$,  faces around the vertex $v$}

Since $\Lambda$ is balanced, we have $\lambda_{\cw_{v}}\leads \lambda_{\cn_{v}}\leads \lambda_{\ce_{v}}$. We define $\theta^{\Lambda}_{v}$ as the unique element of the real interval $[0,\pi]$ such that
\begin{equation}\label{eq:defthetavL}
\cos \theta^{\Lambda}_{v}=\frac{1}{\cont(\lambda_{\ce_{v}}/ \lambda_{\cn_{v}})-\cont(\lambda_{\cn_{v}}/ \lambda_{\cw_{v}})},
\end{equation}
the denominator being non zero according to Lemma \ref{lem:contpasegal}.
In the case where all highest weights have non-negative components, and are identified with integer partitions, the absolute value of the denominator is the length of the hook joining the two boxes of the skew diagram $\lambda_{\ce_{v}}/\lambda_{\cw_{v}}$. Moreover, the sign of the denominator is positive if and only if the box $\lambda_{\ce_{v}}/ \lambda_{\cn_{v}}$
is above or on the right of $\lambda_{\cn_{v}}/ \lambda_{\cw_{v}}$, in the graphical representation adopted in Figure \ref{fig:cardinal}. In more intrinsic terms, let $i$ and $j$ be the indices such that $(\lambda_{\cn_{v}})_{i}=(\lambda_{\cw_{v}})_{i}+1$ and $(\lambda_{\ce_{v}})_{j}=(\lambda_{\cn_{v}})_{j}+1$. Then $\cos \theta^{\Lambda}_{v}$ is positive if and only if $j\leq i$.
  
\index{thetav@$\theta^{\Lambda}_{v}$,  angle associated to $\Lambda$ around $v$} 
 
The configuration $\Lambda$ still being fixed, let us say that 
\begin{equation}\label{eq:deftype}
\text{a vertex } v\in \sV \text{ is {\em of type $1$} if } \lambda_{\cn_{v}}=\lambda_{\cs_{v}}, \text{ and {\em of type $2$} if } \lambda_{\cn_{v}}\neq\lambda_{\cs_{v}}.
\end{equation}
We denote respectively by $\sV^{\Lambda}_{1}$ and $\sV^{\Lambda}_{2}$ the sets of vertices of type $1$ and of type $2$.

\begin{figure}[h!]
\begin{center}
\includegraphics{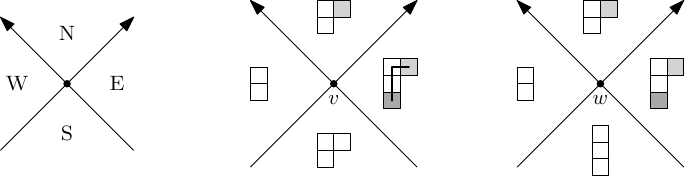}
\caption{\label{fig:cardinal}\small On this picture, highest weights have non-negative components and are represented by Young diagrams, the lengths of the rows corresponding to the non-zero components of the highest weights. For instance, $\lambda_{{\ce}_{v}}=(2,1,1,0,\ldots,0)$.
The vertex $v$ is of type 1, and  $\cos \theta^{\Lambda}_{v}=-\frac{1}{3}$. The vertex $w$ is of type~2, and $\theta^{\Lambda}_{w}=\theta^{\Lambda}_{v}$.}
\end{center}
\end{figure}

\index{Va@$\sV^{\Lambda}_{1},\sV^{\Lambda}_{2},$, vertices of type $1$, of type $2$} 

\subsection{Main result}

We can finally state our main result, using the notation introduced so far. 

\begin{theorem}\label{thm:main}
Consider the Yang--Mills holonomy process on a compact surface $\Sigma$ with structure group $\U(N)$ and invariant scalar product $\langle X,Y\rangle=N\Tr(X^{*}Y)$ on $\u(N)$. The non-normalised Wilson loop expectation associated to a loop configuration $\ul\ell=(\ell_{1},\ldots,\ell_{m})$ on $\Sigma$ is given by
\begin{align*}
\phantomsection
\label{eq:main}
\tag{\raisebox{-0.5pt}{\text{\ding{72}}}} 
&Z(\Sigma)\, {\bf E}\Big[\prod_{j=1}^{m} \Tr(H_{\ell_{j}})\Big]\\
&\hspace{2cm}=\sum_{\Lambda\in \B} e^{\!-\frac{1}{2N}\! \underset{F\in \sF}{\sum}  |F| c_{\lambda_{F}}}\!
\prod_{F\in \sF} (d_{\lambda_{F}})^{\e_{F}}
\prod_{v\in \sV} (d_{\lambda_{\cn_{v}}}d_{\lambda_{\cs_{v}}})^{-\frac{1}{2}}
\prod_{v \in \sV^{\Lambda}_{1}} \cos \theta^{\Lambda}_{v}\prod_{v \in \sV^{\Lambda}_{2}} \sin \theta^{\Lambda}_{v}.
\end{align*}

If $\Sigma$ has a boundary, and if boundary conditions $w_{1},\ldots,w_{\k}$ are imposed on boundary components that are respectively contained in the faces $F_{1},\ldots,F_{\k}$ of the graph determined by the loop configuration $\ul\ell$, then
\begin{align*}
\phantomsection
\label{eq:main2}
&Z(\Sigma,\ul{w})\, {\bf E}\Big[\prod_{j=1}^{m} \Tr(H_{\ell_{j}})\Big] 
\tag{\raisebox{-0.5pt}{\text{\ding{73}}}} 
\\
&\hspace{0.5cm}=\sum_{\Lambda\in \B} e^{\!-\frac{1}{2N}\! \underset{F\in \sF}{\sum}  |F| c_{\lambda_{F}}}\!
\prod_{i=1}^{\k} s_{\lambda_{F_{i}}}(w_{i})
\prod_{F\in \sF} (d_{\lambda_{F}})^{\e_{F}}
\prod_{v\in \sV} (d_{\lambda_{\cn_{v}}}d_{\lambda_{\cs_{v}}})^{-\frac{1}{2}}
\prod_{v \in \sV^{\Lambda}_{1}} \cos \theta^{\Lambda}_{v}\prod_{v \in \sV^{\Lambda}_{2}} \sin \theta^{\Lambda}_{v}.
\end{align*}
\end{theorem}

A few remarks are in order. 

\begin{enumerate}[\sbullet]
\item The closed case is a special case of the case with boundary. The only three differences between the two expressions are the value of the partition function, the constraint on the configurations of highest weights, hidden in the definition of $\B$, of vanishing at the free boundary, and the presence of the product of the Schur functions evaluated at the boundary conditions. 
\smallskip

\item In each term of the sum, the sum of the exponents of the dimensions $d_{\lambda}$ is equal to the Euler characteristic of $\Sigma$, that we denote by ${\sf e}_{\Sigma}$. To put it in symbols, we mean that the following equality holds:
\[\sum_{F\in \sF}{\sf e}_{F} - |\sV|={\sf e}_{\Sigma}.\]
Indeed, since the blow up of $\Sigma$ along $\sG$ that produces $\bar \Sigma$ duplicates the edges and turns each vertex into four points, we can compute the Euler characteristic $\e_{\bar \Sigma}$ of $\bar \Sigma$ in two ways, taking into account that circular edges have Euler characteristic $0$:
\[\sum_{F\in \sF} \e_{F}=\e_{\bar \Sigma}=\e_{\Sigma}-|\sLE|+3|\sV|.\]
Counting in two ways the linear half-edges of $\sG$ yields $2|{\sf L}\sE|=4|\sV|$, and the equality follows.
\smallskip

\item A case that is often considered in $2$-dimensional Yang--Mills theory is that where the surface~$\Sigma$ is the plane. As far as the holonomy along the loops of a loop configuration is concerned, there is no difference between the whole plane and a large disk, that is, a disk large enough to contain the loop configuration, with a free boundary condition. Therefore, the case of the plane is treated by the theorem.
\smallskip

\item When $\Sigma$ has at least one boundary component that is left free of any boundary condition, the partition function $Z(\Sigma,\ul{w})$ is equal to $1$, see \eqref{eq:Z=1}. 
\smallskip

\item By definition, $\cos \theta^{\Lambda}_{v}$ is a rational number, but in general $\sin \theta^{\Lambda}_{v}$ is not. However, we will prove that in the last product, over vertices of type $2$, the sine of each angle appears an even number of times (see Proposition \ref{prop:rational}). 
\end{enumerate}

\subsection{Structure of the proof}\label{sec:structure}

The proof of Theorem \ref{thm:main} is a fairly substantial computation, that we will do step by step, and that will occupy us for the next seven sections. 
\begin{enumerate}[\sbullet]
\item In Section \ref{sec:DS}, we use the characterisation of the Yang--Mills holonomy process through the Driver--Sengupta formula to give a first expression of the Wilson loop expectation as an integral over a finite product of copies of the unitary group, see Proposition \ref{prop:expressiontocompute}.
\item In Section \ref{sec:expansionHK}, we expand the integrand of this integral, a product of heat kernels, in Fourier, or rather Peter--Weyl, series. This gives the Wilson loop expectation the form of an infinite sum of integrals over the same finite product of copies of the unitary group, but now of rational functions, see Proposition \ref{prop:expanded}. From now on, our attention focuses on this integral, that we call the {\em flat contribution} of a configuration of highest weights, see \eqref{eq:I}.
\item In Section \ref{sec:Schurusingperm}, we use the Schur--Weyl duality to express Schur functions in terms of permutations and actions of symmetric group on various vector spaces. This results in an expression, in the spirit of spin networks, of the flat contribution of a configuration of highest weights as the trace of the product of three linear operators, only one of which is still an integral, see Proposition \ref{prop:I3}. 
\item In Section \ref{sec:integrationoverU}, we use the Collins--\'Sniady integration formula to perform the integration, over a product of copies of the unitary group. We then reorganise the trace of the product of three linear operators and arrive at an expression that does not involve unitary matrices anymore, and that is now a sum, over a product of symmetric groups, of the product of local contributions, one for each vertex of the graph, see Proposition \ref{prop:I3.5}.
\item In Section \ref{sec:symmetricmodules}, we go one step further in leaving the unitary world, and express all the quantities in terms of irreducible representations of the symmetric groups, see Proposition~\ref{prop:I4}. This is actually a small step, and the role of this section is to set the scene for the last two steps of the proof, in which the representations of the symmetric groups play the main role. 
\item In Section \ref{sec:summingpermutations}, we perform the sum over a product of symmetric groups that appears in our current expression of the flat contribution of a configuration of highest weights, and finally obtain a concise expression that looks quite similar to the expression that we seek, see Proposition \ref{prop:expI5}. A scalar, defined in \eqref{eq:defa}, makes its appearance during this computation, and remains to be computed.
\item In Section \ref{sec:detscal}, we use the Okounkov--Vershik approach to the representations of symmetric groups to compute this scalar, of which the value is given by Proposition \ref{prop:determinea}. The sign of this scalar depends, in some cases, on certain choices made during the proof, and a part the section is devoted to the construction of good choices, namely of good bases of irreducible symmetric modules, see Proposition \ref{prop:GZbasis}. The final expression of the flat contribution of a configuration of highest weights is given by Proposition \ref{prop:eqI6}, and the proof is then concluded in a few lines. 
\end{enumerate}

Section \ref{sec:after} is not part of the proof, but contains some further results. In particular, we
\begin{enumerate}[\indent\sbullet]
\item prove that the flat contribution of a configuration of highest weights is a rational number,
\item prove that the Wilson loop expectation computed by our main theorem is equal to $0$ if the total homology class of the loop configuration does not vanish in the relative homology group $H^{1}(\Sigma,\partial \Sigma)$,
\item describe from a cohomological point of view the sign ambiguity that is lifted by appropriate choices in Section \ref{sec:detscal},
\item explain how our main result allows one to recover the Makeenko--Migdal equations, which are, in informal terms, a set of differential equations on the space of loops satisfied by Wilson loop equations. 
\end{enumerate}


\setcounter{section}{0}
\renewcommand*{\thesection}{\arabic{section}}

\section{Step 1 --- Derivation of a first integral expression}\label{sec:DS}

The starting point of the proof is the description of the distribution of the Yang--Mills holonomy process on a compact surface, by means of the Driver--Sengupta formula. In this section, we recall this formula and use it to put the left-hand side of our main equalities \eqref{eq:main} and \eqref{eq:main2} under the form of an integral, see \eqref{eq:integraltocompute}.

A presentation of the lattice Yang--Mills theory and of the Driver--Sengupta formula can be found for instance in \cite{SenguptaAMS,LevyMHF,LevyMM}. 

\subsection{Graphs}

The Driver--Sengupta formula describes the joint distribution of the random holonomy along any finite collection of loops traced in a graph on $\Sigma$, provided that the faces of this graph are homeomorphic to disks. Let us start by introducing a notion of graph appropriate to the loop configurations we are working with.

By an {\em edge} on $\Sigma$, we mean any one of the following (see Figure \ref{fig:edges}):
\begin{itemize}
\item a circular edge, that is, an embedding of $S^{1}$ into $\Sigma$, 
\item a closed linear edge, that is, an embedding of $(S^{1},\vsbullet)$, into $\Sigma$,
\item an open linear edge, that is, an embedding of $[0,1]$ into $\Sigma$.
\end{itemize}

Recall that $S^{1}$ is oriented, so that all edges are oriented. A circular edge has no endpoint. An open linear edge has one endpoint which is the image of the distinguished point. A closed linear edge has two distinct endpoints.

If $\Sigma$ has a boundary, we add the condition that an edge must either be disjoint from $\partial \Sigma$, or intersect $\partial \Sigma$ only at one or both of its endpoints, or be completely contained in $\partial \Sigma$, and in this last case be positively oriented as a piece of the boundary component that contains it. 

By a {\em graph} on $\Sigma$, we mean a collection of edges on $\Sigma$ that are pairwise disjoint, except for possible encounters at their endpoints. The {\em vertices} of a graph are defined as the endpoints of its edges, and its {\em faces} as the connected components of the complement of the union of the edges. A graph is denoted as a triple $\bG=(\bV,\bE,\bF)$ consisting of the sets of vertices, edges, and faces. It is entirely determined by the set of  edges. 

\begin{figure}[h!]
\begin{center}
\includegraphics{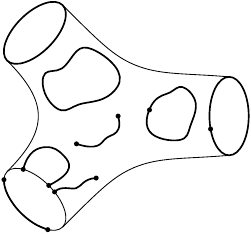}
\caption{\label{fig:edges}\small This figure illustrates, in the case where $\Sigma$ is a $3$-holed sphere, the various kinds of edges that we allow ourselves to consider, and their relation to the boundary of the surface.}
\end{center}
\end{figure}

A graph is said to be {\em cellular} if its faces are homeomorphic to open disks. There are only two cases where a cellular graph can contain a circular edge: the first is when $\Sigma$ is a disk and the graph has a single edge that goes along the boundary of the disk, and the second is when $\Sigma$ is a sphere and the graph has a single edge that is a circular edge. 

Note that if $\Sigma$ has a boundary and $\bG$ is a cellular graph on $\Sigma$, then the boundary of $\Sigma$ must be covered by $\bG$, and every boundary component of $\Sigma$ is the range of a loop of $\bG$.

In the terminology of \cite{LevyMHF}, what we  now call a graph and a cellular graph was respectively called a pre-graph and a graph, with the minor difference that we are allowing circular edges. Another difference is that, in contrast with the presentation of \cite{LevyMHF}, and for the sake of simplicity, we choose to let the set~$\bE$ contain each edge with one orientation only. Let us mention again that the parametrisation of edges is irrelevant to our problem, and we often allow ourselves to identify an edge with its range.

\subsection{The Driver--Sengupta formula} \label{sec:DSformula}

Let us consider a graph $\bG=(\bV,\bE,\bF)$ on $\Sigma$. The configuration space of lattice Yang--Mills theory on $\bG$ is the Cartesian product $\U(N)^{\bE}$. An element of this configuration space will be called a {\em lattice gauge configuration} and will generically be denoted by $g=(g_{e})_{e\in \bE}$. 

A {\em path} in $\bG$ is a well-chained sequence of edges, and a {\em loop} is a path with identical initial and final point. A path $c$ can be written as $c=e_{1}^{\epsilon_{1}}\ldots e_{k}^{\epsilon_{k}}$, where $e_{1},\ldots,e_{k}$ are edges  and the exponents $\epsilon_{1},\ldots,\epsilon_{k}=\pm 1$ indicate if an edge is traversed with positive or negative orientation. To a path~$c$, written in this way, we associate the discrete holonomy map $h_{c}:\U(N)^{\bE}\to \U(N)$ defined by  $h_{c}(g)=g_{e_{k}}^{\epsilon_{k}}\ldots g_{e_{1}}^{\epsilon_{1}}$. 

Let us now assume that $\bG$ is a cellular graph and write the Driver--Sengupta formula. To each face $F$ of $\bG$, which is homeomorphic to a disk, is associated a loop $\partial F$, that goes once positively around the boundary of $F$. This loop can be constructed by performing the same operation of blow up of $\Sigma$ along $\bG$ that we described in Section \ref{sec:graphskein}, considering the closed disk $\bar F$ associated to $F$, and taking the image of the boundary of this disk by the projection map. Let us mention that it can also be constructed by purely combinatorial means, using the cyclic ordering of the edges around each vertex, see for example \cite{LandoZvonkin}. In any case, the origin of this loop is not specified, and it is defined only up to a circular permutation of the sequence of edges that it traverses. As we will see shortly, this is not a problem. 

\index{pt@$p_{t}$, heat kernel on $\U(N)$} 

The Riemannian metric on $\U(N)$ determines a Laplace--Beltrami operator~$\Delta$. The equation $\big(\partial_{t}-\tfrac{1}{2}\Delta\big)p=0$ has a unique smooth positive solution $p:\R^{*}_{+}\times \U(N)\to \R^{*}_{+}$ with the boundary condition $\lim_{t\downarrow 0} \int_{\U(N)}f(x)p(t,x)  \d x=f(I_{N})$ for every continuous real-valued function $f$ on~$\U(N)$, the measure $\nsd x$ being the normalised Haar measure on $\U(N)$. The function $p$ is called the {\em heat kernel} (starting from the identity) on $\U(N)$. For all real $t>0$ and $x\in \U(N)$, we will write $p_{t}(x)$ instead of $p(t,x)$.

For all $t>0$, the function $p_{t}$ on $\U(N)$ is constant on conjugacy classes, so that for all face $F$ of  $\bG$, the positive function $g\mapsto p_{t}(h_{\partial F}(g))$ is well defined on the configuration space. 

In the case where $\Sigma$ is closed, we can write the Driver--Sengupta formula : for all loops $l_{1},\ldots,l_{m}$ in the cellular graph~$\bG$, we have
\begin{equation}\label{eq:DSformula}
Z(\Sigma)\, {\bf E}\Big[\prod_{j=1}^{m} \Tr(H_{l_{j}})\Big]=\int_{\U(N)^{\bE}} \prod_{j=1}^{m} \Tr(h_{l_{j}}(g)) \prod_{F\in \bF} p_{|F|}(h_{\partial F}(g))\d g,
\end{equation}
where $Z(\Sigma)$ is the normalisation constant that ensures that the expectation of the empty product is equal to $1$. It turns out that $Z(\Sigma)$ depends only on the genus and total area of $\Sigma$, and not on the cellular graph $\sG$. We will give an expression of this number shortly, see \eqref{eq:defZ}.

In the case where $\Sigma$ has a boundary, we need to incorporate the boundary conditions, and this is done by changing the reference measure $\nsd g$ on $\U(N)^{\bE}$.  For all $w\in \U(N)$, let $\mathcal O_{w}$ denote the conjugacy class of $w$ in $\U(N)$. For all $r\geq 1$ and $w\in \U(N)$, let $\delta_{[w]}$ denote the unique probability measure on 
\[\{(x_{1},\ldots,x_{r})\in \U(N)^{r} : x_{1}\ldots x_{r} \in \mathcal O_{w}\}\]
that is invariant under the transitive action of $\U(N)^{r}$ on this set by 
\[(y_{1},\ldots,y_{r})\cdot (x_{1},\ldots,x_{r})=(y_{1}x_{1}y_{2}^{-1},y_{2}x_{2}y_{3}^{-1},\ldots,y_{r}x_{r}y_{1}^{-1}).\]
The measure $\delta_{[w]}$ is, in the most natural sense,  the uniform measure on the set of $r$-tuples of elements of $\U(N)$ of which the product belongs to $\mathcal O_{w}$ (see \cite[Section 2.3.2]{LevyMHF} for more details). 

Recall that boundary conditions $w_{1},\ldots,w_{\k}$ are imposed on the constrained boundary components $C_{1},\ldots,C_{\k}$ of $\Sigma$. For each $i\in \{1,\ldots,\k\}$, let us write $C_{i}$ as a product of edges $e_{i,1}\ldots e_{i,r_{i}}$ of $\bG$. Then in order to take the boundary conditions into account, we must replace, in \eqref{eq:DSformula}, the Haar measure $\nsd g$ on $\U(N)^{\bE}$ by the measure $\nsd_{\ul{w}} g$ that is the product of 
\begin{enumerate}[\indent\sbullet]
\item the measure $\delta_{w_{1}}(g_{1,i_{1}},\ldots,g_{1,1})\ldots \delta_{w_{\k}}(g_{\k,i_{\k}},\ldots,g_{\k,1})$ on the edge variables corresponding to the edges located on the constrained boundary components, and 
\item the Haar measure on all the other edge variables. 
\end{enumerate}

This affects the normalisation constant, which now depends on the boundary conditions, that we denote by $Z(\Sigma,\ul{w})$, and of which the value is given by \eqref{eq:Z=1} and \eqref{eq:defZwb}. Therefore, when $\Sigma$ has a boundary, the Driver--Sengupta formula becomes
\begin{equation}\label{eq:DSformulawb}
Z(\Sigma,\ul{w})\, {\bf E}\Big[\prod_{j=1}^{m} \Tr(H_{l_{j}})\Big]=\int_{\U(N)^{\bE}} \prod_{j=1}^{m} \Tr(h_{l_{j}}(g)) \prod_{F\in \bF} p_{|F|}(h_{\partial F}(g))\d_{\ul{w}} g,
\end{equation}

If the loop configuration $\ul\ell=(\ell_{1},\ldots,\ell_{m})$ that we consider is dense enough, in a non-technical sense of the word, the graph $\sG=(\sV,\sE,\sF)$ may very well have faces homeomorphic to disks. In this case, the Driver--Sengupta applies directly to the graph $\sG$. However, in general, there is a small gap to fill in order to apply this formula to a loop configuration that can have isolated embedded loops, and faces with a topology more complicated than that of a disk. 

\subsection{The heat kernel and partition functions}

The integral expression of the Wilson loop expectation that we will obtain at the end of this first step of the proof involves a multivariate version of the heat kernel, that we will now describe. Let us fix a real $t>0$ and two integers~$\b$ and~$\e$, such that $\b\geq 0$ and $\e+\b=2-2\g$ for some non-negative integer $\g$. For all $u,v\in \U(N)$, we use the notation $[u,v]=uvu^{-1}v^{-1}$. Then for all $x_{1},\ldots,x_{\b}\in \U(N)$, we define
\index{ptbe@$p_{t}^{\b,\e}$, multivariate heat kernel}
\begin{equation}\label{eq:defpmulti}
p_{t}^{\b,\e}(x_{1},\ldots,x_{\b})=\int_{\U(N)^{2\g+\b}}\hspace{-5mm}p_{t}\big([u_{1},v_{1}]\ldots [u_{\g},v_{\g}]z_{1}x_{1}z_{1}^{-1}\ldots z_{\b}x_{\b}z_{\b}^{-1}\big)\d u_{1}\nsd v_{1}\ldots \nsd u_{\g}\nsd v_{\g}\d z_{1}\ldots \nsd z_{\b}.
\end{equation}
The function $p_{t}^{\b,\e}$ is  the partition function of the Yang--Mills measure on a compact orientable surface of total area $t$ with Euler characteristic $\e$ and $\b$ boundary components (see \cite{LevyMHF}). It is a symmetric function of $\b$ arguments, invariant by conjugation in each of its arguments. Its integral with respect to any of its arguments is equal to $1$. Note that the function $p^{1,1}_{t}$ is the usual heat kernel $p_{t}$. 

In the case where $\b=0$, the left-hand side of \eqref{eq:defpmulti} is a function without arguments, that is a number, namely the partition function of the Yang--Mills measure on a closed surface of genus~$\g$ and total area $t$. If moreover $\g=0$, then the surface is a sphere and $p^{0,2}_{t}=p_{t}(I_{N})$.

Let us use these multivariate heat kernels to finally give expressions of the normalisation constants $Z(\Sigma)$ and $Z(\Sigma,\ul{w})$ that appeared in \eqref{eq:WLE}, in Theorem \ref{thm:main}, and in the Driver--Sengupta formula \eqref{eq:DSformula}. Recall that $|\Sigma|$ denotes the area of the surface $\Sigma$ on which we work. 

If $\Sigma$ is closed, let us define 
\begin{equation}\label{eq:defZ}
Z(\Sigma)=p_{|\Sigma|}^{0,\e_{\Sigma}}.
\end{equation} 
If $\Sigma$ has a boundary, two cases must be distinguished. The first is when $\Sigma$ has at least one free boundary component. In this case, where $\k$ is strictly smaller than the number of boundary components of $\Sigma$,
\begin{equation}\label{eq:Z=1}
Z(\Sigma,\ul{w})=1.
\end{equation}
The second case is when every boundary component is constrained. In this case,
\begin{equation}\label{eq:defZwb}
Z(\Sigma,\ul{w})=p^{{\k},{\sf e}_{\Sigma}}_{|\Sigma|}(w_{1},\ldots,w_{\k}).
\end{equation}

\index{Z@$Z(\Sigma), Z(\Sigma,\ul{w})$, partition function}

\subsection{An integral expression}
We can now write the integral expression that will be the starting point of the proof of our main result. Let $\sG=(\sV,\sE,\sF)$ be the graph associated to the loop configuration $\ul\ell=(\ell_{1},\ldots,\ell_{m})$ that we consider on the surface $\Sigma$. 

Each of the immersions $\ell_{1},\ldots,\ell_{m}$ determines a loop in the graph $\sG$. We base each of them at a vertex, in an arbitrary way, and keep the notation $\ell_{1},\ldots,\ell_{m}$ for the based loops thus defined. Note that isolated embedded loops cannot be based at a vertex, but they do not need to: if some $\ell_{r}$ is an isolated embedded loop, then it is an edge of $\sG$, and the function $h_{\ell_{r}}(g)$ is well defined. 

Each face $F$ of $\sG$ also determines one or several loops, that bound $F$ positively. They are the images by the projection map $\bar \Sigma\to \Sigma$ of the boundary components of the compact surface~$\bar F$ associated to $F$.

If $\Sigma$ is closed, we denote the corresponding loops by $\partial_{1}F,\ldots,\partial_{\b_{F}}F$, and as before, we base each of them that is not an isolated embedded loop at a vertex in an arbitrary way. 
 
\index{dziF@$\partial_{i}F$, external boundary components of $F$}
 
If $\Sigma$ has a boundary, as discussed at the end of Section \ref{sec:graphskein}, some boundary components of~$\bar F$ belong to $\partial \Sigma$, and the others, called external, stem from the operation of blow up. Only the external boundary components of $F$ give rise to loops in $\sG$. Recall that we denote their number by $\x_{F}$, and let us denote the corresponding loops by $\partial_{1}F,\ldots,\partial_{\x_{F}}F$, as always based at an arbitrary vertex.

We will write an integral over $\U(N)^{\sE}$ with respect to the measure $\nsd g$ that is the product of the normalised Haar measures on each factor.

For the sake of the case where $\Sigma$ has a boundary, let us define $\sF^{\text{free}}$ as the subset of $\sF$ consisting of the faces of $\sG$ that contain at least one free boundary component of $\Sigma$. 

\index{Fa@$\sF^{\text{free}}$, faces containing a free b.\,c.}

\begin{proposition} \label{prop:expressiontocompute} If $\Sigma$ is closed, then the following equality holds:
\begin{equation}\label{eq:integraltocompute}
Z(\Sigma)\, {\bf E}\Big[\prod_{j=1}^{m} \Tr(H_{\ell_{j}})\Big]=\int_{\U(N)^{\sE}} \prod_{j=1}^{m}\Tr(h_{\ell_{j}}(g)) \prod_{F\in \sF} p_{|F|}^{\b_{F},\e_{F}}(h_{\partial_{1}\!F}(g),\ldots,h_{\partial_{{\b_{F}}}\! F}(g)) \d g.
\end{equation}
If $\Sigma$ has a boundary, with boundary conditions $w_{1},\ldots,w_{\k}$ imposed on the boundary components $C_{1},\ldots,C_{\k}$, then the following equality holds:
\begin{align}\label{eq:integraltocomputewb}
&Z(\Sigma,\ul{w})\, {\bf E}\Big[\prod_{j=1}^{m} \Tr(H_{\ell_{j}})\Big]\\
&\hspace{2cm} =\int_{\U(N)^{\sE}} \prod_{j=1}^{m}\Tr(h_{\ell_{j}}(g)) \prod_{F\in \sF\setminus \sF^{\text{free}}} p_{|F|}^{\b_{F},\e_{F}}\big(\{w_{i} : C_{i} \subset F\},\{h_{\partial_{i}\!F}(g) : i=1\ldots \x_{F}\}\big) \d g. \nonumber
\end{align}
\end{proposition}

In \eqref{eq:integraltocomputewb}, we took advantage of the symmetry of the multivariate heat kernel under permutation of its arguments, and wrote them without specifying any order. We could also have written, in \eqref{eq:integraltocompute}, the factor indexed by the face $F$ as $p^{\b_{F},\e_{F}}_{|F|}(\{h_{\partial_{i}F}(g) : i=1\ldots \b_{F}\})$.
Incidentally, the number of arguments of each heat kernel in \eqref{eq:integraltocomputewb} is the right one, because for a face $F$ that does not contain any free boundary component of $\Sigma$, the number $\b_{F}$ of boundary components of $\bar F$ is the sum of the number of constrained internal boundary components and the number $\x_{F}$ of external boundary components. 

\begin{proof} In order to apply the Driver--Sengupta formula, we need to refine the graph $\sG$ into a cellular graph. For this, let us start by adding a vertex on each circular edge of $\sG$, and a closed linear edge on each boundary component of $\Sigma$. Then, let us lift this graph to the surface $\bar \Sigma$ obtained by blowing up $\Sigma$ along $\sG$. In this way, each boundary component of $\bar \Sigma$ is endowed with a graph which covers exactly its boundary, and therefore is not cellular in general. 

Now, we claim that it is possible to refine the graph on each component of $\bar\Sigma$ to a cellular graph with one face. Indeed, let us consider a connected component $\bar F$ of $\bar \Sigma$ that has  $\b$ boundary components and genus $\g$, that is, Euler characteristic $\e=2-2\g-\b$. The~$\b$ boundary components of $\bar F$ are covered by $\b$ loops $l_{1},\ldots,l_{\b}$, and there are no edges in the interior of $\bar F$.

Consider now a $4\g+3\b$-gon and identify pairwise $4\g+2\b$ of its sides in the most classical way, illustrated on an example in Figure \ref{fig:pattern}, that produces a surface with $\b$ boundary components and Euler characteristic $\e$, therefore homeomorphic to $\bar F$, and that we identify with $\bar F$. Among the vertices of the polygon, some are sent on the boundary of $\bar F$, and all others are sent to one single vertex in the interior of the surface. The seams of the identification of the sides of the polygon produce $2\g$ closed linear edges $e_{1},f_{1},\ldots,e_{\g},f_{\g}$ based at this inner vertex and $\b$ open linear edges $c_{1},\ldots,c_{\b}$ joining this vertex to each of the boundary components. These edges, together with the edges already present on the boundary, produce a graph with one single face, with boundary $[e_{1},f_{1}]\ldots [e_{\g},f_{\g}]c_{1}l_{1}c_{1}^{-1}\ldots c_{\b}l_{\b}c_{\b}^{-1}$.

\begin{figure}[h!]
\begin{center}
\scalebox{0.78}{\includegraphics{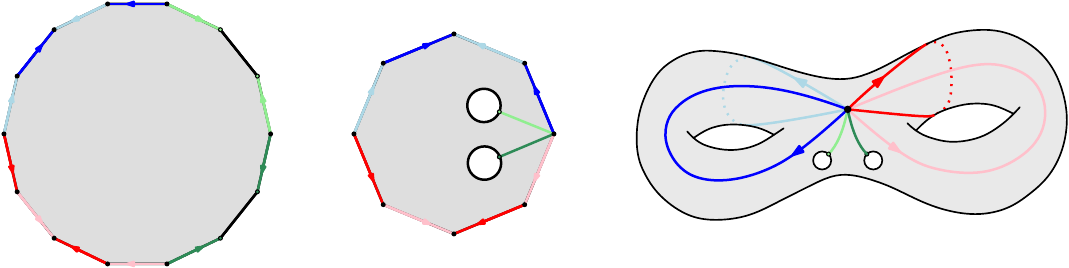}}
\caption{\label{fig:pattern}\small This figure illustrates a standard identification of a $14$-gon to produce a surface of genus $2$ with $2$ boundary components. The black edges are not to be identified and give rise to the boundary components.}
\end{center}
\end{figure}

Finally, let us glue back together the connected components of $\bar\Sigma$ along $\sG$, thereby producing a cellular graph $\bG$ on $\Sigma$ which has exactly one face in each face of $\sG$.

We can now apply the Driver--Sengupta formula \eqref{eq:DSformula} or \eqref{eq:DSformulawb} to the graph $\bG$. Whether $\Sigma$ has a boundary or not, the integrand contains one factor for each loop of the loop configuration, and one factor for each face of $\bG$, that is, for each face of $\sG$. 

Each free boundary component of $\Sigma$, if there is any, is covered by a closed linear edge of $\bG$ that appears exactly once, in the term corresponding to the face of $\sG$ that contains this component. Under the measure $\nsd_{\ul{w}}g$, the corresponding edge variable is distributed according to the Haar measure and independent of all other edge variables. Therefore, integrating with respect to this variable produces the integral of a multivariate heat kernel with respect to one of its arguments, that is equal to $1$, as we indicated shortly after \eqref{eq:defpmulti}. Thus, the corresponding factor disappears.

Under the measure $\nsd_{\ul{w}}g$, each loop that covers a constrained boundary component $C_{i}$ has a holonomy that belongs to the conjugacy class of $w_{i}$, and that is independent of all the other edge variables. Therefore, we can replace the holonomy along each constrained boundary component, which occurs exactly once in the Driver--Sengupta formula, by the corresponding boundary condition. 

This being done, there remains to integrate with respect to the edge variables corresponding to edges of $\bG$ that are in the interior of the faces of $\bG$. We chose the cellular graph on each connected component of $\bar \Sigma$ in such a way that, by definition of the multivariate heat kernel, this integration literally produces the factor 
\[p_{|F|}^{\b_{F},\e_{F}}(h_{\partial_{1}\!F}(g),\ldots,h_{\partial_{{\b_{F}}}\! F}(g)), \, \text{ or } \ p_{|F|}^{\b_{F},\e_{F}}\big(\{w_{i} : C_{i} \subset F\},\{h_{\partial_{i}\!F}(g) : i=1\ldots \x_{F}\}\big).\]
In fact, any graph with a single face would have produced the same number, but proving it would have required an extra argument. In any case, this establishes the formula. 
\end{proof}


\section{Step 2 --- Expansion of heat kernels}  \label{sec:expansionHK}
In this second step of the proof, we will transform the right-hand side of \eqref{eq:integraltocompute} using harmonic analysis on the unitary group, more precisely by expanding the heat kernel~$p_{t}$ and its multivariate analogues in Fourier, or Peter--Weyl series. 

The pieces of the theory of the representations and Fourier analysis on compact Lie groups that we need is described in \cite{Bump,FultonHarris,Faraut}. 

\subsection{Schur functions}\label{sec:SFFS}
Fourier series on $\U(N)$ are indexed by the isomorphism classes of irreducible representations (thereafter {\em irreps}) of $\U(N)$, which are, as we already mentioned in Section \ref{sec:CHW}, in one-to-one correspondence with the set $\HW$ of highest weights. 

The character of the irrep with highest weight $\lambda=(\lambda_{1}\geq\ldots \geq \lambda_{N})$ is the Schur function~$s_{\lambda}$, which on a unitary matrix $x$ with eigenvalues $\xi_{1},\ldots,\xi_{N}$ evaluates, according to the Weyl character formula, as
\begin{equation}\label{eq:defSchur}
s_{\lambda}(x)=\frac{\det \big[(\xi_{i}^{\lambda_{j}+N-j})_{i,j=1\ldots N}\big]}{\det \big[(\xi_{i}^{N-j})_{i,j=1\ldots N}\big]}.
\end{equation}
This expression is a symmetric polynomial function of $\xi_{1},\ldots,\xi_{N}$ which, when evaluated on the identity matrix, yields the dimension $d_{\lambda}$ of the irrep, given by the Weyl formula \eqref{eq:Weyldim}.

\index{sa@$s_{\lambda}$, Schur function}

We gave the expression \eqref{eq:defSchur} of Schur functions for reference because of its importance, but we will not use it in the proof, except to observe the following fact: if $\lambda$ is a highest weight of $\U(N)$, and $q\in \Z$ an integer, then 
\begin{equation}\label{eq:shiftSchur}
s_{\lambda+(q,\ldots,q)}={\det}^{q}s_{\lambda}.
\end{equation} 

Another property of the Schur functions that we will use is the following pair of integral identities, which are consequences of the orthogonality relations for characters of irreps of compact groups : for all $x,y\in \U(N)$ and every highest weight $\lambda$, 
\begin{equation}\label{eq:orthoint}
\int_{\U(N)^{2}}s_{\lambda}([u,v]x)\d u \nsd v=d_{\lambda}^{-2}s_{\lambda}(x) \ \text{ and } \ \int_{\U(N)}s_{\lambda}(xzyz^{-1})\d z =d_{\lambda}^{-1}s_{\lambda}(x)s_{\lambda}(y).
\end{equation}

\subsection{Fourier series of heat kernels} The invariant inner product on $\u(N)$ determines a bi-invariant Riemannian metric on $\U(N)$. The function $s_{\lambda}:\U(N)\to \C$ is an eigenfunction of the Laplace--Beltrami operator of this metric: with the definition of the quadratic Casimir number~$c_{\lambda}$ given by  \eqref{eq:casimir}, we have
\[\Delta s_{\lambda}=-\tfrac{1}{N}c_{\lambda} s_{\lambda}.\]
The Fourier series of the Dirac mass at $I_{N}$ is $\sum_{\lambda} d_{\lambda}s_{\lambda}$ and formally applying the operator $e^{\frac{t}{2}\Delta}$ to this series yields the Fourier expansion of the heat kernel 
\begin{equation}\label{eq:FourierHK}
p_{t}(x)=\sum_{\lambda\in \HW} e^{-\frac{t}{2N} c_{\lambda}} d_{\lambda}s_{\lambda}(x),
\end{equation}
which turns out to converge, for all $\epsilon>0$, uniformly in $x\in \U(N)$ and $t\geq \epsilon$ (this is a special case of \cite[Theorem 4.4]{Liao}).

Combining the expansion \eqref{eq:FourierHK}, the definition of the function $p_{t}^{\b,\e}$ given by \eqref{eq:defpmulti} and the integral relations \eqref{eq:orthoint}, we find that for all $\epsilon>0$, uniformly in $x_{1},\ldots,x_{\b}\in \U(N)$ and $t\geq \epsilon$, 
\begin{equation}\label{eq:FourierHKmulti}
p_{t}^{\b,\e}(x_{1},\ldots,x_{\b})=\sum_{\lambda\in \HW} e^{-\frac{t}{2N} c_{\lambda}} (d_{\lambda})^{\e}s_{\lambda}(x_{1})\ldots s_{\lambda}(x_{\b}).
\end{equation}

\subsection{Flat contribution of a configuration of highest weights} 

\index{F@$\mathscr F(\Lambda)$, flat contribution of $\Lambda$}

According to \eqref{eq:FourierHKmulti}, expanding each heat kernel that appears in the right-hand side of \eqref{eq:integraltocompute} produces a sum over all ways of choosing one highest weight for each face of the graph $\sG$, that is, over configurations of highest weights $\Lambda:\sF\to \HW$. In the case where $\Sigma$ has a boundary, the heat kernels corresponding to faces that contain at least one free boundary component are absent from the second product on the right-hand side of \eqref{eq:integraltocomputewb}. We reintroduce them in the form of the product
\begin{equation}\label{eq:1=1}
\prod_{F\in \sF^{\text{free}}} \prod_{i=1}^{\x_{F}} s_{(0,\ldots,0)}\big(h_{\partial_{i}F}(g)\big),
\end{equation}
which is equal to $1$, because the Schur function $s_{(0,\ldots,0)}$ is identically equal to $1$.

Recall that for each face $F$, the integer $\x_{F}$ is the number of external boundary components of $F$, those which arise from the loop configuration and not from the boundary of $\Sigma$. When $\Sigma$ is closed, all boundary components of a face are external, and $\x_{F}=\b_{F}$.

Let us define the {\em flat contribution} of a configuration of highest weights $\Lambda$ as
\begin{equation}\label{eq:I}
\mathscr F(\Lambda)=\prod_{F\in \sF} (d_{\lambda_{F}})^{\e_{F}}\int_{\U(N)^{\sE}} \Tr(h_{\ell}(g)) \prod_{F\in \sF} \prod_{i=1}^{\x_{F}} s_{\lambda_{F}}\big(h_{\partial_{i}\! F}(g)\big)\d g.
\end{equation}
We arrive at the following intermediate result, where in the case with boundary, we use the notion of configuration of highest weights vanishing at the free boundary, see Definition \ref{def:wellbalanced}.

\begin{proposition} \label{prop:expanded} If $\Sigma$ is closed, then the following equality holds:
\begin{equation}\label{eq:expanded}
Z(\Sigma)\, {\bf E}\Big[\prod_{j=1}^{m} \Tr(H_{\ell_{j}})\Big]=\sum_{\Lambda:\sF\to \HW} e^{-\frac{1}{2N} \underset{F\in \sF}{\sum} |F| c_{\lambda_{F}}} \mathscr F(\Lambda).
\end{equation}
If $\Sigma$ has a boundary, then 
\begin{equation}\label{eq:expandedwb}
Z(\Sigma, \ul{w})\, {\bf E}\Big[\prod_{j=1}^{m} \Tr(H_{\ell_{j}})\Big]=\hspace{-5mm}\sum_{\substack{\Lambda:\sF\to \HW\\ \Lambda \text{ vanishes}\\ \text{at the free boundary}}} \hspace{-5mm}
e^{-\frac{1}{2N} \underset{F\in \sF}{\sum} |F| c_{\lambda_{F}}} \prod_{i=1}^{\k} s_{\lambda_{F_{i}}}(w_{i})\,  \mathscr F(\Lambda).
\end{equation}
\end{proposition}

Our next goal, which will occupy us for most of what remains of the proof, is to compute the flat contribution of a configuration. A crucial point, that we will prove later, is that only balanced configurations (in the sense of Definition \ref{def:wellbalanced}) can have a non-zero flat contribution. For now, it is already possible to see without computing them that many configurations have equal flat contributions. 

\subsection{Non-negative configurations}
For every configuration of highest weights $\Lambda$, and every integer $q\in \Z$, let us define a new configuration $\Lambda+(q,\ldots,q)$ by setting, for each face $F$,
\[(\Lambda+(q,\ldots,q))(F)=\lambda_{F}+(q,\ldots,q).\]

\begin{lemma} \label{lem:translate} For all configuration of highest weights $\Lambda$ and all $q\in \Z$, we have
\[\mathscr F(\Lambda+(q,\ldots,q))=\mathscr F(\Lambda).\]
\end{lemma}

\begin{proof} In the integral \eqref{eq:I} defining the flat contribution of a configuration of highest weights, each Schur function is evaluated at a word in the integration variables $g_{1},\ldots,g_{r}$ and their inverses. The main observation is that each integration variable appears exactly twice in the union of these words, once as itself and once as its inverse.
 
Indeed, choose $j\in \{1,\ldots,r\}$ and consider the variable $g_{j}$, which corresponds to the edge $e_{j}$. This variable appears once as $g_{j}$ in a loop induced by the boundary of the face $\lf_{e_{j}}$, the face located on the left of $e_{j}$, and once as~$g_{j}^{-1}$ in a loop induced by the boundary of  $\rf_{e_{j}}$.

According to \eqref{eq:shiftSchur}, shifting all the highest weights of the configuration $\Lambda$ by $q$ multiplies each Schur function by the $q$-th power of the determinant of its argument. By the observation just made, the product of Schur functions in \eqref{eq:I} is the same function on $\U(N)^{r}$ whether we are computing $\mathscr F(\Lambda)$ or $\mathscr F(\Lambda+(q\ldots,q))$. 
\end{proof}

Let us call {\em non-negative} a highest weight $\lambda=(\lambda_{1}\geq \ldots \geq \lambda_{N})$ such that $\lambda_{N}\geq 0$, and {\em non-negative} a configuration of highest weights which to each face associates a non-negative weight.

It follows from Lemma \ref{lem:translate}, by taking $q$ large enough, that each configuration of highest weights has the same flat contribution as a non-negative configuration. It is thus sufficient for the time being to compute the flat contribution of non-negative configurations. This assumption will be lifted at the end of the proof, in Proposition \ref{prop:eqI6}.


\section{Step 3 --- Expression of Schur functions using permutations}  \label{sec:Schurusingperm}
In the computation of the flat contribution of a non-negative configuration of highest weights, we will use the Schur--Weyl correspondence between irreps of the unitary and symmetric groups to make symmetric quantities appear in the computation.  A general reference that contains the results about the Schur--Weyl duality that we need is \cite{GoodmanWallach}.

\subsection{Schur--Weyl duality}
Let us fix an integer $n\geq 1$. The unitary group $\U(N)$ and the symmetric group $\S_{n}$ act on the vector space $(\C^{N})^{\otimes n}$ by setting, for all $x\in \U(N)$, all $\sigma\in \S_{n}$, and all $v_{1},\ldots,v_{n}\in \C^{N}$,  
\begin{equation}\label{eq:action}
 x (v_{1}\otimes \ldots \otimes v_{n})=(xv_{1})\otimes \ldots \otimes (xv_{n}) \ \text{ and } \ \sigma (v_{1}\otimes \ldots \otimes v_{n})=v_{\sigma^{-1}(1)}\otimes \ldots \otimes v_{\sigma^{-1}(n)}.
\end{equation}
We choose not to name the maps from $\U(N)$ and $\S_{n}$ respectively to $\GL\big((\C^{N})^{\otimes n}\big)$, and to write the action of a unitary matrix or a permutation on a vector simply by the juxtaposition of symbols. There will be a few exceptions to this rule, and they will be duly indicated.

It follows from their definitions that the actions of $\U(N)$ and $\S_{n}$ just defined commute to each other. Therefore, the two sub-algebras of $\End((\C^{N})^{\otimes n})$ respectively generated by each action commute to each other. 
The core of Schur--Weyl duality is the much stronger fact that these two sub-algebras are each other's commutant. This has many implications, notably a close correspondence between the irreps of the symmetric and unitary groups, of which we will make extensive use.

In order to describe this correspondence, for every highest weight $\lambda$, let us denote by $V_{\lambda}$ a $\lambda$-irrep of $\U(N)$. The vector space $V_{\lambda}$ has dimension~$d_{\lambda}$ and the character of the action of $\U(N)$ on $V_{\lambda}$ is the Schur function $s_{\lambda}$.

Then, recall that the isomorphism classes of irreps of the symmetric group $\S_{n}$ are classically parametrised by the set of partitions of the integer $n$, see for example \cite{Sagan}. For each partition~$\lambda$ of $n$, let us denote by $V^{\lambda}$ a $\lambda$-irrep of $\S_{n}$. Let $d^{\lambda}$ denote the dimension of $V^{\lambda}$, and $\chi^{\lambda}$ the character of this representation. 
\index{Vl1@$V_{\lambda}$, irreducible $\U(N)$-module}
\index{Vl2@$V^{\lambda}$, irreducible $\S_{n}$-module}
\index{czhila@$\chi^{\lambda}$, character of an irreducible $\S_{n}$-module}

Note that we will consistently use subscripts for unitary irreps and superscripts for symmetric irreps. The unitary irreps $V_{\lambda}$ are indexed by highest weights $\lambda\in \HW$ and the symmetric irreps~$V^{\lambda}$ by partitions $\lambda\vdash n$. These are not the same sets, and their intersection is the set of partitions of the integer $n$ with at most $N$ parts, that we denote by $\Part_{n,N}$.

The essence of the Schur--Weyl duality is summarised in the existence of an isomorphism of $\U(N)\times \S_{n}$-modules
\begin{equation}\label{eq:SW}
(\C^{N})^{\otimes n} \simeq \bigoplus_{\lambda\in \Part_{n,N}} V_{\lambda}\otimes V^{\lambda}.
\end{equation}

For each partition $\lambda$ of $n$, let us introduce
\begin{equation}\label{eq:defpi}
\pi^{\lambda}=\frac{d^{\lambda}}{n!}\sum_{\sigma\in \S_{n}} \chi^{\lambda}(\sigma) \sigma,
\end{equation}
the element of the group algebra $\C[\S_{n}]$ that acts in every representation of $\S_{n}$ as the projector on the $\lambda$-isotypical component. We will make repeated use of the fact that the projectors $\pi^{\lambda}$ are mutually orthogonal idempotents of the group algebra $\C[\S_{n}]$, in the sense that for all partitions~$\lambda$ and $\mu$ of $n$,
\begin{equation}\label{eq:projorth}
\pi^{\lambda}\pi^{\mu}=\delta_{\lambda\mu} \pi^{\lambda}.
\end{equation}
Moreover, they belong to the centre of $\C[\S_{n}]$, which is to say that they commute with every element of~$\S_{n}$. Finally, because the characters of the symmetric group are real-valued, the coefficient of $\sigma$ and $\sigma^{-1}$ are equal in $\pi^{\lambda}$, for all partition $\lambda$ of $n$ and all $\sigma\in \S_{n}$.

\index{pzil@$\pi^{\lambda}$, isotypical symmetric projector}

Letting $\pi^{\lambda}$ act on $(\C^{N})^{\otimes n}$ projects on the $\lambda$-isotypical component for $\S_{n}$, that is, on $V_{\lambda}\otimes V^{\lambda}$, which is also the $\lambda$-isotypical component for $\U(N)$. For every $x\in \U(N)$, the trace of the action of $x$ on this isotypical component is equal to $d^{\lambda}s_{\lambda}(x)$, so that 
\begin{equation}\label{eq:schurastrace}
s_{\lambda}(x)=\frac{1}{d^{\lambda}}\Tr_{(\C^{N})^{\otimes n}}\big(x \pi^{\lambda}\big).
\end{equation}
To be clear, in this formula, $x\pi^{\lambda}$ denotes the product, in the space of endomorphisms of $(\C^{N})^{\otimes n}$, of the action of $x\in \U(N)$ and that of $\pi^{\lambda}\in \C[\S_{n}]$. 

The relation \eqref{eq:schurastrace} will be for us the main bridge between unitary and symmetric quantities, and will allow us to express the flat contribution of a configuration of highest weights in combinatorial terms. 

\subsection{The integrand as a trace}
Starting from a gauge configuration $g$ and a configuration of highest weights $\Lambda$ on the graph~$\sG$ induced by the loop configuration $\ul\ell$, we will construct three endomorphisms ${\sf I}(g)$, ${\sf \Pi}(\Lambda)$ and ${\sf X}(\Lambda)$ of a same vector space, that we will specify shortly, with the property that the trace of the product of these endomorphisms is equal, up to a mild correction, to the function of the gauge configuration $g$ of which the integral over the configuration space~$\U(N)^{r}$ defines the flat contribution of the highest weight configuration $\Lambda$ (see \eqref{eq:I} for the definition of the flat contribution and \eqref{eq:I2} for the expression to come).


Let us start by defining a family of vector spaces, namely, for all integers $n,m\geq 0$, the space
\[T^{n}_{m}=(\C^{N})^{\otimes n}\otimes \big((\C^{N})^{*}\big)^{\otimes m}.\]

Let us choose and fix a non-negative configuration of highest weights $\Lambda$. For each face $F$ of the graph $\sG$, we regard the non-negative highest weight $\lambda_{F}=(\lambda_{F,1}\geq \ldots \geq \lambda_{F,N}\geq 0)$ as a partition of the integer $n_{F}=\lambda_{F,1}+ \ldots + \lambda_{F,N}$. 

For each edge $e\in \sE$, recall that we denote respectively by $\lf_{e}$ and $\rf_{e}$ the faces of the graph located on the left and on the right of $e$. Let us define the integers $l_{e}=n_{\lf(e)}$ and $r_{e}=n_{\rf(e)}$, which depend on the configuration $\Lambda$. Let us also associate to the edge $e$ the vector space
\[\T_{e}=T^{l_{e}+1}_{r_{e}}=(\C^{N})^{\otimes l_{e}}\otimes \C^{N}\otimes \big((\C^{N})^{*}\big)^{\otimes r_{e}}.\]
In the product of three factors on the right-hand side, the first one corresponds, in a sense that will become clear in the following, to the face on the left of $e$, the middle one to the edge itself, and the right one to the face on the right of $e$.

\index{Te@$\T_{e}$, vector space associated to the edge $e$}
\index{re@$r_{e}$, sum of the components of $\lambda_{\rf_{e}}$}
\index{le@$l_{e}$, sum of the components of $\lambda_{\lf_{e}}$}

The vector space of which we are going to construct two (actually, three) endomorphisms is the tensor product $\bigotimes_{e\in \sE} \T_{e}$. 

To construct endomorphisms of this space, we will let the unitary and symmetric groups act on tensor powers of the dual space $(\C^{N})^{*}$. In the symmetric case, the contragredient of the action of $\S_{n}$ on $T_{0}^{n}$ is the action of $\S_{n}$ on $T_{n}^{0}$ given by setting, for all $\sigma \in \S_{n}$ and all $\phi_{1},\ldots,\phi_{n}\in \C^{N}$, 
\[\sigma^{\vee}(\phi_{1}\otimes \ldots \otimes \phi_{n})=\phi_{\sigma^{-1}(1)}\otimes \ldots \otimes \phi_{\sigma^{-1}(n)}.\] 
Since this action is given by the same formula as that defining the action of $\S_{n}$ on $T^{N}_{0}$, we will not use the notation $\sigma^{\vee}$ to distinguish these actions.

In the unitary case, for every $x\in \U(N)$, let us denote by $x^{\vee}={}^{t}(x^{-1})$ the image of $x$ by the contragredient of the natural representation of $\U(N)$. If $\psi\in T_{1}^{0}=(\C^{N})^{*}$, then by definition, $x^{\vee}\psi=\psi \circ x^{-1}$. In matricial terms, $x^{\vee}$ is the complex conjugate of $x$. The action of $\U(N)$ on~$T_{n}^{0}$ is simply the $n$-th tensor power of this action on $T_{1}^{0}$.

\index{xvee@$x^{\vee}$, complex conjugate of the unitary matrix $x$}

\subsection{The first two endomorphisms} For each edge $e\in \sE$, let us define 
\begin{equation}\label{eq:defIe}
{\sf I}_{e}(x)=x^{\otimes (l_{e}+1)}\otimes (x^{\vee})^{\otimes r_{e}} \in {\rm End}(\T_{e}).
\end{equation}

Given a gauge configuration $g=(g_{e})_{e\in \sE}$, let us define the first endomorphism by setting
\begin{equation}\label{eq:defI}
{\sf I}(g)=\bigotimes_{e\in \sE} {\sf I}_{e}(g_{e}) \ \in \ {\rm End}\big(\bigotimes_{e\in \sE} \T_{e}\big).
\end{equation}

Let us also define, for each edge $e\in \sE$,
\begin{equation}\label{eq:defPie}
{\sf \Pi}_{e}(\Lambda)= \pi^{\lambda_{\lf_{e}}}\otimes {\rm id}_{T^{1}_{0}} \otimes \pi^{\lambda_{\rf_{e}}} \ \in \ {\rm End}(\T_{e}),
\end{equation}
and set
\begin{equation}\label{eq:defPi}
{\sf \Pi}(\Lambda)=\bigotimes_{e\in \sE} {\sf \Pi}_{e}(\Lambda) \ \in \ {\rm End}\big(\bigotimes_{e\in \sE} \T_{e}\big).
\end{equation}

These definitions are quite simple, but some more complicated are coming, and it is a fact that computations involving large tensor products can quickly become impractical. A purely literal approach is possible, but especially for the problem at hand, which is $2$-dimensional in nature, it seems that a diagrammatic approach, which uses the two dimensions of the sheet of paper, allows for much clearer and easier calculations.

Let us take the definitions above as an opportunity to introduce the graphical representation of tensors, in the style of Penrose diagrams, that we will use in many of the computations to come. 

In these diagrams, a tensor is represented by a box, labelled by the name of the tensor, to which oriented strands attached, some incoming and some outgoing. Tensor product of tensors is represented by the mere juxtaposition of boxes, so that a diagram usually consists of a combination of several boxes. It can also contain strands without boxes.

A diagram is to be looked at up to local deformation : the exact shape of boxes and strands does not matter. For the sake of computational convenience, some strands are drawn above the box and some below, but this does not bring any additional information. 

The meaning of the strands depends on the situation, but for the time being, each strand represents a factor $\C^{N}$ or $(\C^{N})^{*}$. More specifically, an element of $T^{n}_{m}$ is represented by a box with~$m$ incoming and $n$ outgoing strands. As the special case of this rule when $m=n=0$, a box without any strand represents a scalar. Another special case is when $m=n=1$, and in this case the box represents an element of $T^{1}_{1}$, that is, an endomorphism of $\C^{N}$.

The last rule is that the contraction of a pair $\C^{N}\otimes (\C^{N})^{*}$ is represented by connecting the two corresponding strands, in a way that respects their orientation. For example, connecting the two strands of the diagram representing an endomorphism of $\C^{N}$ yields a diagram without any incoming or outgoing strand, that represents a scalar, the trace of the endomorphism.

We said that a diagram can contain strands without a box: each such strand stands for the canonical element of $T^{1}_{1}=(\C^{N})^{*}\otimes \C^{N}\simeq \End(\C^{N})$, namely the identity of $\C^{N}$. One could say that the box is implied, with the canonical label ${\rm id}_{\C^{N}}$. 

Before referring to Figure \ref{fig:edgetensor} for a graphical representation of the tensors ${\sf I}(g)$ and ${\sf \Pi}(\Lambda)$, let us make three more remarks. 

Firstly, for any vector space $V$, always finite-dimensional, we will frequently and tacitly make the identification $\End(V)\simeq V^{*}\otimes V$. Secondly, from the point of view of this identifications, there is no essential difference between an endomorphism $u$ of $V$ and the transposed endomorphism~${}^{t}u$ of $V^{*}$. Indeed, ${}^{t}u$ is obtained from $u$ by the chain of isomorphisms
\[\End(V)\simeq V^{*}\otimes V\simeq V\otimes V^{*} \simeq \End(V^{*}).\]
Therefore, there is no essential difference between a box labelled $x^{-1}$ for some $x\in \U(N)$, with one incoming and one outgoing strand, and the same box with the same strands labelled $x^{\vee}$. Only the perspective changes, as the first is rather seen as an endomorphism of $\C^{N}$, and the second as an endomorphism of $(\C^{N})^{*}$. 

Thirdly, we use the graphical simplification of writing a thick strand instead of a large number of strands. We will sometimes use several levels of thickness in the same picture, and will then say precisely which thickness corresponds to which number of strands. 

\begin{figure}[h!]
\begin{center}
\includegraphics{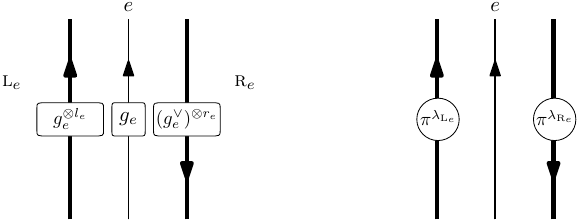}
\caption{\label{fig:edgetensor} \small On the left, a graphical representation of the endomorphism ${\sf I}_{e}(g)$ of $\T_{e}$. On the right, the factor $\pi^{\lambda_{\lf_{e}}}\otimes {\rm id}_{T^{1}_{0}} \otimes \pi^{\lambda_{\rf_{e}}}$ corresponding to the edge $e$ in ${\sf \Pi}(\Lambda)$. The arrows indicate the orientation of the edge $e$, and of the boundaries of the faces on the left and on the right of $e$. By analogy with the left part of the picture, one would expect to have, instead of the projector $\pi^{\lambda_{\rf_{e}}}$, its image by the endomorphism of $\C[\S_{r_{e}}]$ which sends each permutation to its inverse, but this endomorphism leaves $\pi^{\lambda_{\rf_{e}}}$ invariant.}
\end{center}
\end{figure}

These conventions being established, and before going further, let us make a comment about the definition \eqref{eq:defPi} of the endomorphism ${\sf \Pi}(\Lambda)$. 
In the definition \eqref{eq:defIe} of ${\sf I}_{e}(g)$, the rightmost strand carries the contragredient representation of $x^{\otimes r_{e}}$.
We explained a few lines above that this is the same object, only looked at from a different perspective, as the natural representation of $(x^{-1})^{\otimes r_{e}}$. It may therefore seem inconsistent with the orientation of the strands, as well as with the definition of ${\sf I}_{e}(g)$, that in the definition \eqref{eq:defIe} of ${\sf I}_{e}(g)$, the rightmost strand carries the action of $\pi^{\lambda_{\rf_{e}}}$ on $T_{0}^{r_{e}}$, instead of its image by the endomorphism of $\C[\S_{r_{e}}]$ which sends each permutation $\sigma$ to its inverse. However, according to a comment made shortly after~\eqref{eq:defpi}, the projectors $\pi^{\lambda}$ are invariant by this endomorphism, and it makes no difference to take one or the other.

\subsection{The third endomorphism}
We will now construct the endomorphism ${\sf X}(\Lambda)$, which depends on the configuration of highest weights $\Lambda$, but not on the gauge configuration $g$. It is a tensor product over all vertices $v\in \sV$ of a tensor ${\sf X}_{v}(\Lambda)$ that we will now define.

Let us consider a vertex $v$. It is an endpoint of four edges, that are not necessarily distinct, and that we denote, extending the metaphor of cardinal points, by $\cne_{v}$, $\cnw_{v}$, $\csw_{v}$ and $\cse_{v}$, as shown below in Figure~\ref{fig:cardinaledges}. The first two edges are outgoing, the last two are incoming, and they are incident at $v$ in this cyclic order. Recall that we already agreed to denote, in the natural way given the labelling of the edges, the four faces incident to $v$ by $\cn_{v}$, $\cw_{v}$, $\cs_{v}$ and $\ce_{v}$. For the sake of clarity, we will drop the subscript $v$ from the names of faces and edges.

We will define the tensor ${\sf X}_{v}(\Lambda)$ as an element of ${\rm Hom}(\T_{\csw}\otimes \T_{\cse},\T_{\cnw}\otimes \T_{\cne})$. For this, we will explain how this endomorphism acts on an element of
\[\T_{\csw}\otimes \T_{\cse}=(T^{n_{\cw}}_{0}\otimes T_{0}^{1}\otimes T^{0}_{n_{\cs}})\otimes (T^{n_{\cs}}_{0}\otimes T_{0}^{1}\otimes T^{0}_{n_{\ce}})
\]
and produces an element of 
\[\T_{\cnw}\otimes \T_{\cne}=(T^{n_{\cw}}_{0}\otimes T_{0}^{1}\otimes T^{0}_{n_{\cn}})\otimes (T^{n_{\cn}}_{0}\otimes T_{0}^{1}\otimes T^{0}_{n_{\ce}}).\]

A graphical representation of the action of ${\sf X}_{v}(\Lambda)$, in the style of Penrose diagrams introduced in the previous section, is given in Figure \ref{fig:cardinaledges} below. 

\begin{figure}[h!]
\begin{center}
\includegraphics{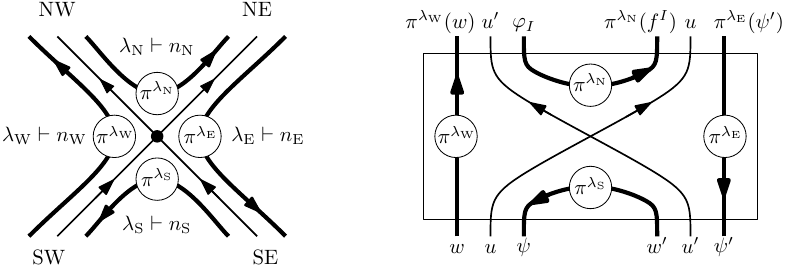}
\caption{\label{fig:cardinaledges} \small Two graphical representations of the endomorphism ${\sf X}_{v}(\Lambda)$.
}
\end{center}
\end{figure}

Let us give a more formal description. Let us choose arbitrary tensors \[w\in T^{n_{\cw}}_{0}, \ u\in T^{1}_{0}, \ \psi \in T^{0}_{ n_{\cs}}, \ w'\in T_{0}^{n_{\cs}}, \ u'\in T^{1}_{0}, \ \psi' \in T^{0}_{n_{\ce}}\]
and write the image of their tensor product by ${\sf X}_{v}(\Lambda)$. To do this, we will use the identification $T^{0}_{n_{\cn}}\otimes T^{n_{\cn}}_{0}\simeq \End(T^{n_{\cn}}_{0})=\End((\C^{N})^{\otimes n_{\cn}})$ and write ${\sf X}_{v}(\Lambda)(w\otimes u \otimes \psi \otimes w'\otimes  u' \otimes \psi')$ as an element of the space
\[T^{n_{\cw}}_{0}\otimes T_{0}^{1}\otimes \End(T^{n_{\cn}}_{0})\otimes T_{0}^{1}\otimes T^{0}_{n_{\ce}} \simeq \T_{\cnw}\otimes \T_{\cne}.\]
This being said, ${\sf X}_{v}(\Lambda)$ acts on $w\otimes u \otimes \psi \otimes w'\otimes  u' \otimes \psi'$ by doing four things:
\begin{enumerate}[\indent \sbullet]
\item apply $\pi^{\lambda_{\cw}}$ on $w$ and  $\pi^{\lambda_{\ce}}$ on $\psi'$,
\item exchange $u$ and $u'$,
\item apply $\pi^{\lambda_{\cs}}$ on $w'$ and evaluate $\psi$ on the result to produce a scalar,
\item replace $\psi\otimes w'$ by the endomorphism $\pi^{\lambda_{\cn}}$ of $T_{0}^{n_{\cn}}$.
\end{enumerate}
In one line, this can be written as
\[{\sf X}_{v}(\Lambda)(w\otimes u \otimes \psi \otimes w'\otimes  u' \otimes \psi')=\big(\psi,\pi^{\lambda_{\cs}}w'\big) \ (\pi^{\lambda_{\cw}}w)\otimes u' \otimes \pi^{\lambda_{\cn}} \otimes u \otimes (\pi^{\lambda_{\ce}}\psi').\]
Here, the first bracket on the right-hand side indicates the pairing of the linear form $\psi$ with the vector $\pi^{\lambda_{\cs}}w'$.

Let us give another, less compact, but perhaps more explicit expression of ${\sf X}_{v}(\Lambda)$. Let us denote by $(f^{1},\ldots,f^{N})$ the canonical basis of $T^{1}_{0}=\C^{N}$, by $(\varphi_{1},\ldots,\varphi_{N})$ the dual basis of $T^{0}_{1}$ and, for every multi-index $I\in \{1,\ldots,N\}^{n_{\cn}}$, by $f^{I}$ and $\varphi_{I}$ the corresponding elements of $T^{n_{\cn}}_{0}$ and $T_{n_{\cn}}^{0}$ respectively. Then, with the notation $\pi^{\lambda}(\sigma)$ for the coefficient of $\sigma$ in $\pi^{\lambda}$,
\begin{align*}
{\sf X}_{v}(\Lambda)(w\otimes u \otimes \psi \otimes w'\otimes  u' \otimes \psi')
&=\sum_{\sigma_{\cw},\sigma_{\cs},\sigma_{\ce}, \sigma_{\cn}, I } \pi^{\lambda_{\cw}}(\sigma_{\cw})\pi^{\lambda_{\cs}}(\sigma_{\cs})\pi^{\lambda_{\ce}}(\sigma_{\ce})\pi^{\lambda_{\cn}}(\sigma_{\cn})\\
&\hspace{1cm} \big(\psi,\sigma_{\cs}(w')\big)\ \ 
\sigma_{\cw} (w)\otimes u'\otimes  \varphi_{I} \otimes \sigma_{\cn}(f^{I}) \otimes u \otimes \sigma_{\ce} (\psi')
\end{align*}
where the sum runs over $(\sigma_{\cw},\sigma_{\cs},\sigma_{\ce}, \sigma_{\cn})\in \S_{n_{\cw}}\times \S_{n_{\cs}}\times \S_{n_{\ce}}\times \S_{n_{\cn}}$ and $I\in \{1,\ldots,N\}^{n_{\cn}}$.
  
Finally, we define 
\[{\sf X}(\Lambda)=\bigotimes_{v\in \sV} {\sf X}_{v}(\Lambda) \ \in \  \bigotimes_{e\in \sE}(\T_{e}^{*}\otimes \T_{e}) \simeq {\rm End}\big(\bigotimes_{e\in \sE} \T_{e}\big),\]
where the isomorphism is given by a mere reorganisation of the factors.

\subsection{Computation of the trace}
Let us now compute the trace $\Tr\big[{\sf X}(\Lambda){\sf \Pi}(\Lambda){\sf I}(g)\big]$ of the composition of the three endomorphisms just defined.

The composition ${\sf \Pi}(\Lambda){\sf I}(g)$ can be computed edge by edge, as
\[{\sf \Pi}(\Lambda){\sf I}(g)=\bigotimes_{e\in \sE} {\sf \Pi}_{e}(\Lambda)  {\sf I}_{e}(g),\]
that is, graphically, by piling the right-hand side of Figure \ref{fig:edgetensor} on top of the left-hand side of the same figure. Note that the order is unimportant, as the two endomorphisms of $\T_{e}$ commute.

Let us now compose ${\sf X}(\Lambda)$ with the resulting endomorphism of $\bigotimes_{e\in \sE} \T_{e}$. For this, we consider both ${\sf X}(\Lambda)$ and ${\sf \Pi}(\Lambda){\sf I}(g)$ as elements of $\bigotimes_{e\in \sE}(\T_{e}^{*}\otimes \T_{e}) $ and contract all pairs of factors that are dual spaces attached to the same edge, or more properly to the same half-edge, of the graph. 

At the initial point and at the final point of each linear edge $e$, there are three pairs of tensors to contract, with one factor stemming from ${\sf X}(\Lambda)$ and one from ${\sf \Pi}(\Lambda){\sf I}(g)$ in each pair. 

Circular edges yield a different, but simpler, contribution: for each circular edge $e$, the six strands of the tensor ${\sf \Pi}(\Lambda){\sf I}_{e}(g)$ are contracted pairwise.  

The result of these contractions is best understood graphically, as one follows the lines shown in Figures \ref{fig:edgetensor} and $\ref{fig:cardinaledges}$. By doing this, we arrive at the following product of traces.
\begin{enumerate}[\sbullet]
\item For each loop $\ell_{j}$ of the loop configuration, we obtain the trace of the discrete holonomy of the gauge configuration $g$ along this loop, that is, $\Tr(h_{\ell_{j}}(g))$. This is a trace on $\C^{N}$.
\item For each face $F$, and each loop $\partial_{i}F$ covering an external boundary component of $F$, we obtain the trace of the product of the edge variables of $g$ associated with the edges of $\partial_{i}F$, with an occurrence of~$\pi^{\lambda_{F}}$ between each pair of successive variables. Thus, if $\partial_{i} F=e_{j_{1}}^{\epsilon_{1}}\ldots e_{j_{k}}^{\epsilon_{k}}$, we get 
\[\Tr_{(\C^{N})^{\otimes n_{F}}}\big(g_{j_{k}}^{\epsilon_{k}} \pi^{\lambda_{F}}\ldots g_{j_{2}}^{\epsilon_{2}}\pi^{\lambda_{F}} g_{j_{1}}^{\epsilon_{1}}\pi^{\lambda_{F}}\big).\]
The presence of ${\sf \Pi}(\Lambda)$ ensures that this formula also holds for circular edges, although they do not meet a vertex. 

Now, since the unitary and symmetric actions on $(\C^{N})^{\otimes n_{F}}$ commute, and thanks to \eqref{eq:schurastrace}, this is equal to 
\[\Tr_{(\C^{N})^{\otimes n_{F}}}\big(g_{j_{k}}^{\epsilon_{k}} \ldots g_{j_{1}}^{\epsilon_{1}}\pi^{\lambda_{F}}\big)=\Tr_{(\C^{N})^{\otimes n_{F}}}\big(h_{\partial_{i} F}(g)\pi^{\lambda_{F}}\big)=d^{\lambda_{F}}s_{\lambda_{F}}\big(h_{\partial_{i} F}(g)\big).\]
\end{enumerate}

The outcome of this discussion is the equality
\[\Tr\big[{\sf X}(\Lambda){\sf \Pi}(\Lambda){\sf I}(g)\big]= \prod_{j=1}^{m}\Tr(h_{\ell_{j}}(g)) \prod_{F\in \sF} \prod_{i=1}^{\x_{F}} d^{\lambda_{F}}s_{\lambda_{F}}\big(h_{\partial_{i} F}(g)\big)\]
which, combined with the definition \eqref{eq:I} of $\mathscr F(\Lambda)$, yields
\begin{equation}\label{eq:I2}
\mathscr F(\Lambda)=\prod_{F} \frac{(d_{\lambda_{F}})^{\e_{F}}}{(d^{\lambda_{F}})^{\x_{F}}} \int_{\U(N)^{\sE}} \Tr\big[{\sf X}(\Lambda){\sf \Pi}(\Lambda){\sf I}(g)\big]\d g.
\end{equation}

Although we did not present it in this way, our rewriting of the integrand amounts exactly to recognising that it is, as a function of the gauge configuration $g$, a {\em spin network}. See \cite{Baez} for a general presentation of spin networks, and \cite{LevySN} for their use in a context related to the present work. The importance of this rewriting for us stems from the fact that it decouples the contributions of the edge variables, and will allow us to perform the integration. 

%

\begin{proposition}\label{prop:I3}
The flat contribution of a non-negative configuration of highest weights $\Lambda$ is given by
\begin{equation}\label{eq:I3}
\mathscr F(\Lambda)=\prod_{F\in \sF} \frac{(d_{\lambda_{F}})^{\e_{F}}}{(d^{\lambda_{F}})^{\x_{F}}} \ \Tr \bigg[\, {\sf X}(\Lambda)\circ  \bigotimes_{e\in \sE} \Big(\big(
\pi^{\lambda_{\lf_{e}}}\otimes {\rm id}_{T^{1}_{0}} \otimes \pi^{\lambda_{\rf_{e}}}
\big) \circ \int_{\U(N)} {\sf I}_{e}(x)\d x\Big)\bigg].
\end{equation}
\end{proposition}

\begin{proof} This expression follows from \eqref{eq:I2} by taking three tiny steps: replacing ${\sf \Pi}(\Lambda)$ by its definition, taking out of the integral the quantities that do not depend on the gauge configuration, and renaming the integration variable attached to each edge.
%
%
\end{proof}


\section{Step 4 --- Integration over unitary variables} \label{sec:integrationoverU}
We will now perform the unitary integration in  \eqref{eq:I3}, using the Collins--\'Sniady formula, see \cite{Collins,CollinsSniady}.

\subsection{The Collins--\'Sniady formula}\label{sec:CSF} According to \eqref{eq:defIe} and \eqref{eq:I3}, we need to know the value of the integral
\begin{equation}\label{eq:tocomputeCS}
\int_{\U(N)} x^{\otimes n}\otimes (x^{\vee})^{\otimes m}\d x \in \End(T^{n}_{m})
\end{equation}
for all non-negative integers $n$ and $m$. This value is given by the Collins--\'Sniady formula, which we explain in the form under which we will use it.

The first observation is that the invariance of the Haar measure under translations by scalar matrices implies that the integral vanishes if $n$ and $m$ are distinct. The Collins--\'Sniady formula gives the value of the integral in the case where $n=m$. 

Let us write this formula by showing the effect of the integral, seen as an endomorphism of~$T^{n}_{n}$, on an element of $T^{n}_{n}$ of the form $w\otimes \psi$, with $w\in T_{0}^{n}$ and $\psi \in T^{0}_{n}$. The result is an element of $T^{n}_{n}\simeq \End\big((\C^{N})^{\otimes n}\big)$, and this endomorphism will be written as a linear combination of elements of $\S_{n}$, which act on $(\C^{N})^{\otimes n}$ according \eqref{eq:action}. From this perspective, and with the definition \eqref{eq:defpi} of the projectors $\pi^{\lambda}$, the formula reads: 
\[\bigg[\int_{\U(N)} x^{\otimes n}\otimes (x^{\vee})^{\otimes n}\d x\bigg] (w\otimes \psi)=\sum_{\lambda\in \Part_{n,N}}\frac{d^{\lambda}}{d_{\lambda}}  \frac{1}{n!}\sum_{\sigma \in \S_{n}} \big( \psi, \sigma^{-1}(w)\big) \pi^{\lambda}\sigma.
\]
where $\Part_{n,N}$ is the set of partitions of the integer $n$ with at most $N$ parts, which already appeared in \eqref{eq:SW}. A graphical representation of this formula is given in Figure \ref{fig:CS} below. 

\begin{figure}[h!]
\begin{center}
\includegraphics{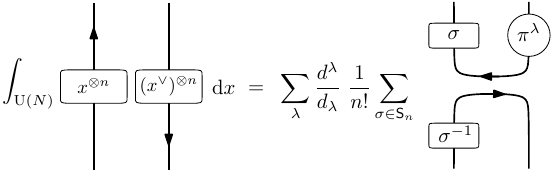}
\caption{\label{fig:CS} \small A graphical representation of the Collins--\'Sniady integration formula. The first sum runs over the set $\Part_{n,N}$ of partitions of $n$ with at most $N$ parts.}
\end{center}
\end{figure}

\subsection{Performing the integration}

Let us perform the integration in \eqref{eq:I3} using the Collins--\'Sniady formula as described in the previous paragraph. 

In \eqref{eq:I3}, we have one integration variable for each edge. Before actually applying the integration formula, the fact that the integral \eqref{eq:tocomputeCS} vanishes if $n\neq m$ gives us a weak form of the constraint of balance that a configuration of highest weights must satisfy in order to have a non-zero flat contribution. Indeed, for each edge $e$, the integers $l_{e}$ and $r_{e}$, which are respectively the sums of the components of the highest weights $\lambda_{\lf(e)}$ and $\lambda_{\rf(e)}$, must satisfy the relation
\begin{equation}\label{eq:weakbalance}
r_{e}=l_{e}+1.
\end{equation}
From now on, we will assume that this condition is satisfied for each edge $e$.

This condition being enforced, for the edge $e$, the integration formula, which applies with $n=r_{e}$, replaces the integral by a sum over a partition $\lambda\in \Part_{r_{e},N}$ and an average over a permutation $\sigma \in \S_{r_{e}}$. 

In view of \eqref{eq:projorth}, the presence of the factor $(\pi^{\lambda_{\lf_{e}}}\otimes {\rm id}_{T^{1}_{0}}) \otimes \pi^{\lambda_{\rf_{e}}}$ in front of the integral makes all terms of the first sum vanish, except for the one term corresponding to $\lambda=\lambda_{\rf_{e}}$. There remains a coefficient that is a quotient dimensions of irreps, respectively of $\U(N)$ and of $\S_{n}$, and the average over a permutation.

Carrying out the integration for all edges results, graphically, in the splitting of the diagram in the middle of each edge, that is, in the factorisation of the trace of \eqref{eq:I3} as a product of traces, one for each vertex, and one for each circular edge, which here again behave in a slightly exceptional, but simple way. 

Up to the quotients of dimensions and the averaging over permutations, the contribution of a vertex is represented graphically in Figure \ref{fig:around}.

\begin{figure}[h!]
\begin{center}
\includegraphics{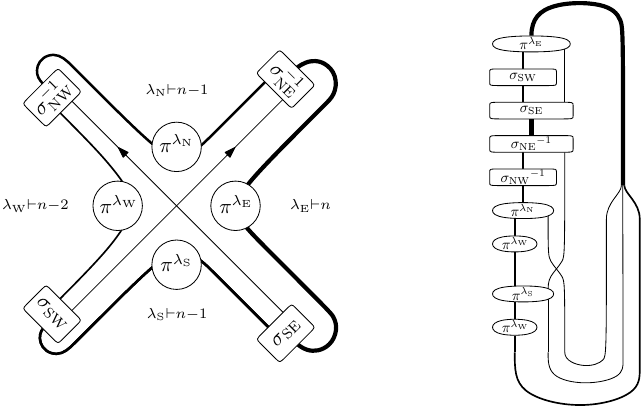}
\caption{\label{fig:around} \small Two representations of the contribution of a vertex after the Collins--\'Sniady formula has been applied to each integral in \eqref{eq:I3}. The first representation, on the left, in which we have set $n=n_{\ce}$, results from the application of the integration formula. The second representation, on the right, results directly  from the first by duplicating~$\pi^{\lambda_{\cw}}$, that is, replacing it by 
$(\pi^{\lambda_{\cw}})^{2}$, using the commutation properties of the various projectors $\pi^{\mu}$, and continuously deforming the diagram. There are four kinds of strands corresponding, from the thinnest to the thickest, to vectors of $T_{0}^{1}$, $T_{0}^{n-2}$, $T_{0}^{n-1}$, and $T_{0}^{n}$.}
\end{center}
\end{figure}

Let us focus on a vertex $v$ and set $n=n_{\ce_{v}}$. The number that is represented on the right of Figure \ref{fig:around} is the trace, in the space $T_{0}^{n}$, of a product of endomorphisms given by the actions of several elements of the group algebras of the groups $\S_{n-2}$, $\S_{n-1}$ and $\S_{n}$. It is understood that each of these groups is included in the next in the most natural way as the subgroup that leaves the greatest integer fixed.

All the elements of the group algebras of the symmetric groups are indicated with their names on the picture, except for one, that is the permutation, or indeed the transposition, stemming from the crossing of the two strands of the loop configuration at the vertex that we are focusing on. It is the transposition $(n-1 \, n)$. This being said, the number represented on the right of Figure \ref{fig:around} is
\begin{equation}\label{eq:contribvertex}
\Tr_{T^{n}_{0}}\big(\pi^{\lambda_{\ce}}\sigma_{\csw}\sigma_{\cse}\sigma_{\cne}^{-1}\sigma_{\cnw}^{-1}\ \   \pi^{\lambda_{\cn}}\pi^{\lambda_{\cw}}(n-1\, n) \pi^{\lambda_{\cs}}\pi^{\lambda_{\cw}}\big).
\end{equation}

We said that circular edges contribute in a special way: up to the quotient of dimensions and the averaging over the symmetric group $\S_{r_{e}}$, the contribution of a circular edge $e$ is  
\begin{equation}\label{eq:contribcircedge}
\sum_{\sigma\in \S_{r_{e}}} \Tr_{T^{r_{e}}_{0}}\big(\pi^{\lambda_{\rf_{e}}} \sigma \pi^{\lambda_{\lf_{e}}} \sigma^{-1}\big) = n_{r_{e}} ! \, \Tr_{T^{r_{e}}_{0}}\big(\pi^{\lambda_{\rf_{e}}}\pi^{\lambda_{\lf_{e}}}\big).
\end{equation}

Recall that we denote the set of circular edges by $\sCE$. Collecting the contribution of all vertices and all circular edges, including the coefficients that we left aside, we find the following expression.

\begin{proposition} \label{prop:I3.5}
The flat contribution of a non-negative configuration of highest weights $\Lambda$ satisfying the weak balance condition \eqref{eq:weakbalance} is given by
\begin{align}\label{eq:I3.5}
&\mathscr F(\Lambda)=\prod_{F\in \sF} \frac{(d_{\lambda_{F}})^{\e_{F}}}{(d^{\lambda_{F}})^{\x_{F}}} 
 \prod_{e\in \sE}\frac{d^{\lambda_{\rf_{e}}}}{d_{\lambda_{\rf_{e}}}} \frac{1}{n_{\rf_{e}}!}
 \prod_{e\in \sCE} n_{r_{e}}! \, \Tr_{T^{r_{e}}_{0}}\big(\pi^{\lambda_{\rf_{e}}}\pi^{\lambda_{\lf_{e}}}\big)
 \\
 &\hspace{4.5cm}\sum_{\ul\sigma}\prod_{v\in \sV}
 \Tr_{T^{n}_{0}}\big(\pi^{\lambda_{\ce}}\sigma_{\csw}\sigma_{\cse}\sigma_{\cne}^{-1}\sigma_{\cnw}^{-1}\ \   \pi^{\lambda_{\cn}}\pi^{\lambda_{\cw}}(n-1\, n) \pi^{\lambda_{\cs}}\pi^{\lambda_{\cw}}\big),
 \nonumber
\end{align}
where $\ul\sigma$ runs over $\prod_{e\in \sE} \S_{r_{e}}$ and it is understood, in the last product, that $n=n_{\ce_{v}}$ and the faces and edges are labelled relatively to the vertex $v$ under consideration. 
\end{proposition}


\section{Step 5 --- Introduction of modules over symmetric groups} \label{sec:symmetricmodules}

We started our computation with Schur functions, which are characters of irreducible modules over $\U(N)$. We used the Schur--Weyl duality to express these Schur functions using permutations. This led us to consider the $\U(N)\times \S_{n}$ module $(\C^{N})^{\otimes n}$ and more generally the spaces $T^{n}_{m}$. There, we could perform the integration over unitary variables, thereby getting rid of them. We are now going to switch completely to the symmetric side, and to continue, and eventually complete, our computation in modules over symmetric groups. 

This section is short and contains little computation. Its interest is to introduce the framework in which the end of the proof will take place. For a description of the representation theory of symmetric groups, the reader should consult \cite{FultonHarris, Macdonald,Sagan}.

\subsection{Symmetric modules and the branching rule}\label{sec:branching}
 In the course of the present computation, we consider the unitary group $\U(N)$ for a single, fixed, value of $N$. In contrast, we need to consider the symmetric groups $\S_{n}$ for many values of $n$. It is useful to consider these groups as the steps of an ascending chain of inclusions, as follows.

For each integer $n\geq 1$, the symmetric group $\S_{n}$ is the group of permutations of the set $\{1,\ldots,n\}$. For all integer $n\geq 2$, we identify $\S_{n-1}$ with the subgroup $\{\sigma \in \S_{n} :  \sigma(n)=n\}$ of~$\S_{n}$. We have therefore a chain of inclusions $\S_{1}\subset \ldots  \subset \S_{n}\subset \ldots$ Accordingly, every~$\S_{n}$-module, that is, every vector space on which $\S_{n}$ acts linearly, becomes an $\S_{k}$-module for every integer $k\leq n$.

We mentioned already the fact that the irreps of $\S_{n}$ are classically labelled by the integer partitions of $n$. Let us choose, once and for all, for each partition $\lambda$ of each integer $n\geq 1$, an actual instance $V^{\lambda}$ of the irrep of $\S_{n}$ labelled by the partition $\lambda$. 

Let us take $V^{\lambda}$ to be a real vector space, as is made possible by the fact that all the irreps of the symmetric groups are real. Moreover, let us endow $V^{\lambda}$ with an invariant Euclidean structure, that is, an inner product with respect to which $\S_{n}$ acts by isometries. This inner product is unique up to a multiplication by a positive scalar. 

Let $\mu$ be a partition of the integer $n$. The {\em branching rule}, stated below as Proposition \ref{prop:branching}, expresses how the irrep $V^{\mu}$ of $\S_{n}$ splits into irreps of $\S_{n-1}$. Recall the relation $\uparrow$ introduced in \eqref{eq:extends}. It was introduced for highest weights, but we use the same symbol for the same relation between partitions.

\begin{proposition} \label{prop:branching} Let $n\geq 2$ be an integer. Let $\mu$ be a partition of $n$. Then there is an isomorphism of $\S_{n-1}$-modules
\[V^{\mu}=\bigoplus_{\lambda\vdash n-1 , \,  \lambda\, \leads\, \mu} V^{\lambda}.\]
\end{proposition}

A crucial feature of this result is that $V^{\mu}$ is an $\S_{n-1}$-module without multiplicity. Therefore, the vector space $V^{\mu}$ splits in a canonical way as a direct sum of irreducible $\S_{n-1}$-submodules. To be clear, there is, for all partitions $\lambda\vdash n-1$ and $\mu\vdash n$ such that $\lambda\leads \mu$, a unique linear subspace of $V^{\mu}$ that is an irreducible $\S_{n-1}$-module isomorphic to $V^{\lambda}$. The space of $\S_{n-1}$-equivariant maps $V^{\lambda}\to V^{\mu}$ is therefore a line. The subset of this line formed by isometric maps has exactly two elements, which differ by a sign. 

We will now make a choice of one of these two maps, for all pair of partitions $\lambda\leads \mu$. For the moment, this choice will be arbitrary, but it will become apparent at a later stage that we want to make it in a way that fits a certain purpose. Explaining which purpose now and how it can be achieved would unduly interrupt the flow of the proof, so that we postpone this discussion until Section \ref{sec:detscal}. 

For the time being, let us choose, for all $n\geq 2$ and all partitions $\lambda \vdash n-1$ and $\mu \vdash n$ such that $\lambda\leads \mu$, an isometric $\S_{n-1}$-equivariant map $i_{\mu\lambda}:V^{\lambda}\to V^{\mu}$. Let us introduce the notation $i_{\lambda\mu}=i_{\mu\lambda}^{*}:V^{\mu}\to V^{\lambda}$ for the Euclidean adjoint of $i_{\mu\lambda}$, which is, up to a multiplicative constant, the unique isometric $\S_{n-1}$-equivariant map from $V^{\mu}$ to $V^{\lambda}$. See Figure \ref{fig:young} for a representation of these maps for $n\leq 5$.

In what follows, we will collectively call {\em isometric intertwiners} the maps $i_{\mu\lambda}$ and $i_{\lambda\mu}$ which we just defined.

\index{iml@$i_{\mu\lambda}, i_{\lambda\mu}$, isometric intertwiners}

\begin{figure}[h!]
\begin{center}
\includegraphics{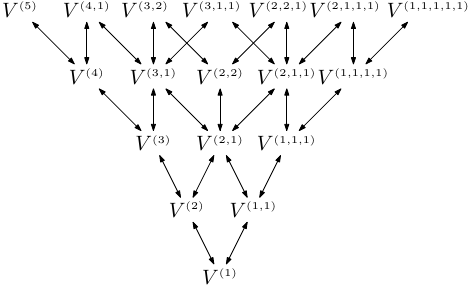}
\caption{\label{fig:young} \small This figure shows the first levels of the Young graph, with arrows corresponding to the isometric intertwiners defined above.}
\end{center}
\end{figure}

As an exception to our decision of not giving names to the representations of the various groups that we consider, let us denote, for all $n\geq 1$ and all $\mu\vdash n$, by
\[\rho_{\mu}:\S_{n}\to \GL(V^{\mu})\]
the action of $\S_{n}$ on $V^{\mu}$.

\index{rho@$\rho_{\mu}$, action of the symmetric group on $V^{\mu}$}

We will need the following pair of simple properties of the isometric intertwiners. 

\begin{lemma}
Consider two partitions $\lambda\leads \mu$, of $n-1$ and $n$ respectively. Then, for all $s\in \C[\S_{n-1}]$, 
\begin{equation}\label{eq:ident}
\rho_{\lambda}(s)=i_{\lambda\mu}\, \rho_{\mu}(s)\, i_{\mu\lambda} \ \ \text{ and } \ \ i_{\mu\lambda}\, \rho_{\lambda}(s)\, i_{\lambda\mu}=\rho_{\mu}(\pi^{\lambda}s)=\rho_{\mu}(s\pi^{\lambda}).
\end{equation}
\end{lemma}

\begin{proof} By definition of $i_{\lambda\mu}$ as the adjoint of $i_{\mu\lambda}$ and the fact that $i_{\mu\lambda}$ is injective, we have the equality $i_{\lambda\mu}i_{\mu\lambda}={\rm id}_{V^{\lambda}}$, which is the special of the first identity when $s={\rm id}$. For an arbitrary $s$, the $\S_{n-1}$-equivariance of the map $i_{\lambda\mu}$ yields
\[i_{\lambda\mu}\, \rho_{\mu}(s)\, i_{\mu\lambda}=\rho_{\lambda}(s)i_{\lambda\mu}i_{\mu\lambda}=\rho_{\lambda}(s).\]
Similarly, $i_{\mu\lambda}i_{\lambda\mu}$ is the orthogonal projection of $V^{\mu}$ onto its sub-$\S_{n-1}$-module isomorphic to $V^{\lambda}$. Thus, $i_{\mu\lambda}i_{\lambda\mu}=\rho_{\mu}(\pi^{\lambda})$. Then, the $\S_{n-1}$-equivariance of $i_{\mu\lambda}$ entails
\[i_{\mu\lambda}\, \rho_{\lambda}(s)\, i_{\lambda\mu}=\rho_{\mu}(s)i_{\mu\lambda}i_{\lambda\mu}=\rho_{\mu}(s)\rho_{\mu}(\pi^{\lambda})=\rho_{\mu}(s\pi^{\lambda}).\]
The fact that $\pi^{\lambda}$ is commutes to every element of $\S_{n-1}$ completes the proof.
\end{proof}

\subsection{The condition of balance} Recall that for all partition $\lambda$ of the integer $n$, we denote by~$\chi^{\lambda}$ the character of the $\lambda$-irrep of $\S_{n}$. Just as we deduced \eqref{eq:schurastrace} from \eqref{eq:SW}, we can deduce, from the same equation \eqref{eq:SW}, the fact that for all integer $n\geq 1$, all partition $\lambda$ of $n$, and all $\xi\in \C[\S_{n}]$, 
\begin{equation}\label{eq:charastrace}
\chi^{\lambda}(\xi )=\frac{1}{d_{\lambda}}\Tr_{(\C^{N})^{\otimes n}}\big( \pi^{\lambda} \xi\big).
\end{equation}
Therefore, the traces that appears in \eqref{eq:I3.5} are respectively equal to 
\[d_{\lambda_{\rf_{e}}}\chi^{\lambda_{\rf_{e}}}\big(\pi^{\lambda_{\lf_{e}}}\big)
\ \text{ and } \ 
d_{\lambda_{\ce}}\chi^{\lambda_{\ce}}\big(\sigma_{\csw}\sigma_{\cse}\sigma_{\cne}^{-1}\sigma_{\cnw}^{-1}\ \   \pi^{\lambda_{\cn}}\pi^{\lambda_{\cw}}(n-1\, n) \pi^{\lambda_{\cs}}\pi^{\lambda_{\cw}}\big).
\]

Since the moment where we performed the integration with respect to the unitary variables, and more specifically since the paragraph containing \eqref{eq:weakbalance}, we are assuming that, because $\lambda_{\ce}$ is a partition of $n$, $\lambda_{\cn}$, $\lambda_{\cs}$ and $\lambda_{\cw}$ must  respectively be partitions of $n-1$, $n-1$ and $n-2$. Otherwise, we know that the flat contribution of the configuration of highest weights that we are studying vanishes.

A further necessary condition of non-vanishing arises from the fact that the projector $\pi^{\lambda_{\lf_{e}}}$ must not vanish on the space of the representation $\lambda_{\rf_{e}}$ of $\S_{r_{e}}$, nor the products  $\pi^{\lambda_{\cn}}\pi^{\lambda_{\cw}}$ and $\pi^{\lambda_{\cs}}\pi^{\lambda_{\cw}}$ on the space of the representation $\lambda_{\ce}$ of $\S_{n}$. According to the branching rule for the irreps of the symmetric groups, stated above as Proposition \ref{prop:branching}, this is the case only if for each circular edge $e$, we have $\lambda_{\lf_{e}}\leads \lambda_{\rf_{e}}$, and at each vertex the following inclusions of Young diagrams hold:
\begin{equation}\label{eq:admissibility}
\raisebox{-1cm}{\includegraphics{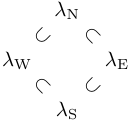}}
\end{equation}
or, with our notation, if and only if $\lambda_{\cw}\leads \lambda_{\cn}\leads \lambda_{\ce}$ and $\lambda_{\cw}\leads \lambda_{\cs}\leads \lambda_{\ce}$. These relations hold if and only if the configuration of highest weights satisfies the condition of balance that we introduced in Definition \ref{def:wellbalanced}. 

\begin{proposition} \label{prop:balance} If the configuration of highest weights $\Lambda$ is not balanced, then $\mathscr F(\Lambda)=0$.
\end{proposition}

\begin{proof} If $\Lambda$ is non-negative, this follows from the discussion preceding the statement of the proposition. Otherwise, choose $q\in \Z$ such that $\Lambda+(q,\ldots,q)$ is non-negative. Adding $(q,\ldots,q)$ does not alter the fact that the balance condition is not satisfied, so that, by Lemma \ref{lem:translate}, $\mathscr F(\Lambda)=\mathscr F(\Lambda+(q,\ldots,q))=0$.
\end{proof}

We arrive at the following mild reformulation of \eqref{eq:I3.5}.

\begin{proposition} \label{prop:I4} The flat contribution of a balanced non-negative configuration of highest weights $\Lambda$ is given by
\begin{align}\label{eq:I4}
&\mathscr F(\Lambda)=\prod_{F\in \sF} \frac{(d_{\lambda_{F}})^{\e_{F}}}{(d^{\lambda_{F}})^{\x_{F}}} 
\prod_{v\in \sV} d_{\lambda_{\ce_{v}}}
 \prod_{e\in \sE}\frac{d^{\lambda_{\rf_{e}}}}{d_{\lambda_{\rf_{e}}}} \frac{1}{n_{\rf_{e}}!}
\prod_{e\in \sCE} n_{r_{e}}! \, d_{\lambda_{\rf_{e}}} \chi^{\lambda_{\rf_{e}}}\big(\pi^{\lambda_{\lf_{e}}}\big) 
 \\
 &\hspace{5.3cm} \sum_{\ul\sigma}\prod_{v\in \sV}\chi^{\lambda_{\ce}}\big(\sigma_{\csw}\sigma_{\cse}\sigma_{\cne}^{-1}\sigma_{\cnw}^{-1}\ \   \pi^{\lambda_{\cn}}\pi^{\lambda_{\cw}}(n-1\, n) \pi^{\lambda_{\cs}}\pi^{\lambda_{\cw}}\big),
 \nonumber
\end{align}
where $\ul\sigma$ runs over $\prod_{e\in \sE} \S_{r_{e}}$ and it is understood, in the last product, that $n=n_{\ce_{v}}$ and the faces and edges are labelled relatively to the vertex $v$ under consideration. 
\end{proposition}


\section{Step 6 --- Summation over permutations}\label{sec:summingpermutations}
 We will now perform the summation with respect to the permutations attached to the edges, that is, compute the sum over $\ul\sigma$ in \eqref{eq:I4}. After we devoted so much work to make these permutations appear, this may seem to take us backwards. However, firstly, we will of course not go back to unitary quantities; and secondly, we will now reap the benefits of the fact that the strands of the original loop configuration $\ul\ell$ almost disappeared from our current expression of the flat contribution of a configuration of highest weights. Indeed, the only way in which they appear as such in \eqref{eq:I4} is through the transposition $(n-1\,  n)$, which is the computational reflection of the crossing of two strands at each vertex of the graph.

\subsection{A scalar endomorphism}
In order to sum over the permutations, it will be useful to change a little the way in which we write the interesting part, that is, the second half, of \eqref{eq:I4}. Let us start with a mild rewriting: since symmetric characters are central functions invariant by inversion, we can replace the last part of \eqref{eq:I4} by
\begin{equation}\label{eq:asommersigma}
\chi^{\lambda_{E}}\big(\sigma_{\cnw}\sigma_{\cne}\sigma_{\cse}^{-1}\sigma_{\csw}^{-1}\ \   \pi^{\lambda_{\cs}}\pi^{\lambda_{\cw}}(n-1\, n) \pi^{\lambda_{\cn}}\pi^{\lambda_{\cw}}\big).
\end{equation}

Then, for all quadruplet $(\lambda,\mu,\nu,\xi)$ of partitions such that $\lambda\leads \mu\leads \xi$ and $\lambda\leads\nu\leads \xi$, with $\xi\vdash n$, let us define
\begin{equation}\label{eq:defa}
a_{\lambda,\mu,\nu,\xi}= i_{\lambda\nu}i_{\nu\xi} \, (n-1\, n)\, i_{\xi\mu}i_{\mu\lambda}.
\end{equation}
\index{alam@$a_{\lambda,\mu,\nu,\xi}$, scalar endomorphism of $V^{\lambda}$}

\begin{lemma} \label{lem:ascalar} The endomorphism $a_{\lambda,\mu,\nu,\xi}$ is a multiple of the identity of $V^{\lambda}$.
\end{lemma}

\begin{proof} The endomorphism $a_{\lambda,\mu,\nu,\xi}$ is defined as the composition of the five maps
\[V^{\lambda} \build{\longrightarrow}{}{i_{\mu\lambda}} V^{\mu} \build{\longrightarrow}{}{i_{\xi\mu}} V^{\xi} \build{\longrightarrow}{}{(n-1\; n)} V^{\xi} \build{\longrightarrow}{}{i_{\nu\xi}} V^{\nu} \build{\longrightarrow}{}{i_{\lambda\nu}} V^{\lambda}.\]
The two leftmost maps and the two rightmost maps are morphisms of $\S_{n-2}$-modules by definition. Since $(n-1\, n)$ commutes with every element of the subgroup $\S_{n-2}$ of $\S_{n}$, the middle map is also a morphism of $\S_{n-2}$-modules. Therefore, by Schur's lemma, the endomorphism is a scalar.
\end{proof}

The last step of the proof will consist in the explicit computation of this scalar endomorphism, which will ultimately produce the sines and cosines that appear in our main result. For the time being, let us make use of all the newly introduced notation to go one step further in the computation of \eqref{eq:I4}. We assume that the configuration of highest weights of which we compute the flat contribution is balanced.

\begin{proposition} The following equality holds:
\begin{equation} \label{eq:octo}
\eqref{eq:asommersigma}=
\text{\rm Tr } \ \left[ \phantom{\rule{1pt}{1.8cm}}\right.  
\begin{tikzcd}[row sep=small, column sep=small, yshift=1mm]
 & [-15pt] V^{\lambda_{\cn}} \arrow[dl, "i_{\lambda_{\cw}\lambda_{\cn}}"',twoheadrightarrow,pos=0.45] &  V^{\lambda_{\cn}} \arrow[l, "\raisebox{3pt}{$\sigma_{\cnw}$}"']& [-15pt] \\
V^{\lambda_{\cw}} \arrow[d,"a_{\lambda_{\cw},\lambda_{\cn},\lambda_{\cs},\lambda_{\ce}}\;"'] 
&& & V^{\lambda_{\ce}} \arrow[ul,"i_{\lambda_{\cn}\lambda_{\ce}}"',twoheadrightarrow,pos=0.35] \\
V^{\lambda_{\cw}} \arrow[dr, "i_{\lambda_{\cs}\lambda_{\cw}}"' {yshift=2pt},hook]&&& V^{\lambda_{\ce}} \arrow[u,"\;\sigma_{\cne}\sigma_{\cse}^{-1}"'] \\
& V^{\lambda_{\cs}} \arrow[r, "\sigma_{\csw}^{-1}"' {yshift=-1pt}] &  V^{\lambda_{\cs}} \arrow[ur, "i_{\lambda_{\ce}\lambda_{\cs}}"',hook]&
\end{tikzcd}
\left. \phantom{\rule{1pt}{1.8cm}}\right]
\end{equation}
\end{proposition}

The endomorphism on the right-hand side is written so as to correspond to the structure of the graph around a vertex, in a way that will be made more explicit by Figure \ref{fig:avs} below. The arrows are drawn so as to suggest which intertwiners are injective and which are surjective.

\begin{proof} Set $n=n_{\ce_{v}}$. Recall that $\sigma_{\cnw}$ and $\sigma_{\csw}$ belong to $\S_{n-1}$, whereas $\sigma_{\cne}$ and $\sigma_{\cse}$ belong to $\S_{n}$.
Using twice the second identity of \eqref{eq:ident}, we find
\[\rho_{\lambda_{\ce}}(\pi^{\lambda_{\cn}}\sigma_{\cnw}\sigma_{\cne}\sigma_{\cse}^{-1}\sigma_{\csw}^{-1}\pi^{\lambda_{\cs}})=i_{\lambda_{\ce}\lambda_{\cn}}\, \Big[\rho_{\lambda_{\cn}}(\sigma_{\cnw})\  i_{\lambda_{\cn}\lambda_{\ce}}\  \rho_{\lambda_{\ce}}(\sigma_{\cne}\sigma_{\cse}^{-1})\  i_{\lambda_{\ce}\lambda_{\cs}}\  \rho_{\lambda_{\cs}}(\sigma_{\csw}^{-1})\Big]\, i_{\lambda_{\cs}\lambda_{\ce}},\]
where the five maps between the square brackets are exactly those which appear on the right half of the octagon in \eqref{eq:octo}, including the two horizontal sides.
 
Using the cyclicity of the trace, we deduce from the last equality that 
\[\eqref{eq:asommersigma}=\Tr_{V^{\lambda_{\cn}}}\Big(\Big[\cdots\Big]\, i_{\lambda_{\cs}\lambda_{\ce}} \, 
\rho_{\lambda_{\ce}}\big(\pi^{\lambda_{\cs}}\pi^{\lambda_{\cw}}(n-1\, n) \pi^{\lambda_{\cn}}\pi^{\lambda_{\cw}}\big)\, i_{\lambda_{\ce}\lambda_{\cn}}\Big),
\]
with the same linear map between the square brackets as in the previous equation. Using the second identity of \eqref{eq:ident} with $x=\pi^{\lambda_{\cw}}$, we find that $\rho_{\lambda_{\ce}}(\pi^{\lambda_{\cs}}\pi^{\lambda_{\cw}})=i_{\lambda_{\ce}\lambda_{\cs}}\, \rho_{\lambda_{\cs}}(\pi^{\lambda^{\cw}})\, i_{\lambda_{\cs}\lambda_{\ce}}$ and $\rho_{\lambda_{\ce}}(\pi^{\lambda_{\cn}}\pi^{\lambda_{\cw}})=i_{\lambda_{\ce}\lambda_{\cn}}\, \rho_{\lambda_{\cn}}(\pi^{\lambda^{\cw}})\, i_{\lambda_{\cn}\lambda_{\ce}}$.
The same identity with $x=1$ yields $\rho_{\lambda_{\cs}}(\pi^{\lambda^{\cw}})=i_{\lambda_{\cs}\lambda_{\cw}}i_{\lambda_{\cw}\lambda_{\cs}}$ and $\rho_{\lambda_{\cn}}(\pi^{\lambda^{\cw}})=i_{\lambda_{\cn}\lambda_{\cw}}i_{\lambda_{\cw}\lambda_{\cn}}$. Therefore, we have
\begin{align*}
i_{\lambda_{\cs}\lambda_{\ce}} \, 
\rho_{\lambda_{\ce}}\big(\pi^{\lambda_{\cs}}\pi^{\lambda_{\cw}}(n-1\, n) \pi^{\lambda_{\cn}}\pi^{\lambda_{\cw}}\big)\, i_{\lambda_{\ce}\lambda_{\cn}}&\\
&\hspace{-2cm} =i_{\lambda_{\cs}\lambda_{\ce}}
i_{\lambda_{\ce}\lambda_{\cs}}i_{\lambda_{\cs}\lambda_{\cw}}i_{\lambda_{\cw}\lambda_{\cs}}i_{\lambda_{\cs}\lambda_{\ce}}\, \rho_{\lambda_{\ce}}\big((n-1\, n)\big) \, i_{\lambda_{\ce}\lambda_{\cn}} i_{\lambda_{\cn}\lambda_{\cw}}i_{\lambda_{\cw}\lambda_{\cn}} i_{\lambda_{\cn}\lambda_{\ce}} i_{\lambda_{\ce}\lambda_{\cn}}\\
&\hspace{-2cm} =i_{\lambda_{\cs}\lambda_{\cw}}\, \Big(i_{\lambda_{\cw}\lambda_{\cs}}i_{\lambda_{\cs}\lambda_{\ce}}\, \rho_{\lambda_{\ce}}\big((n-1\, n)\big) \, i_{\lambda_{\ce}\lambda_{\cn}} i_{\lambda_{\cn}\lambda_{\cw}}\Big)\, i_{\lambda_{\cw}\lambda_{\cn}}. 
\end{align*}
The expression between the brackets is exactly $a_{\lambda_{\cw},\lambda_{\cn},\lambda_{\cs},\lambda_{\ce}}$ and find the three remaining sides of the octagon.  
\end{proof}

\subsection{Performing the summation}

To understand how this proposition helps us sum over~$\ul\sigma$, let us draw the endomorphism depicted on the right-hand side of \eqref{eq:octo} on the surface. This is done in Figure \ref{fig:avs} below.

\begin{figure}[h!]
\begin{center}
\includegraphics{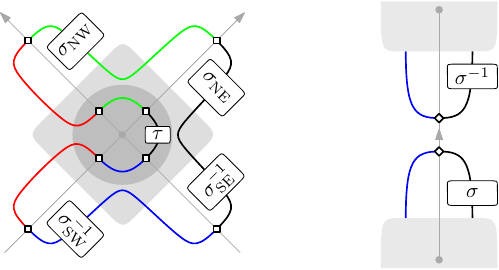}
\caption{\label{fig:avs} \small On the left, the linear map which appears in \eqref{eq:octo}. The strands of the loop configuration are there only for the sake of clarity, but they do not belong to the diagram. The colours of the strands indicates to which of the spaces $V^{\lambda_{\cw}}$ (red), $V^{\lambda_{\cn}}$ (green), $V^{\lambda_{\cs}}$ (blue), or $V^{\lambda_{\ce}}$ (black) the vectors belong. The labelled boxes indicate the action of a permutation, as usual. The small square boxes indicate the action of an isometric intertwiner. The grey rounded square separates graphically the contribution of the vertex from the contribution of the edges. The darker grey disk contains the map $a_{\lambda_{\cw},\lambda_{\cn},\lambda_{\cs},\lambda_{\ce}}$.\\
On the right, the contribution of an edge, consisting of two pieces of the contributions of each of its end points. 
}
\end{center}
\end{figure}

The right half of Figure \ref{fig:avs} shows that the summation over the permutations now takes place edge by edge. We will use the following formula, which is not unrelated to the Collins--\'Sniady formula, but is an  elementary consequence of the classical orthogonality relations: for all $\lambda\vdash n$ and all $v,w\in V^{\lambda}$, 
\begin{equation}\label{eq:symsum}
\frac{1}{n!}\sum_{\sigma\in \S_{n}} (\sigma v) \otimes (\sigma^{-1}w)=\frac{1}{d^{\lambda}} \, w\otimes v.
\end{equation}
Figure \ref{fig:symsum} presents a graphical representation of this formula, as well as the way in which it applies to the present situation.

\begin{figure}[h!]
\begin{center}
\includegraphics{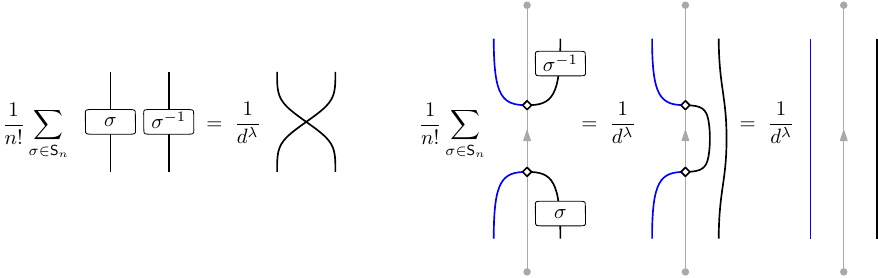}
\caption{\label{fig:symsum} \small A graphical representation of \eqref{eq:symsum}. On the right, the same formula, applied to the situation at hand. The partition $\lambda$ is the partition $\lambda_{\rf(e)}$, corresponding to the face located on the right of the edge under consideration.
}
\end{center}
\end{figure}

Let us now consider \eqref{eq:I4}. Recall that we denote by $\sLE$ the set of linear edges, so that $\sE$ is the disjoint union of $\sLE$ and $\sCE$, and let us reorganize the products over edges, thanks to some simplifications and the equality $\chi^{\lambda_{\rf_{e}}}\big(\pi^{\lambda_{\lf_{e}}}\big)=d^{\lambda_{\lf_{e}}}$, as
\begin{equation}\label{eq:reorga}
\prod_{e\in \sE}\frac{d^{\lambda_{\rf_{e}}}}{d_{\lambda_{\rf_{e}}}} \frac{1}{n_{\rf_{e}}!}
\prod_{e\in \sCE} n_{r_{e}}! \, d_{\lambda_{\rf_{e}}} \chi^{\lambda_{\rf_{e}}}\big(\pi^{\lambda_{\lf_{e}}}\big) 
=\prod_{e\in \sLE}\frac{d^{\lambda_{\rf_{e}}}}{d_{\lambda_{\rf_{e}}}} \frac{1}{n_{\rf_{e}}!}
\prod_{e\in \sCE}d^{\lambda_{\rf_{e}}}d^{\lambda_{\lf_{e}}}.
\end{equation}
Let us now sum over $\ul\sigma$. This 
\begin{enumerate}[\indent \sbullet]
\item absorbs the remaining factors $(n_{r_{e}}!)^{-1}$,
\item produces a factor $(d^{\lambda_{\rf_{e}}})^{-1}$ for each linear edge,
\item makes a loop appear for each external boundary component of each face that is not a circular edge, each loop contributing with the symmetric dimension $d^{\lambda_{F}}$ of the corresponding irrep.
\item Moreover, each loop passing by a vertex from the west side goes once around this vertex and this produces a scalar $a_{\lambda_{\cw},\lambda_{\cn},\lambda_{\cs},\lambda_{\ce}}$, for which we will use  the notation
\begin{equation}\label{eq:defa(v)}
a(v)=a_{\lambda_{\cw},\lambda_{\cn},\lambda_{\cs},\lambda_{\ce}}.
\end{equation}
\end{enumerate}

The symmetric dimensions produced by the circular edges in \eqref{eq:reorga} and by the loops in the third point of the enumeration above contribute together for each face $F$ a factor $(d^{\lambda_{F}})^{\x_{F}}$ which exactly cancels the denominator of the first product in \eqref{eq:I4}.

Similarly, the symmetric dimensions produced by the linear edges in the second point of the enumeration above cancel exactly the numerator of the first product in \eqref{eq:reorga}.

At this point, symmetric dimensions have completely disappeared and we obtain the following expression of the flat contribution of $\Lambda$: 
\begin{equation}\label{eq:I5avecLE}
\mathscr F(\Lambda)=\prod_{F\in \sF} (d_{\lambda_{F}})^{\e_{F}}
 \prod_{e\in \sLE}\frac{1}{d_{\lambda_{\rf_{e}}}} \prod_{v\in \sV} d_{\lambda_{\ce_{v}}} \prod_{v\in \sV} a(v).\end{equation}

We are only one small step away from the next proposition.

\begin{proposition} \label{prop:expI5} The flat contribution of a balanced non-negative configuration of highest weights $\Lambda$ is given by
\begin{equation}\label{eq:I5}
\mathscr F(\Lambda)=\prod_{F\in \sF} (d_{\lambda_{F}})^{\e_{F}}
\prod_{v\in \sV} (d_{\lambda_{\cn_{v}}}d_{\lambda_{\cs_{v}}})^{-\frac{1}{2}}
 \prod_{v\in \sV} a(v).
\end{equation}
\end{proposition}

\begin{proof} In \eqref{eq:I5avecLE}, let us consider the product over the set $\sLE$ of linear edges, split each factor as the square of $1/\sqrt{d_{\lambda_{\rf_{e}}}}$, and assign each factor to one of the two halves of the edge $e$, respectively the one located near its initial point and the one near its final point. 
Reorganizing this product over linear edges as a product over vertices, we find, around each vertex, a contribution of each adjacent edge, or rather half-edge, namely:
\begin{enumerate}[\indent\sbullet]
\item from the north-west edge a factor $1/\sqrt{d_{\lambda_{\cn}}}$, 
\item from the south-west edge a factor $1/\sqrt{d_{\lambda_{\cs}}}$,
\item and from the north-east and south-east edges, a factor $1/\sqrt{d_{\lambda_{\ce}}}$ each.
\end{enumerate}
The dimensions associated to the east face cancel with the product over vertices already present in \eqref{eq:I5avecLE}, and the result follows.
\end{proof}

There remains to compute the scalar $a(v)$, and this will be the object of the last step of the proof. Before we delve into it, let us review, in the last few steps of the computation, the influence of the choices that we made when we fixed the isometric intertwiners. Recall that they were specified up to a sign, each choice affecting simultaneously an intertwiner and its adjoint.

The expression \eqref{eq:asommersigma} did not depend on the intertwiners, and it is comforting that the right-hand side of \eqref{eq:octo} does not either, for each intertwiner appears twice, once in its original form and once in adjoint form. It may therefore come as an unpleasant surprise that the scalar $a(v)$ does depend on the choices that we made. However, a choice that affects the values of $a(v)$ for some vertex $v$ also affects the value of $a(w)$ for some adjacent vertex $w$, and the product of $a(v)$ over all vertices does not depend on the choices.

It is also natural to wonder if the scalar $a(v)$ depends on the Euclidean structure that we chose on each symmetric module. Fortunately, it does not. Indeed, changing the inner product in $V^{\mu}$, for instance, affects $i_{\mu\lambda}$ and $i_{\xi\mu}$ by inverse factors, and leaves $a(v)$ unaffected.


\section{Step 7 --- Determination of the scalar and completion of the proof} \label{sec:detscal}
The last substantial step of the proof is the computation of $a(v)$. In order to do this, we are going to lift the sign indeterminacy of the isometric intertwiners, using the approach to the representations of symmetric groups developed by Okounkov and Vershik \cite{OkounkovVershik}. 

\subsection{The geometric situation}\label{sec:geom}
For the convenience of the reader, let us state the problem again and describe it from the point of view of the geometry of the irreps of the symmetric groups. An integer $n\geq 3$ is given, and four partitions: $\lambda\vdash n-2$, $\mu\vdash n-1$, $\nu\vdash n-1$, and $\xi \vdash n$. We assume that $\lambda\leads \mu \leads \xi$ and $\lambda \leads \nu \leads \xi$. With the notation of Section \ref{sec:branching}, we want to compute the endomorphism
\[a_{\lambda,\mu,\nu,\xi}= i_{\lambda\nu}i_{\nu\xi} \, (n-1\, n)\, i_{\xi\mu}i_{\mu\lambda} \in \End(V^{\lambda}),\]
of which we proved that it is a multiple of the identity (see Lemma \ref{lem:ascalar}). 

Let us examine step by step the action of each of the applications that make up this endomorphism. Firstly, $i_{\mu\lambda}$ takes $V^{\lambda}$ and applies it into $V^{\mu}$, to the unique sub-$\S_{n-2}$-module of~$V^{\mu}$ isomorphic to $V^{\lambda}$ (this submodule exists and is unique by virtue of the branching rule, see Proposition \ref{prop:branching}). Then $i_{\xi\mu}$ takes $V^{\mu}$ and applies it into $V^{\xi}$, to the unique sub-$\S_{n-1}$-module of~$V^{\xi}$ isomorphic to $V^{\mu}$. There, the transposition $(n-1\, n)$ acts. Since this transposition commutes to $\S_{n-2}$, its action on $V^{\xi}$ leaves each $\S_{n-2}$-isotypical component stable. Therefore, it sends our original space $V^{\lambda}$, embedded in $V^{\xi}$ by $i_{\xi\mu}i_{\mu\lambda}$, into the sum of all sub-$\S_{n-2}$-modules of~$V^{\xi}$ isomorphic to $V^{\lambda}$.

To go further, we need to understand this isotypical component, the $\lambda$-isotypical component of $V^{\xi}$. In other words, we need to understand the multiplicity of $V^{\lambda}$ in $V^{\xi}$. Once again, the answer is given by the branching rule: the isotypical component is the direct sum of a number of copies of $V^{\lambda}$ equal to the number of partitions $\kappa\vdash n-1$ such that $\lambda\leads \kappa \leads \xi$. Since our assumptions rule out the possibility that it is $0$, this number can be $1$ or $2$. 

The case where this number is $1$ is the simplest. It is the case where the two boxes of the skew diagram $\xi/\lambda$ are adjacent, on the same row or on the same column. In this case, $(n-1\, n)$ acts as a scalar involution on $i_{\xi\mu}i_{\mu\lambda}V^{\lambda}$, that is, by the identity or minus the identity. Then, we must have $\nu=\mu$, and the linear maps $i_{\nu\xi}$ and $i_{\lambda\nu}$ take us back into $V^{\lambda}$. Finally, in this case, $a_{\lambda,\mu,\nu,\xi}=\pm{\rm id}_{V^{\lambda}}$ and we will see that the sign depends in a simple way on the shape of the skew diagram $\xi/\lambda$. Note that in this case, as in any case where $\nu=\mu$, the scalar $a_{\lambda,\mu,\nu,\xi}$ does not depend on the choice of isometric intertwiners.

The case where the multiplicity of $V^{\lambda}$ in $V^{\xi}$ is $2$ is more interesting. In this case, $\mu$ is one of the two partitions that fit between $\lambda$ and $\xi$ in the leading order. Let $\kappa$ be the other. Of course,~$\nu$ must be either $\mu$ or $\kappa$, and we will distinguish these two sub-cases (see also Figure \ref{fig:caspoura} below). For the moment, set $W=i_{\xi\mu}i_{\mu\lambda}V^{\lambda}$ and $W'=i_{\xi\kappa}i_{\kappa\lambda}V^{\mu}$, so that the $\lambda$-isotypical component of~$V^{\xi}$ is $W\oplus W'$.

\begin{figure}[h!]
\begin{center}
\includegraphics{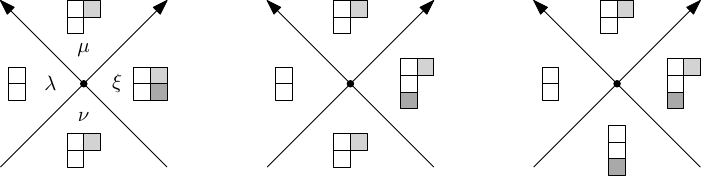}
\caption{\label{fig:caspoura} \small This figure illustrates the different situations for the partitions $\lambda,\mu,\nu,\xi$ with respect to the computation of the scalar $a(v)$. In the first case, $\mu$ and $\nu$ are determined by $\lambda$ and $\xi$, and must be equal. In the second case, given $\lambda$ and $\xi$, $\mu$ and $\nu$ could be different, but they are equal. In the third case, $\mu$ and $\nu$ are different. 
}
\end{center}
\end{figure}

On $W\oplus W'$, the transposition $(n-1\, n)$ acts as an isometric involution commuting to $\S_{n-2}$, the most general form of which is given blockwise by 
\[
\begin{blockarray}{ccc}
  &{\scriptstyle W} & {\scriptstyle W'} \\
\begin{block}{c(ll)}
    {\scriptstyle W}  & \cos \theta \  {\rm id}_{W}  & \phantom{-}\sin \theta \ j_{WW'}   \\
  {\scriptstyle W'} & \sin \theta\  j_{W'W}  & -\cos \theta\  {\rm id}_{W'}  \\
\end{block}
\end{blockarray}
\]
where $j_{W'W}:W\to W^{'}$ is an isometric isomorphism of $\S_{n-2}$-modules, $j_{WW'}$ its adjoint, and $\theta$ is a real number.

If $\nu=\mu$, then the last two operations $i_{\lambda\nu}$ and $i_{\nu\xi}$ that make up $a_{\lambda,\mu,\nu,\xi}$ take $W$ back to $V^{\lambda}$, and the scalar that we are looking for is $\cos \theta$. We will compute it, and it is entirely determined by $\lambda, \mu,\nu,\xi$.

In contrast, if $\nu\neq \mu$, then $i_{\lambda\nu}i_{\nu\xi}$ takes the submodule $W'$ back to $V^{\lambda}$, and the part of the matrix above that is relevant to us is the first entry of the second row, the endomorphism $\sin \theta \, j_{W'W}$. However, $j_{W'W}$ is determined only up to a sign. While what we said is enough to ensure that the scalar that we are looking for is $\pm \sin \theta$, we will need some care to determine which is the right sign.

\subsection{Jucys--Murphy elements and Gelfand--Zetlin lines} \label{sec:JMGZ} We will use a part of the approach of Okounkov--Vershik to the representations of symmetric groups, which we start by recalling. 

Consider a partition $\lambda$ of an integer $n$. According to the branching rule (see Proposition \ref{prop:branching}), the space $V^{\lambda}$ splits as a direct sum of subspaces, each of which is isomorphic, as a $\S_{n-1}$-module, to $V^{\lambda_{n-1}}$, where $\lambda_{n-1}$ is a partition of $n-1$ such that $\lambda_{n-1}\leads \lambda$. Iterating this application of the branching rule, we arrive at a splitting of $V^{\lambda}$ as a direct sum of irreducible $\S_{1}$-modules, which therefore must be lines. The set of these lines is in one-to-one correspondence with the set of sequences $\lambda_{1},\ldots,\lambda_{n-1}$ such that 
\[\lambda_{1}\leads \lambda_{2}\leads \ldots \leads \lambda_{n-1}\leads \lambda_{n}=\lambda.\]
A convenient way of recording such a sequence consists in writing, in each box of a Young diagram of shape $\lambda$, the index of the first partition of the sequence which contains this box. By doing this, one writes all the integers from $1$ to $n$ in the boxes of the Young diagram, in a way that is increasing down the columns and the rows. This is called a {\em standard tableau} of shape $\lambda$, and we will write $T\Vdash \lambda$ to indicate that $T$ is a standard tableau of shape $\lambda$.

We just explained how the branching rule implies that $V^{\lambda}$ splits in a canonical way as a direct sum of lines indexed by the set of standard tableaux of shape $\lambda$. These lines are called the {\em Gelfand--Zetlin lines} of $V^{\lambda}$. Let us mention that the Gelfand--Zetlin splitting is orthogonal with respect to the Euclidean structure of $V^{\lambda}$. 
\index{LT@$L_{T}$, Gelfand--Zetlin line}

Closely related to the Gelfand--Zetlin lines are the Jucys--Murphy elements. For all $k\geq 1$, let us define the $k$-th Jucys--Murphy element $X_{k}\in \C[\S_{k}]$ to be $0$ if $k=1$, and to be the sum of transpositions
\[X_{k}=(1\, k)+\ldots + (k-1 \, k)\]
for $k\geq 2$. We can see $X_{k}$ as belonging to $\C[\S_{n}]$ for any $n\geq k$. 

\index{Xk@$X_{k}$, Jucys--Murphy element}
 
Consider a tableau $T$ of shape $\lambda$ and choose $k\in \{1,\ldots,n\}$. There is a unique box of $T$ containing the integer $k$. If this box is located on row $i$ and column $j$, then we define the {\em content} of this box as the integer $j-i$, and write
\[\cont_{T}(k)=j-i.\]
Because this integer will play a crucial role for us, let us describe it in a way that is less dependent on the graphical representation of Young diagrams. For this, let $\lambda_{1}\leads \ldots \leads \lambda_{n}=\lambda$ be the increasing chain of partitions corresponding to $T$. For the sake of the case where $k=1$, set $\lambda_{0}=(0)$. Then for some integer $i\geq 1$,
\[\lambda_{k-1}=(a_{1},\ldots,a_{i},\ldots) \ \text{ and } \ \lambda_{k}=(a_{1},\ldots,a_{i}+1,\ldots).\]
Then, recalling the notation introduced in \eqref{eq:defcont},
\[\cont_{T}(k)=a_{i}+1-i=\cont(\lambda_{k}/\lambda_{k-1}).\]

It turns out that for all $n\geq 1$, the Jucys--Murphy elements $X_{1},\ldots,X_{n}$ generate a maximal abelian subalgebra of $\C[\S_{n}]$, and that this subalgebra is the set of elements of $\C[\S_{n}]$ which, in each irreducible $\S_{n}$-module, stabilize every Gelfand--Zetlin line. 

Here is a precise statement of the part of this information that we need.

\begin{proposition} \label{prop:caraclambda} Let $W$ be an $\S_{n}$-module. 
\begin{enumerate}[1.]
\item Suppose that $W$ is isomorphic to $V^{\lambda}$. Let  $W=\bigoplus_{T\Vdash \lambda} L_{T}$
be the Gelfand--Zetlin splitting of~$W$.  Then for every tableau $T$ and every $k\in \{1,\ldots,n\}$, the Jucys--Murphy element $X_{k}$ acts on the line $L_{T}$ by multiplication by $\cont_{T}(k)$.
\item Suppose that there exists a splitting in lines $W=\bigoplus_{T\Vdash \lambda} L_{T}$
indexed by standard tableaux of shape $\lambda$ such that for every tableau $T$ and every $k\in \{1,\ldots,n\}$, the Jucys--Murphy element~$X_{k}$ acts on the line $L_{T}$ by multiplication by $\cont_{T}(k)$. Then $W$ is isomorphic to $V^{\lambda}$, and the splitting is the Gelfand--Zetlin splitting.
\end{enumerate}
\end{proposition}

\subsection{Gelfand--Zetlin bases and isometric intertwiners}
Consider a partition $\lambda$ of $n$. The Gelfand--Zetlin lines of $V^{\lambda}$ are orthonormal, and almost yield an orthonormal basis of $V^{\lambda}$, that is, up to multiplication of each vector by $\pm 1$. We will go one step further and specify an almost canonical basis of $V^{\lambda}$, unique up to global multiplication by a sign.

Let $T$ be a standard tableau of shape $\lambda$. Let $k\in \{2,\ldots,n\}$ be an integer. According to Lemma~\ref{lem:contpasegal}, the integers $\cont_{T}(k)$ and $\cont_{T}(k-1)$ are not equal, and we define an angle $\theta_{T,k}\in [0,\pi]$ without ambiguity by the relation 
\[\cos (\theta_{T,k})=\frac{1}{\cont_{T}(k)-\cont_{T}(k-1)}.\]

Let us denote by $(k-1\, k)T$ the tableau obtained from $T$ by exchanging the integers $k-1$ and $k$. An elementary verification shows that $(k-1\, k)T$ is still a standard tableau, unless $k-1$ and $k$ are in neighbouring boxes of $T$, on the same row or on the same column. This in turns happens if and only if $|\cont_{T}(k)-\cont_{T}(k-1)|=1$. Therefore, either $|\cont_{T}(k)-\cont_{T}(k-1)|=1$, or $(k-1\, k)T$ is a standard tableau.

In the following proposition, we will consider vectors $v_{T}$ indexed by standard tableaux, and we will write $v_{(k-1\, k)T}$ although, according to what we just said, $(k-1\, k)T$ is not always a standard tableau. Let us agree that in this case, $v_{(k-1\, k)T}$ denotes the zero vector.

\index{vT@$v_{T}$, Gelfand--Zetlin vector}

\begin{proposition} \label{prop:GZbasis} There exists an orthonormal basis $(v_{T}: T \Vdash \lambda)$ of $V^{\lambda}$ indexed by standard tableaux of shape $\lambda$ such that for every tableau $T$ and every integer $k\in \{2,\ldots,n\}$, 
\begin{equation}\label{eq:defaction}
(k-1 \, k) v_{T}=\cos (\theta_{T,k})\  v_{T}+\sin( \theta_{T,k}) \ v_{(k-1\, k)T}.
\end{equation}
The basis $(v_{T}: T \Vdash \lambda)$ is unique up to global multiplication by a scalar of modulus $1$. Moreover, for every $T$, the vector $v_{T}$ belongs to the Gelfand--Zetlin line $L_{T}$.
\end{proposition}

In the case where $c_{T}(k)-c_{T}(k-1)=\pm 1$, we decided that $v_{(k-1\, k)T}=0$, but $\sin (\theta_{T,k})=0$ anyway, and the equality \eqref{eq:defaction} reduces to $(k-1 \, k) v_{T}=\cos (\theta_{T,k}) \ v_{T}=\pm v_{T}$. 

Although we will not really use this remark, let us point out that an orthonormal basis $(v_{T}: T \Vdash \lambda)$ of $V^{\lambda}$ such that each vector $v_{T}$ belongs to the Gelfand--Zetlin line $L_{T}$ satisfies \eqref{eq:defaction} if and only if 
\begin{equation}\label{eq:scalpos}
\forall T \Vdash \lambda, \ \forall \tau=(k-1\, k) \text{ with } k\in \{2,\ldots,n\}, \ \ \langle \tau v_{T},v_{\tau T}\rangle_{V^{\lambda}} \geq 0.
\end{equation}

The substance of Proposition \ref{prop:GZbasis} is already present in the work \cite{OkounkovVershik} of Okounkov and Vershik, and the proof is inspired by their presentation.

\begin{proof} For every $k\in \{2,\ldots,n\}$, let us define $\tau_{k}=(k-1\, k)$. The group $\S_{n}$ is generated by $\tau_{2},\ldots,\tau_{n}$ with the Coxeter relations $\tau_{k}^{2}=1$ for all $k\in \{2,\ldots,n\}$, $(\tau_{k}\tau_{l})^{2}=1$ for all $k,l\in \{2,\ldots,n\}$ such that $|k-l|\geq 2$, and $(\tau_{k}\tau_{k+1})^{3}=1$ for all $k\in \{2,\ldots,n-1\}$.

Let $W$ be a vector space with basis $(v_{T}:T\Vdash \lambda)$. For each $k\in \{2,\ldots,n\}$ 
let $s_{k}$ be the linear transformation of $W$ which for every tableau $T$ sends $v_{T}$ to the right-hand side of \eqref{eq:defaction}, with the convention that $v_{\tau_{k}T}=0$ if $\tau_{k}T$ is not a standard tableau.

To show that the map $\tau_{k}\mapsto s_{k}$ extends to a linear action of $\S_{n}$ on $W$, it suffices to show that the operators $s_{2},\ldots,s_{k}$ satisfy the Coxeter relations in $\GL(W)$. This is a somewhat tedious, but completely elementary computation, which we will spare the reader.

Once this is done, we have a linear representation of $\S_{n}$ on $W$, which satisfies \eqref{eq:defaction}. Let us endow $W$ with the inner product for which the basis $(v_{T}:T\Vdash \lambda)$ is orthonormal. The linear transformations $s_{2},\ldots,s_{k}$ are unitary, and so is the action of $\S_{n}$.

In order to prove that $W$ is a $\lambda$-irrep of $\S^{n}$, we use the characterization given by Proposition~\ref{prop:caraclambda}. For this, we compute the action of the Jucys--Murphy elements of $\S_{n}$ on the vectors~$v_{T}$. The key to this computation is the fact that for all $k\in \{2,\ldots,n\}$, the transposition $\tau_{k}$ and the Jucys--Murphy elements $X_{k-1}$ and $X_{k}$ are related by the equality
\[X_{k}=\tau_{k}X_{k-1}\tau_{k}+\tau_{k}.\]
From this equality, the fact that $X_{1}=0$ and the relation \eqref{eq:defaction}, one checks by induction on $k$ that for every standard tableau $T$, the relation $X_{k}v_{T}=\cont_{T}(k)v_{T}$ holds. 

Finally, we have constructed a $\lambda$-irrep $W$ of $\S_{n}$ with the desired properties. We can now transport the basis $(v_{T}:T\Vdash \lambda)$ from $W$ to $V^{\lambda}$ by a unitary isomorphism of $\S_{n}$-modules.

If two bases of $V^{\lambda}$ satisfy \eqref{eq:defaction}, then the linear isomorphism of $V^{\lambda}$ which sends one to the other is a unitary isomorphism of $\S_{n}$-module, thus, by Schur's lemma, a scalar of modulus $1$. 
\end{proof}

For each partition $\lambda$, let us endow $V^{\lambda}$ with one of the two bases specified by Proposition \ref{prop:GZbasis}.
We will now use these bases to make a choice of isometric intertwiners. 

Let us consider two partitions $\lambda\vdash n-1$ and $\mu\vdash n$ such that $\lambda \leads \mu$. We want to define $i_{\mu\lambda}:V^{\lambda}\to V^{\mu}$ and we do this by specifying the image by $i_{\mu\lambda}$ of the vector $v_{T}$ for each standard tableau $T$ of shape $\lambda$. We do this in the simplest possible way given the situation : to every standard tableau $T\Vdash \lambda$, we associate the standard tableau $S(T,\mu)$ of shape $\mu$ obtained by adding to $T$ the unique box of $\mu$ that is not in $\lambda$, and filling it with the integer $n$. Then, we set
\[i_{\mu\lambda}(v_{T})=v_{S(T,\mu)}.\]
This map $i_{\mu\lambda}$ thus defined sends an orthonormal basis of $V^{\lambda}$ to an orthonormal family of $V^{\mu}$. Moreover, in view of \eqref{eq:defaction}, it commutes with the action of the Coxeter generators of $\S_{n-1}$, so that it is $\S_{n-1}$-equivariant. In other words, $i_{\mu\lambda}$ is an isometric intertwiner.

Let us describe the adjoint of $i_{\mu\lambda}$. Consider a standard tableau $S$ of shape $\mu$. Then, for every tableau $T$ of shape $\lambda$, 
\[\langle i_{\lambda\mu}(v_{S}),v_{T}\rangle_{V^{\lambda}}=\langle v_{S},i_{\mu\lambda}(v_{T})\rangle_{V^{\mu}}=\langle v_{S},v_{S(T,\mu)}\rangle_{V^{\mu}}=\delta_{S,S(T,\mu)}.\]
This scalar product is equal to $1$ if and only if in the tableau $S$, the integer $n$ is in the box $\mu/\lambda$, and $T$ is the tableau obtained from $S$ by removing this box. 

Let us define $T(S)$ as the tableau obtained by removing from $S$ the box containing $n$. This tableau $T(S)$ may or may not be of shape $\lambda$. Then 
\[i_{\lambda\mu}(v_{S})=\left\{\!\!\begin{array}{ll} v_{T(S)} &\text{if } T(S) \text{ is of shape } \lambda, \\ \! 0 &\text{otherwise.}
\end{array}\right.\]

\subsection{Computation of the scalar}
A good choice of isometric intertwiners being made, we can finally compute the scalar $a(v)$ that appeared in \eqref{eq:I5}.

\begin{proposition} \label{prop:determinea} Let $n\geq 2$ be an integer. Consider four partitions 
$\lambda \vdash n-2$, $\mu,\nu \vdash n-1$, and $\xi \vdash n$ such that $\lambda \leads \mu \leads \xi$ and $\lambda \leads \nu \leads \xi$. Set $\cont_{1}=\cont(\xi/\mu)$ and $\cont_{2}=\cont(\mu/\lambda)$.
Let $\theta$ be the unique element of $[0,\pi]$ such that $\cos \theta=(\cont_{1}-\cont_{2})^{-1}$.
Then the endomorphism $a_{\lambda,\mu,\nu,\xi}$ defined by \eqref{eq:defa} is equal to 
\[a_{\lambda,\mu,\nu,\xi}=\Big\{\! \begin{array}{ll} \hspace{-0.8pt}\cos \theta\  {\rm id}_{V^{\lambda}} &\text{if } \mu=\nu,\\
\sin \theta\  {\rm id}_{V^{\lambda}} & \text{if } \mu\neq\nu.
\end{array} \]
\end{proposition}

\begin{proof} For the sake of clarity, let us set $a=a_{\lambda,\mu,\nu,\xi}$. Let us choose a tableau $T$ of shape $\lambda$ and compute the action of $a$ on the vector $v_{T}$ of $V^{\lambda}$. 

Let us define the tableaux $T_{n-1}=S(T,\mu)$ of shape $\mu$ and $T_{n}=S(T_{n-1},\xi)$ of shape $\xi$. The tableau $T_{n}$ contains $n$ in the box $\xi/\mu$ and $n-1$ in the box $\mu/\lambda$, so that $\cont_{T_{n}}(n)=\cont_{1}$ and $\cont_{T_{n}}(n-1)=\cont_{2}$. Therefore, $\theta_{T_{n},n}=\theta$. 

Our definition of the isometric intertwiners gives $i_{\mu\lambda} v_{T}=v_{T_{n-1}}$, then $i_{\xi\mu}i_{\mu\lambda} v_{T}=v_{T_{n}}$, and
\[(n-1\, n)\, i_{\xi\mu}i_{\mu\lambda} v_{T}=\cos \theta\, v_{T_{n}}+\sin \theta \, v_{(n-1\, n)T_{n}}.\]

Let us treat first the case where $\nu=\mu$. In this case, we claim that $i_{\nu\xi} v_{(n-1\, n)T_{n}}=0$. Indeed, if the tableau $(n-1\, n)T_{n}$ is not standard, then $v_{(n-1\, n)T_{n}}=0$ and the claim is true. If the tableau $(n-1\, n)T_{n}$ is standard, then the box $\xi/\nu=\xi/\mu$ of this tableau contains the integer~$n-1$, so that erasing it does not produce a standard tableau of shape $\nu$, and $i_{\nu\xi} v_{(n-1\, n)T_{n}}=0$. 

Therefore,
\[a (v_{T})= \cos \theta \, i_{\lambda\nu}i_{\nu\xi} v_{T_{n}}=\cos \theta \, i_{\lambda\nu}v_{T_{n-1}}=\cos \theta \, v_{T}.\]

Let us now treat the case where $\nu\neq \mu$. In this case, the boxes $\xi/\nu$ and $\xi/\mu$ are distinct, and each of them can be removed from $\nu$ to produce a new Young diagram. Therefore, these boxes are not neighbours in the same row or column, and the tableau $(n-1\, n)T_{n}$ is standard.

On the one hand, the tableau $T_{n}$ contains $n$ in the box $\xi/\mu$ and $n-1$ in the box $\mu/\lambda=\xi/\nu$, so that $i_{\nu\xi}v_{T_{n}}=0$. On the other hand, the tableau $(n-1\, n)T_{n}$ contains $n$ in the box $\mu/\lambda=\xi/\nu$, and erasing this box produces the tableau $S(T,\nu)$, so that $i_{\nu\xi}v_{T_{n}}=v_{S(T,\nu)}$. Therefore,
\[a(v_{T})= \sin \theta \, i_{\lambda\nu}i_{\nu\xi} v_{(n-1\, n)T_{n}}= \sin \theta \, i_{\lambda\nu}v_{S(T,\nu)}=\sin \theta \, v_{T},\]
and the proof is complete.
\end{proof}

\subsection{Completion of the proof of the main theorem}
The proof of Theorem \ref{thm:main} is now essentially finished. Let us state the final expression of the flat contribution of a configuration of highest weights, which holds without the assumption of non-negativity. 

Recall the definition of the angle $\theta_{v}^{\Lambda}$ by \eqref{eq:defthetavL}, as well as the definition of vertices of type~$1$ and of type~$2$ given by \eqref{eq:deftype}.

\begin{proposition} \label{prop:eqI6} The flat contribution of a balanced configuration of highest weights $\Lambda$ is 
\begin{equation}\label{eq:I6}
\mathscr F(\Lambda)=\prod_{F\in \sF} (d_{\lambda_{F}})^{\e_{F}}
\prod_{v\in \sV} (d_{\lambda_{\cn(v)}}d_{\lambda_{\cs(v)}})^{-\frac{1}{2}}
\prod_{v \in \sV^{\Lambda}_{1}} \cos \theta^{\Lambda}_{v}\prod_{v \in \sV^{\Lambda}_{2}} \sin \theta^{\Lambda}_{v}.
\end{equation}
\end{proposition}

\begin{proof} If $\Lambda$ is non-negative, then $\mathscr F(\Lambda)$ is given by Proposition \ref{prop:expI5}. For every vertex $v$, the scalar~$a(v)$, defined by \eqref{eq:defa(v)}, is computed by Proposition \ref{prop:determinea}, in terms of an angle $\theta$ that is exactly the angle $\theta_{v}^{\Lambda}$. The distinction between vertices of type $1$ and of type $2$ is just the distinction which decides, in Proposition \ref{prop:determinea}, if $a(v)$ is the cosine or the sine of $\theta$.

If we do not assume anymore that $\Lambda$ is non-negative, we can choose an integer $q$ such that the configuration $\Lambda+(q,\ldots,q)$ is non-negative, with the notation of Lemma \ref{lem:translate}. This lemma tells us that the flat contribution of $\Lambda$ equals that of $\Lambda+(q,\ldots,q)$, that is, the right-hand side of \eqref{eq:I6} applied to $\Lambda+(q,\ldots,q)$.

It remains to observe that the right-hand side of \eqref{eq:I6}  is invariant under the addition of a constant vector $(q,\ldots,q)$ to each highest weight of the configuration. Indeed, firstly, it follows from~\eqref{eq:Weyldim} that for every highest weight $\lambda$ and every $q\in \Z$, we have $d_{\lambda+(q,\ldots,q)}=d_{\lambda}$. 
Secondly, adding $(q,\ldots,q)$ to each highest weight adds $q$ to $\cont(\lambda_{\ce_{v}}/ \lambda_{\cn_{v}})$ and to $\cont(\lambda_{\cn_{v}}/ \lambda_{\cw_{v}})$ for every vertex~$v$, and therefore does not change the angle $\theta^{\Lambda}_{v}$.
\end{proof}

Theorem \ref{thm:main} now follows from the combination of
\begin{itemize}
\item Proposition \ref{prop:expanded}, which expresses the left-hand sides of \eqref{eq:main} and \eqref{eq:main2} as sums over configurations of highest weights, the contribution of each configuration being an exponential prefactor times its flat contribution;
\item Proposition \ref{prop:balance}, which states that the flat contribution of a configuration of highest weights that is not balanced vanishes, and therefore restricts the sum to the set $\B$ of balanced configurations;
\item Proposition \ref{prop:eqI6} which computes the flat contribution of an arbitrary balanced configuration of highest weights. 
\end{itemize}

\section{Afterthoughts}\label{sec:after}

\subsection{Rationality of flat contributions} By its very definition \eqref{eq:defthetavL}, the angle $\theta^{\Lambda}_{v}$ has a rational cosine. Its sine, in contrast, is either zero or irrational, and it can be zero only if the vertex $v$ is of type $1$. Therefore, the expression \eqref{eq:I6} of the flat contribution of any configuration of highest weights with at least one vertex of type $2$ is a product that involves irrational numbers. We will now prove that it is nevertheless a rational number.  

\begin{proposition} \label{prop:rational} The flat contribution of any balanced configuration of highest weights is a rational number. 
\end{proposition}

The argument of the following proposition was suggested to me by Antoine Dahlqvist, to whom I express my friendly gratitude.

\begin{proof} Consider a balanced configuration $\Lambda$. For each edge $e\in \sE$, the highest weights $\lambda=\lambda_{\lf_{e}}$ and  $\mu=\lambda_{\rf_{e}}$ differ by one component only, of index $i$, for which $\mu_{i}=\lambda_{i}+1$. Let us give to each edge $e$ a {\em label} with the couple of integers $(i,\mu_{i})$. 

At any vertex, the two incoming edges have distinct labels, and each incoming edge has the same label as one of the two outgoing edges (this is not in contradiction with the fact that an edge can be at the same time incoming and outgoing at a vertex). The possible situations are shown in Figure \ref{fig:couleurs} below.

Let us focus on the set of all edges with a particular label. It follows from what we just said that at each vertex, there is either one incoming edge and one outgoing edge with this particular label, or none at all. 

Therefore, the set of all edges with a given label forms a collection of pairwise disjoint oriented simple loops in $\sG$. Let us call these loops, which depend on $\Lambda$, the {\em level lines} of the configuration. Level lines are labelled by a couple of integers $(i,a)$ with $i\in \{1,\ldots,N\}$ and $a\in \Z$, they are simple loops, and two distinct level lines with the same label are disjoint. 

Looking more closely at the possible situations at a vertex (see Figure \ref{fig:couleurs}), one sees that vertices of type $2$ are exactly the points where two level lines cross transversally. Moreover, the labels $(i,a)$ and $(j,b)$ of two level lines crossing transversally must satisfy $i\neq j$.

\begin{figure}[h!]
\begin{center}
\includegraphics{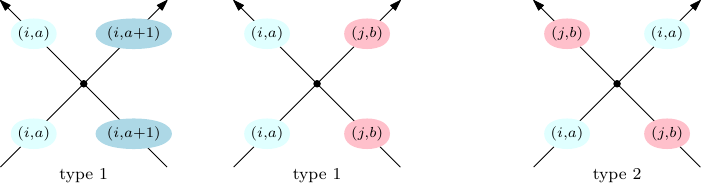}
\caption{\label{fig:couleurs} \small The three possible situations regarding the labels of the four strands adjacent to a vertex. The integers $i$ and $j$ are distinct.
}
\end{center}
\end{figure}

A vertex where two level lines of respective labels $(i,a)$ and $(j,b)$ cross transversally contributes to \eqref{eq:I6} with a factor 
\begin{equation}\label{eq:racine}
\frac{\sqrt{(a-b+j-i)^{2}-1}}{|a-b+j-i|},
\end{equation}
that we write only in order to show that it is completely determined by the labels of the level lines crossing at this vertex. 

The fact that \eqref{eq:I6} is rational will follow from the following assertion: given any two distinct labels $(i,a)$ and $(j,b)$, any level line labelled $(i,a)$ has an even number of transversal crossings with the union of all level lines labelled $(j,b)$.

This assertion implies that the total number of transversal crossings of a level line labelled~$(i,a)$ with a level line labelled $(j,b)$ is even, so that the factor \eqref{eq:racine} appears with an even power in~\eqref{eq:I6}, which is therefore rational. 

Let us now prove the assertion. Let us consider two distinct labels $(i,a)$ and $(j,b)$, and let us consider a level line labelled $(i,a)$. If $j=i$, then this level line never crosses transversally a level line labelled $(j,b)$. We may therefore assume that $j\neq i$. 

Let us travel once along the chosen level line labelled $(i,a)$, while monitoring the value of the $j$-th component of the highest weight associated to the face located immediately to the right of the level line. At each vertex that the level line crosses, this value changes by $-1$, $0$ or $1$. 

Now, on the one hand, the vertices where this value changes from $b-1$ to $b$ or from $b$ to $b-1$ are exactly the transversal crossings with a level line labelled $(j,b)$. On the other hand, this value comes back to its initial value once the travel is complete, and must therefore go as many times from $b-1$ to $b$ as it goes from $b$ to $b-1$. This implies that the chosen level line crosses transversally an even number of times the union of all level lines labelled $(j,b)$.
\end{proof}

The fact that a level line labelled $(i,a)$ has an even number of transversal crossings with the union of all level lines labelled $(j,b)$ can also be explained from a topological point of view. Indeed, the sum of all edges labelled $(j,b)$ determines a $1$-chain on $\Sigma$, and this chain is a cycle. The point is that this cycle is homologically trivial, because it is the boundary of the sum of the faces $F$ for which $(\lambda_{F})_{j}\leq b-1$, each face being seen as a $2$-chain. Therefore, any other cycle, for example the cycle determined by a level line labelled $(i,a)$, has an intersection number of $0$ with it. Since this intersection number can be counted as a sum over transverse crossings of a sign, there must be an even number of such crossings.

Making these statements precise would involve the relative homology of the pairs $(\sG,\sV)$ and $(\Sigma,\sG)$, but as we already gave a complete argument, it does not seem necessary to go into this. Let us nevertheless state the following consequence of the triviality of the cycle formed by adding all edges with a given label.

\begin{proposition} \label{prop:homcond} If the homology class of the loop configuration $\ul\ell$
in $H_{1}(\Sigma)$ does not belong to the subgroup generated by the classes of the constrained boundary components, then the Wilson loop expectation \eqref{eq:WLE} vanishes.
\end{proposition}

When $\Sigma$ is closed, the necessary condition for the Wilson loop expectation not to be zero is simply that the total homology class of the loop configuration vanishes. 

\begin{proof} Let us assume that the Wilson loop expectation associated to the loop configuration $\ul\ell$ is not zero. Then on the graph determined by $\ul\ell$, there exists a balanced configuration of highest weights vanishing on the free boundary.

Along any curve drawn on $\Sigma$ which is transverse to $\ul\ell$ at each intersection point, the sum of the components of the highest weight of the face being currently traversed by the curve increases by~$1$ at each crossing of a strand of $\ul\ell$ coming from the right, and it decreases by~$1$ at each crossing of a strand of $\ul\ell$ coming from the left. Therefore, the loop configuration $\ul\ell$ has a total algebraic intersection number of $0$ with any closed curve on $\Sigma$, and with any curve that starts and finishes at a free boundary component. 

According to  \cite[Theorem 3.43]{Hatcher}, this implies that the class of $\ul{\ell}$ in the relative homology group $H_{1}(\Sigma,C_{1}\cup \ldots \cup C_{\k})$ is zero. This in turn implies that the homology class of $\ul{\ell}$ in $H_{1}(\Sigma)$ belongs to the subgroup generated by the classes of $C_{1},\ldots,C_{\k}$.
\end{proof}

\subsection{Cohomological aspects of the sign problem} In this section, we will explain how the problem of the determination of the sign of $a(v)$ when $v$ is a vertex of type $2$ can be reformulated in cohomological terms.

For this, let us construct a $2$-dimensional CW-complex $\Y$ that could be called the Young complex. The set of $0$-cells of $\Y$ is $\Part$, the set of all partitions. For each couple of partitions $\lambda$ and $\mu$ such that $\lambda\leads \mu$, let us attach a $1$-cell joining $\lambda$ to $\mu$. So far, $\Y$ is the well-known Young graph. Now, for all quadruple $(\lambda,\mu,\nu,\xi)$ of partitions with $\mu\neq \nu$ such that $\lambda\leads\mu\leads\xi$ and $\lambda\leads\nu\leads\xi$ let us attach a square with edges $\lambda\leads \mu$, $\mu\leads \xi$, $\lambda\leads \xi$ and $\nu\leads \xi$  to $\Y$. 

\index{Y@$\Y$, Young complex}

A non-negative balanced configuration of highest weights can be seen as a map from the dual graph of $\sG$, in a natural sense, to the complex $\Y$, but we will not need this point of view. 

\begin{figure}[h!]
\begin{center}
\includegraphics{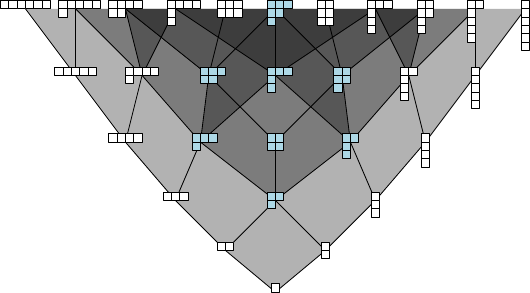}
\caption{\label{fig:youngcomplex} \small A small part of the complex $\Y$. The eight blue partitions are the vertices of a cube of which all faces belong to the complex.}
\end{center}
\end{figure}

Above each edge $\lambda\leads \mu$, let us consider the set $I_{\mu\lambda}$ of isometric intertwiners $V^{\lambda}\to V^{\mu}$. This set has two elements which differ by a sign. Thus, a choice $i=(i_{\mu\lambda})_{\lambda\uparrow\mu}$ of intertwiners being made, any other choice can be deduced from the first by a change of sign of some of the intertwiners. The choice of a sign on each edge can be represented by a $1$-cochain $\epsilon \in C^{1}(\Y,\Z/2\Z)$, and any choice of intertwiners is of the form $\epsilon \cdot i$. 

Let $i=(i_{\mu\lambda})_{\lambda\uparrow\mu}$ be a choice of isometric intertwiners. Let $(\lambda,\mu,\nu,\xi)$ be a $2$-cell of $\Y$. There is no orientation of the $2$-cell, so that $\mu$ and $\nu$ play perfectly symmetric roles. Therefore, we cannot define an angle in the same way as we did for instance in \eqref{eq:defthetavL}, but we can define $\theta\in [0,\frac{\pi}{2})$ by the relation
\[\sin^{2} \theta_{\lambda,\mu,\nu,\xi} = 1-\frac{1}{\big(\cont(\xi/\mu)-\cont(\mu/\lambda)\big)^{2}}= 1-\frac{1}{\big(\cont(\xi/\nu)-\cont(\nu/\lambda)\big)^{2}}.\]
Now, denoting by $n$ the weight of the partition $\xi$, that is, the sum of its parts, there exists $\sigma_{\lambda,\mu,\nu,\xi}\in \{\pm 1\}$, which depends on $i$, such that
\[i_{\lambda\nu}i_{\nu\xi}\, (n-1\, n)\,  i_{\xi\mu}i_{\mu\lambda}=\sigma_{\lambda,\mu,\nu,\xi}\, \sin \theta_{\lambda,\mu,\nu,\xi}.\]
The object $\sigma$ thus defined belongs to $C^{2}(\Y,\Z/2\Z)$. Let us denote it by $\sigma(i)$ to indicate its dependence on $i$. It is readily checked that for all $\epsilon \in C^{1}(\Y,\Z/2\Z)$,
\begin{equation}\label{eq:d1}
\sigma(\epsilon \cdot i)=\sigma(i)+d\epsilon.
\end{equation}

Since our complex is $2$-dimensional, the $2$-chain $\sigma$ is closed, and the last equation shows that its class in $H^{2}(\Y,\Z/2\Z)$ does not depend on the choice of $i$. It is canonically defined in the present situation, and computing the sign of $a(v)$ amounts to computing this cohomology class.

Indeed, Propositions \ref{prop:GZbasis} and \ref{prop:determinea} show that there exists a choice of $i$ for which $\sigma(i)$ vanishes identically. In other words, they solve the problem of the determination of the sign of the scalar~$a(v)$ by showing that the canonical class that we just defined is zero --- it hardly gets any more canonical. 

Let us say a few more words about the way in which these propositions solve the problem. Proposition \ref{prop:GZbasis} singles out two bases of each module $V^{\lambda}$.
Just as $C^{1}(\Y,\Z/2\Z)$ acts freely and transitively on the set of possible choices of isometric intertwiners, the space $C^{0}(\Y,\Z/2\Z)$ acts freely and transitively on the set of possible choices of one of the two bases in each module~$V^{\lambda}$. A choice of bases $b=(b_{\lambda})_{\lambda}$ being made, any other choice is of the form $\alpha \cdot b$ for some $\alpha\in C^{0}(\Y,\Z/2\Z)$.

Now, each choice $b$ of bases determines a choice $i(b)$ of isometric intertwiners, in the way that we explained between Proposition \ref{prop:GZbasis} and Proposition \ref{prop:determinea}. And we have, for every $\alpha\in C^{0}(\Y,\Z/2\Z)$, the relation
\begin{equation}\label{eq:d2}
i(\alpha \cdot b)=(d\alpha)\cdot i(b),
\end{equation}
from which it follows, together with \eqref{eq:d1}, that $\sigma(i(\alpha\cdot b))=\sigma(i)$.

This last equality explains why the choice of bases did not play any role in the last step of our determination of the scalar $a(v)$.

\subsection{The Makeenko--Migdal equations} In this section, we will show, without getting bogged down in unnecessary technicalities, how our result allows us to recover the Makeenko--Migdal equations, which express the variation of the Wilson loop expectations as one varies the areas of the faces of the graph formed by the loop configuration without changing the total area of the surface. 

To be more precise, the Makeenko--Migdal equations describe the variation of a Wilson loop expectation when, around a vertex $v$, one increases the areas of the north and south faces, and decreases the areas of the east and west faces. The result is another Wilson loop expectation, of a loop configuration that does not exactly fit in the setting that we chose, because in this new configuration, two loops have a corner at $v$, and a contact point that is not a crossing. The equation can nevertheless be drawn suggestively as follows:
\begin{equation}\label{eq:MM}
\bigg(\frac{d}{d|\cn|}-\frac{d}{d|\cw|}+\frac{d}{d|\cs|}-\frac{d}{d|\ce|}\bigg)\, {\bf E}\Bigg[ \Tr\bigg(\ \raisebox{-5mm}{\includegraphics{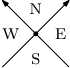}}\ \bigg) \ldots\Bigg]=\frac{1}{N}\, {\bf E}\Bigg[ \Tr\bigg(\ \raisebox{-5mm}{\includegraphics{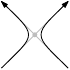}}\ \bigg) \ldots\Bigg],
\end{equation}
where we accept the loop configuration drawn on the right as a legitimate one, and call it the loop configuration {\em desingularized at $v$}. Note that this loop configuration has one more loop, or one less, than $\ul\ell$, depending on whether the two strands crossing at $v$ belong to the same loop or to two different loops of $\ul\ell$. Note also that, since we compute expectations of product of traces, and not of normalised traces, the factor $\frac{1}{N}$ is present in all cases.

These equations, discovered by Makeenko and Migdal \cite{MakeenkoMigdal}, were established mathematically on the plane by the author \cite{LevyMF}, then by Driver, Hall and Kemp \cite{DHK}, and on a surface by these authors and Gabriel \cite{DGHK}. They were also investigated in \cite{Hall}, \cite{DriverMM} and \cite{DahlqvistNorris}.

\begin{proposition}\label{prop:MM}
The Wilson loop expectations \eqref{eq:WLE} satisfy the Makeenko--Migdal equations~\eqref{eq:MM}.
\end{proposition}

\begin{proof} Let us consider a loop configuration $\ul\ell$ and compute the left-hand side of \eqref{eq:MM} at a vertex~$v$ of the graph induced by $\ul\ell$.

A first observation is that the expressions \eqref{eq:main} and \eqref{eq:main2} make it very easy to differentiate a Wilson loop expectation with respect to the area of certain faces, as long as the total area of $\Sigma$, and therefore the normalisation constant $Z(\Sigma)$ or $Z(\Sigma,\ul{w})$, stays constant.

Differentiating with respect to the area of a face $F$ multiplies the contribution of each configuration of highest weights by a factor $-\frac{c_{\lambda_{F}}}{2N}$, where $c$ denotes the quadratic Casimir number defined by \eqref{eq:casimir}. Therefore, the combination of derivatives that appears in the Makeenko--Migdal equations applied at the vertex $v$ gives rise to a factor
\[\frac{1}{2N}\big((c_{\ce}-c_{\cn})-(c_{\cs}-c_{\cw})\big),\]
where $\cn$, $\cw$, $\cs$ and $\ce$ are as usual the faces surrounding the vertex $v$. 

A second observation is the following relation, easily derived from \eqref{eq:casimir}: if $\lambda$ and $\mu$ are highest weights such that $\lambda\leads \mu$, and if $i\in \{1,\ldots,N\}$ is the index such that $\mu_{i}=\lambda_{i}+1$, then
\[c_{\mu}-c_{\lambda}=2(\mu_{i}-i)+N=2\cont(\mu/\lambda)+N,\]
where the last $\cont$ is the content, defined by \eqref{eq:defcont}.

Therefore, the left-hand side of \eqref{eq:MM} is equal to the Wilson loop expectation of the original loop configuration $\ul\ell$, in which the contribution of each configuration of highest weights $\Lambda$ is multiplied by
\[\tfrac{1}{N} \big(\cont(\lambda_{\ce}/\lambda_{\cn})-\cont(\lambda_{\cs}/\lambda_{\cw})\big).\]
If the vertex $v$ is of type $2$ for the configuration $\Lambda$, then the skew-diagrams $\lambda_{\ce}/\lambda_{\cn}$ and $\lambda_{\cs}/\lambda_{\cw}$ are the same one box, and this factor is $0$. If $v$ is of type $1$ however, then this factor is equal to
\[\tfrac{1}{N} \big(\cont(\lambda_{\ce}/\lambda_{\cn})-\cont(\lambda_{\cn}/\lambda_{\cw})\big)=\frac{1}{N\cos \theta^{\Lambda}_{v}}.\]

This proves that the left-hand side of \eqref{eq:MM} is equal to $\frac{1}{N}$ times the right-hand side of \eqref{eq:main} or~\eqref{eq:main2}, where
\begin{enumerate}[\indent\sbullet]
\item the sum is restricted to the configurations of highest weights for which $v$ is of type $1$,
\item the cosine of the angle attached to $v$ is removed. 
\end{enumerate}

Let us compare this with the right-hand side of \eqref{eq:MM}. A configuration of highest weights on the graph associated to $\ul\ell$ for which $v$ is of type $1$ is the same thing as a configuration of highest weights on the graph associated to the loop configuration $\ul\ell$ desingularized at $v$, in which the north and south faces have merged. Apart from factors that are literally identical on both sides, such a configuration comes 
\begin{enumerate}[\indent\sbullet]
\item on the left-hand side of \eqref{eq:MM} with the coefficient 
\[e^{-\frac{1}{2N} (|\cn|c_{\lambda_{\cn}}+|\cs|c_{\lambda_{\cs}})}(d_{\lambda_{\cn}})^{\e_{\cn}}(d_{\lambda_{\cs}})^{\e_{\cs}} (d_{\lambda_{\cn}}d_{\lambda_{\cs}})^{-\frac{1}{2}},\]
\item on the right-hand side of \eqref{eq:MM} with the coefficient 
\[e^{-\frac{1}{2N} |F|c_{\lambda_{F}}}(d_{\lambda_{F}})^{\e_{F}},\]
where $F$ is the face resulting from the fusion of the north and south faces. 
\end{enumerate}
Since $\lambda_{N}=\lambda_{S}=\lambda_{F}$, and the area of $F$ is the sum of the areas of $\cn$ and $\cs$, the exponential terms are equal. 

The only point left to verify is that $\e_{F}=\e_{N}+\e_{S}-1$. This can be proved by saying that the compact surface $\bar F$ is obtained by merging a boundary component of $\cn$ with one of $\cs$, as illustrated by Figure \ref{fig:3holed}.
\begin{figure}[h!]
\begin{center}
\scalebox{1}{\includegraphics{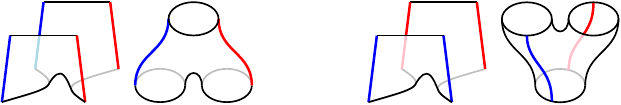}}
\caption{\label{fig:3holed} \small This figure illustrates the merging of a boundary component of the face~$\cn$ with a boundary component of the face $\cs$, in the case where they are distinct (on the left), and in the case where they are equal (on the right).}
\end{center}
\end{figure}

To be more precise, if $\cn$ and $\cs$ are distinct faces of $\sG$, then $\bar F$ is obtained, up to homeomorphism, by taking one boundary component of $\bar \cn$ and one of $\bar \cs$, and gluing them to two boundary components of a $3$-hold sphere. If $\cn$ and $\cs$ are the same face of $\sG$, then $\bar F$ can be obtained by gluing one boundary component of $\bar \cn$ on one boundary components of a $3$-hold sphere. In both cases, the result is a surface with Euler characteristic $\e_{\cn}+\e_{\cs}-1$.
\end{proof}



\newpage

\newcommand{\arXiv}[1]{\href{http://arxiv.org/abs/#1}{arXiv:#1}}
\bibliographystyle{alpha}
\bibliography{wlebib}
\bigskip




{\footnotesize
\printindex
}

\centerline{\rule{\textwidth}{0.3pt}}
\bigskip

\end{document}